\documentclass[reqno,12pt]{amsart}
\usepackage{amsthm,amsfonts,amssymb,euscript,
mathrsfs,graphics,color,amsmath,latexsym,marginnote}
\usepackage{dsfont}   
   
\theoremstyle{plain}

\usepackage{hyperref}
\usepackage{times}
\usepackage{mathtools}

\numberwithin{equation}{section}

\newcommand{\la}{\langle}
\newcommand{\ra}{\rangle}

\newcommand{\Lip}{\mathrm{Lip}}

\newtheorem{theorem}{Theorem}[section]
\newtheorem{proposition}[theorem]{Proposition}
\newtheorem{lemma}[theorem]{Lemma}
\newtheorem{corollary}[theorem]{Corollary}

\newtheorem{remark}[theorem]{Remark}
\newtheorem{remarks}[theorem]{Remark}

\newtheorem{definition}[theorem]{Definition}
\newcommand{\be}{\begin{equation}}
\newcommand{\ee}{\end{equation}}
\newcommand{\om}{\omega}

\newcommand{\e}{\varepsilon}

\newcommand{\R}{\mathbb R}
\newcommand{\C}{\mathbb C}

\newcommand{\Z}{\mathbb Z}

\newcommand{\N}{\mathbb N}
\newcommand{\T}{\mathbb T}

\renewcommand{\a }{\alpha }

\newcommand{\ii }{{\rm i} }

\newcommand{\g }{\gamma}

\newcommand{\vphi}{\varphi }

\renewcommand{\t }{\tau }

\newcommand{\mN}{\mathcal{N}}

\newcommand{\ph}{\varphi}

\newcommand{\mL}{\mathcal{L}}

\newcommand{\lm}{\lambda}

\newcommand{\mA}{\mathcal{A}}

\newcommand{\mR}{\mathcal{R}}
\newcommand{\mF}{\mathcal{F}}

\newcommand{\mS}{\mathcal{S}}
\newcommand{\mQ}{\mathcal{Q}}

\newcommand{\di}{{\mathrm d}}

\def\div{{\,\mbox{div}\,}}

\def\hat{\widehat}

\def\bar{\overline}

\def\cal{\mathcal}

\newcommand{\mD}{\mathcal{D}}
\newcommand{\odd}{\text{odd}}
\newcommand{\even}{\text{even}}

\def\ba{\begin{aligned}}
\def\ea{\end{aligned}}
\def\beginm{\begin{multline}}
\def\endm{\end{multline}}

\newcommand{\mB}{\mathcal{B}}

\makeatletter

\def\l@subsection{\@tocline{2}{0pt}{2.5pc}{5pc}{}}

\begin{document}
\bibliographystyle{plain}

\title{{\bf A KAM approach to the inviscid limit for the 2D Navier-Stokes equations}}

\date{}

\author{Luca Franzoi, Riccardo Montalto}

\maketitle

\noindent
{\bf Abstract.}
In this paper we investigate the inviscid limit $\nu \to 0$ for time-quasi-periodic solutions 
of the incompressible Navier-Stokes equations
on the two-dimensional torus $\T^2$, 
with a small time-quasi-periodic 
external force. 
 More precisely, we construct solutions of the forced Navier Stokes equation,  bifurcating from a given time quasi-periodic solution of the incompressible Euler equations and admitting vanishing viscosity limit to the latter, uniformly for all times and independently of the size of the external perturbation. Our proof is based on the construction of an approximate solution, up to an error of order $O(\nu^2)$ and on a fixed point argument starting with this new approximate solution. A fundamental step is to prove the invertibility of the linearized Navier Stokes operator at a quasi-periodic solution of the Euler equation, with smallness conditions and estimates which are uniform with respect to the viscosity parameter. To the best of our knowledge, this is the first positive result for the inviscid limit problem that is global and uniform in time and it is the first KAM result in the framework of the singular limit problems. 

\smallskip 





\tableofcontents

\section{Introduction}\label{introduction}

In this paper we consider the two dimensional Navier Stokes equations for an incompressible fluid on the two-dimensional torus $\T^2$, $\T := \R / 2 \pi \Z$, with a small time quasi-periodic forcing term
\begin{equation}\label{Eulero1}
\begin{cases}
\partial_t U + U \cdot \nabla U - \nu \Delta U + \nabla p = \e f(\omega t, x) \\
\div U = 0
\end{cases} 
\end{equation}
where $\e \in (0, 1)$ is a small parameter, $\omega \in \R^d$ is a Diophantine $d$-dimensional vector, $\nu > 0$ is the viscosity parameter,
the external force $f$ belongs to ${\mathcal C}^q(\T^d \times \T^2, \R^2)$ 
for some integer $q > 0$ large enough, 
$U = (U_1, U_2) : \R \times \T^2 \to \R^2$ is the velocity field, 
and $p : \R \times \T^2 \to \R$ is the pressure. The main purpose of this paper is to investigate the inviscid limit of the Navier Stokes equation from the perspective of the KAM (Kolmogorov-Arnold-Moser) theory for PDEs, which in a broad sense is the theory of the existence and the stability of periodic, quasi-periodic and almost periodic solutions for Partial Differential Equations. More precisely, we construct quasi-periodic solutions of \eqref{Eulero1} converging to the quasi-periodic solutions of the Euler equation, constructed in Baldi \& Montalto \cite{BaldiMontalto}, as the viscosity parameter $\nu \to 0$ and we provide a convergence rate $O(\nu)$ which is uniform for all times. As a consequence of our result, we obtain families of initial data for which the corresponding global quasi-periodic solutions  of the Navier Stokes equations converge to the ones of the Euler equation with a rate of convergence $O(\nu)$, uniformly in time. 
The main difficulty is that this is a {\it singular perturbation problem}, namely there is a small parameter in front of the highest order derivative. To the best of our knowledge, this is both the first result in which one exhibits solutions of the Navier Stokes equations converging globally and uniformly in time to the ones of the Euler equation in the vanishing viscosity limit $\nu \to 0$ and  the first KAM  result in the context of singular limit problems for PDEs. 

\noindent
 The zero-viscosity limit of the incompressible Navier-Stokes equations in bounded domains  is one of the most challenging  problems in Fluid Mechanics. The first results for smooth initial data ($H^s$ with $s \gg 0$ large enough) have been proved by Kato \cite{Kato1}, \cite{Kato2}, Swann \cite{Swann}, Constantin \cite{Constantin0} and Masmoudi \cite{Masmoudi}  in the Euclidean domain $\R^n$ or in the periodic box $\T^n$, $n=2,3$. For instance, it is proved  in \cite{Masmoudi} that, if the initial velocity field $u_0 \in H^s(\T^n)$, $s > n/2 + 1$, then the corresponding solutions $u_\nu(t, x)$ of Navier Stokes and $u(t, x)$ of Euler, defined on $[0, T] \times \T^n$, satisfiy
\begin{equation*}\label{stima Nader}
\begin{aligned}
& \| u_\nu - u \|_{L^\infty([0, T], H^s)} \to 0 \quad \text{as} \quad \nu \to 0 \\
& \text{and for} \quad s' < s \quad  \quad \| u_\nu(t) - u(t) \|_{H^{s'}} \lesssim (\nu t)^{\frac{s - s'}{2}}, \quad \forall t \in [0, T]\,. 
\end{aligned}
\end{equation*}
It is immediate to notice that the latter estimate holds only on finite time intervals and it is not uniform in time, with the estimate of the difference $u_\nu(t) - u(t)$ eventually diverging as $t \to + \infty$. For $n=2$, this kind of result has been proved in low regularity by Chemin \cite{Chemin} and Seis \cite{Seis}, with rates of convergence in $L^2$. We also mention same results for non-smooth vorticity. In particular, the inviscid limit of Navier Stokes equation has been addressed in the case of vortex patches in Constantin \& Wu \cite{Constantin1}, \cite{Constantin2}, Abidi \& Danchin \cite{Abidi-Danchin} and Masmoudi \cite{Masmoudi}, with {\it low} Besov type regularity in space. In this results one typically gets a bound only in $L^2$ of the form
$$
\| u_\nu(t) - u(t) \|_{L^2} \lesssim (\nu t)^\alpha \quad \text{for some} \quad \alpha> 0\,.
$$
In the case of non-smooth vorticity, the inviscid limit has been investigated by using a Lagrangian stochastic approach in Constantin, Drivas \& Elgindi \cite{Constantin3}, with initial vorticity $\omega_0 \in L^\infty(\T^2)$, and in Ciampa, Crippa \& Spirito \cite{Gennaro}, where the initial vorticity $\omega_0 \in L^p(\T^2)$, $p \in (1, + \infty)$. When the domain has an actual boundary, the zero-viscosity limit is closely related to the validity of the Prandtl equation for the formation of boundary layers. For completeness of the exposition, we mention the  work of Sammartino \& Caflisch \cite{SammCaflisch1}-\cite{SammCaflisch2} and recent results by Maekawa \cite{Maekawa}, Constantin, Kukavica \& Vicol \cite{ConstKukVicol} and G\'erard-Varet, Lacave, Nguyen \& Rousset \cite{GLNR}, with references therein.  The inviscid limit has been also investigated in other physical model for complex fluids, see for instance \cite{CaiLeiLinMasm} for the 2D incompressible viscoelasticity system.

\noindent
Our approach is different and it is based on KAM (Kolmogorov-Arnold-Moser) and Normal Form methods for Partial Differential Equations. This fields started from the Ninetiees, with the pioneering papers of Bourgain \cite{B}, Craig \& Wayne \cite{CW}, Kuksin \cite{K87}, Wayne \cite{Wayne}. We refer to the recent review article \cite{Berti-BUMI-2016} for a complete list of references on this topic. In the last years, new techniques have been developed in order to study periodic and quasi-periodic solutions for PDEs arising from fluid dynamics. For the two dimensional water waves equations, we mention Iooss, Plotnikov \& Toland \cite{IPT} for periodic standing waves,  \cite{Berti-Montalto}, \cite{BBHM} for quasi-periodic standing waves and \cite{BFM1}, \cite{BFM2}, \cite{FeolaGiuliani} for quasi-periodic traveling wave solutions. 

\noindent
We also recall that the challenging problem of constructing quasi-periodic solutions for the three dimensional water waves equations is still open. Partial results have been obtained by Iooss \& Plotnikov, who proved existence of symmetric and asymmetric {\it diamond waves} (bi-periodic waves stationary in a moving frame) in \cite{iooss-plotnikov-1}, \cite{iooss-plotnikov-2}. Very recently, KAM techniques have been successfully applied also for the contour dynamics of vortex patches in active scalar equations.  The existence of time quasi-periodic solutions have been proved in Berti, Hassainia \& Masmoudi \cite{BertiHassMasm} for vortex patches of the Euler equations close to Kirchhoff ellipses, in Hmidi \& Roulley \cite{HmidiRoulley} for the surface quasi-geostrophic (SQG) equations and in Hassainia, Hmidi \& Masmoudi \cite{HassHmidiMasm} for generalized SQG equations. All the aforementioned results concern 2D Euler equations.
The quasi-periodic solutions for the 3D Euler equations with time quasi-periodic external force have been constructed in \cite{BaldiMontalto} and also extended in \cite{Montalto} for the Navier-Stokes equations in arbitrary dimension, \emph{without} dealing with the zero-viscosity limit. The result of the present paper closes also the gap between these two works.

\subsection{Main result.}

We now state precisely our main result. 
We look for time-quasi-periodic solutions of \eqref{Eulero1}, 
oscillating with time frequency $\omega$. 
In particular, we look for solutions which are small perturbations of constant velocity fields $\zeta \in \R^2$, namely solutions of the form 
$$
U (t, x ) = \zeta +  u (\vphi, x)|_{\vphi = \omega t} \quad \text{with} \quad \div u = 0\,,
$$
where the new unknown velocity field $u : \T^d \times \T^2 \to \R^2$ is a function of $(\vphi,x)\in \T^d\times \T^2$.
Plugging this ansatz into the equation, one is led to solve 
\begin{equation}\label{Eulero3}
\begin{cases}
\omega \cdot \partial_\vphi u + \zeta \cdot \nabla u + u \cdot \nabla u - \nu \Delta u + \nabla p 
= \e f(\vphi, x) \\
\div u  = 0\,,
\end{cases}
\end{equation}
with $p : \T^d \times \T^2 \to \R$  and $\omega \cdot \partial_\vphi := \sum_{i = 1}^d \omega_i \partial_{\vphi_i}$. 
According to \cite{BaldiMontalto}, we shall assume that the forcing term $f$ is odd with respect to $(\vphi, x)$, that is
\begin{equation}\label{ipotesi forzante}
	f(\vphi, x) = - f(- \vphi, - x), \quad \forall (\vphi, x) \in \T^d \times \T^2\,.
\end{equation}
It is convenient to work in the well known {\it vorticity formulation}. We define the scalar vorticity $v(\vphi, x)$ as 
\begin{equation*}\label{def vorticita}
v := \nabla \times u :=  \partial_{x_1} u_2 - \partial_{x_2} u_1 \,.
\end{equation*}
Hence, rescaling the variable $v \mapsto \sqrt{\e} v$ and the small parameter $\e \mapsto \e^2$, the equation \eqref{Eulero3} is equivalent to 
\begin{equation}\label{equazione vorticita media nulla}
	\begin{cases}
		\omega \cdot \partial_\vphi v + \zeta \cdot \nabla v  -  \nu \Delta v
		+ \e \Big(  u \cdot \nabla v -   F(\vphi, x) \Big) = 0, 
		\quad F := \nabla \times f, \\
		u = \nabla_\bot \big[(- \Delta)^{- 1} v \big]\,, \quad  \nabla_\bot := (\partial_{x_2}, - \partial_{x_1})\,,
	\end{cases}
\end{equation}
and $(-\Delta)^{-1}$ is the inverse of the Laplacian, namely the Fourier multiplier with symbol
$|\xi|^{-2}$ for $\xi \in \Z^2$, $\xi \neq 0$, 
Since $\int_{\T^2} v(\cdot, x)\, d x$ is a prime integral, we shall restrict to the space of zero average in $x$.  Then, the pressure is recovered, once the velocity field is known, by the formula $p = \Delta^{- 1} \big[ \e {\rm div} f(\omega t, x) - \div ( u \cdot \nabla u ) \big]$.

For any real $s \geq 0$, we consider the Sobolev spaces $	H^s = H^s(\T^{d + 2}) $ of real scalar 
and vector-valued functions of $(\ph,x)$, defined in \eqref{def sobolev},
and the Sobolev space of functions with zero space average, defined by
\begin{equation*}\label{def sobolv H0s}
	H^s_0 := \Big\{ u \in H^s : \int_{\T^{2}} u(\vphi, x)\, d x= 0 \Big\}\,.
\end{equation*}
Furthermore, we introduce the subspaces of $L^2$ of the even and odd functions in $(\vphi,x)$, respectively:
\begin{equation}\label{X even Y odd}
	\begin{aligned}
		&X := \Big\{ v \in L^2(\T^{d + 2}) : v(\vphi, x) =  v(- \vphi, - x) \Big\} \\
		& Y := \Big\{ v \in L^2(\T^{d + 2}) : v(\vphi, x) = - v(- \vphi, - x) \Big\}\,.
	\end{aligned}
\end{equation}

We first state the result concerning the existence of quasi-periodic solutions of the Euler equation $(\nu = 0)$ for most values of the parameters $(\omega,\zeta)$ in a fixed bounded open set $\Omega\subset \R^d\times \R^2$, proved in \cite{BaldiMontalto}. The statement is slightly modified for the purposes of this paper.


\begin{theorem}\label{main theorem 2}
	{\bf (Baldi-Montalto \cite{BaldiMontalto}).}
There exists $\bar S := \bar S(d) > 0$ such that, for any $S \geq \bar S(d)$, there exists $q := q(S) > 0$ such that, for every forcing term $f \in {\mathcal C}^q(\T^d \times \T^2, \R^2)$ satisfying \eqref{ipotesi forzante}, there exists $\e_0 := \e_0(f, S, d) \in (0, 1)$ and $C := C(f, S, d) > 0$ such that, for every $\e \in (0, \e_0)$, the following holds. There exists a ${\mathcal C}^1$ map
$$
\begin{aligned}
& \R^{d + 2} \to H^S_0(\T^{d + 2}) \cap Y , \quad \lambda = (\omega, \zeta) \mapsto v_e(\cdot; \lambda)\,, 
\end{aligned}
$$ and a Borel set $\Omega_\e \subset \Omega$ of asymptotically full Lebesgue measure, i.e. $\lim_{\e \to 0} |\Omega \setminus \Omega_\e| = 0$, such that, for any $\lambda= (\omega, \zeta) \in \Omega_\e$, the function $v_e(\cdot; \lambda)$ is a quasi-periodic solution of the Euler equation
$$
 \omega \cdot \partial_\vphi v_e + \zeta \cdot \nabla v_e 
+ \e \Big(  u_e \cdot \nabla v_e -   F \Big) = 0, \quad u_e = \nabla_\bot (- \Delta)^{- 1} v_e\,. 
$$ Moreover, there exists a  constant $\mathtt a := \mathtt a(d) \in (0, 1)$ such that $\sup_{\lambda \in \R^{d + 2}} \|  v(\cdot; \lambda) \|_S \leq C \e^{\mathtt a}$ and, for any $i = 1, \ldots, d + 2$, $\sup_{\lambda \in \R^{d + 2}} \| \partial_{\lambda_i} v(\cdot; \lambda) \|_S \leq C \e^{\mathtt a}$.  
\end{theorem}
We now are ready to state the main result of this paper. Rougly speaking, we will prove that for any value of the viscosity parameter $\nu > 0$ and for $\e \ll 1$ small enough, \emph{independent of the viscosity parameter}, the Navier-Stokes equation \eqref{equazione vorticita media nulla} admits a quasi-periodic solution $v_\nu (\vphi, x)$ for most of the parameters $\lambda = (\omega, \zeta)$ such that $\| v_\nu - v_e \|_S = O(\nu)$. This implies that $v_\nu(\omega t, x)$ converges strongly to $v_e(\omega t, x)$, uniformly in $(t, x) \in \R \times \T^2$, with a rate of convergence $\sup_{(t, x) \in \R \times \T^2}|v_\nu(\omega t, x) - v_e(\omega t, x)| \lesssim \nu$. To the best of our knowledge, this is the first case in which the inviscid limit is uniform in time. We now give the precise statement of our main theorem. 
\begin{theorem}\label{teorema limite singolare}
{\bf (Singular KAM for 2D Navier-Stokes in the inviscid limit).}
There exist $\bar s := \bar s(d)$ and $ \overline \mu := \overline \mu(d) > 0$ such that, for any $s \geq \bar s(d)$, there exists $q := q(s) > 0$ such that, for every forcing term $f \in {\mathcal C}^q(\T^d \times \T^2, \R^2)$ satisfying \eqref{ipotesi forzante}, there exists $\e_0 := \e_0(f, s, d) \in (0, 1)$ and $C := C(f, s, d) > 0$ such that, for every $\e \in (0, \e_0)$ and  for any value of the viscosity parameter $\nu > 0$, the following holds. Let $v_e(\cdot; \lambda) \in H^{s + \overline \mu}_0(\T^{d + 2}) \cap Y$, $\lambda \in \Omega_\e$ be the family of solutions of the Euler equation provided by Theorem \ref{main theorem 2}. Then, there exists a Borel set ${\mathcal O}_\e \subseteq \Omega_\e$, satisfying $\lim_{\e \to 0} |{\mathcal O}_\e| = |\Omega|$ such that, for any $\lambda = (\omega, \zeta) \in {\mathcal O}_\e$, there exists a unique quasi-periodic solution $v_\nu(\cdot; \lambda) \in H^s_0(\T^{d + 2})$, $\lambda \in {\mathcal O}_\e$, of the Navier Stokes equation
$$
\omega \cdot \partial_\vphi v_\nu + \zeta \cdot \nabla v_\nu 
  - \nu \Delta v_\nu + \e\Big(u_\nu \cdot \nabla v_\nu -  F(\vphi, x) \Big) =0\,, \quad 
u_\nu = \nabla_\bot [(- \Delta)^{- 1} v_\nu \big]\,,
$$
satisfying the estimate 
$$
\sup_{\lambda \in {\mathcal O}_\e}\| v_\nu(\cdot; \lambda) - v_e(\cdot; \lambda) \|_s \lesssim_s \nu \,. 
$$
As a consequence, for any value of the parameter $\lambda \in {\mathcal O}_\e$, the quasi-periodic solutions of the Navier Stokes equation $v_\nu$ converge to the ones of the Euler equation $v_e$ in $H^s_0(\T^{d + 2})$ in the limit $\nu\to 0$. 
\end{theorem}
From the latter theorem we shall deduce the following corollary which provides a family of quasi-periodic solutions of Navier-Stokes equation converging to solutions of the Euler equation with rate of convergence $O(\nu)$ and uniformly for all times. The result is a direct consequence of the Sobolev embeddings.
\begin{corollary}\label{corollario bla bla}
{\bf (Uniform rate of convergence for the inviscid limit).}
Assume the same hypotheses of Theorem \ref{teorema limite singolare} and let $s \geq s_0$ large enough, $v_e \in H^{s + \bar\mu}_0(\T^{d + 2})$, $v_\nu \in H^s_0(\T^{d + 2})$ and, for $\omega \in {\mathcal O}_\e$, let $v_{\nu}^\omega(t, x) := v_\nu(\omega t, x)$, $v_{ e}^\omega(t, x) := v_e(\omega t, x)$ be defined for any $(t, x) \in \R \times \T^2$. Then, for any $\alpha \in \N$, $\beta \in \N^2$ with $|\alpha| + |\beta| \leq s - (\lfloor\frac{d+2}{2}\rfloor +1)$, one has 
$$
 \|\partial_t^\alpha \partial_x^\beta (v_{ \nu}^\omega - v_{ e}^\omega) \|_{L^\infty(\R \times \T^2)} \lesssim_{\alpha, \beta} \nu\,. 
$$

\end{corollary}

 Let us make some remarks on the result.
\\[1mm]
1) {\it Vanishing viscosity solutions of the Cauchy problem.} The time quasi-periodic solutions in Theorem \ref{teorema limite singolare} are slight perturbations of constant velocity fields $\zeta \in \R^2$ with frequency vector $\omega\in\R^d$ induced by the perturbative forcing term $f(\omega t,x)$. Since they exist only for most values of the parameters $(\omega,\zeta)$, we obtain equivalently that the Cauchy problem associated with \eqref{equazione vorticita media nulla} (and so of \eqref{Eulero1}) admits a subset of small amplitude initial data of relatively large measure, with elements evolving \emph{for all time}, in a eventually larger but still bounded neighbourhood in the Sobolev topology, and whose flows exhibit a uniform vanishing viscosity limit to solutions of the Cauchy problem for the Euler equations with same initial data.
\\[1mm]
2) {\it The role of the forcing term.} It is worth to note that the time quasi-periodic external forcing term $F(\omega t,x)$ in \eqref{Eulero1} is independent of the viscosity parameter $\nu>0$. Its presence ensures the existence of the time quasi-periodic Euler solution $v_e$ in Theorem \ref{main theorem 2}, while the construction of the viscous correction $v_{\nu}-v_{e}$ does not rely explicitly on it: if one is able to exhibits time quasi-periodic solutions close to constant velocity fields for the free 2D Euler equation, namely \eqref{equazione vorticita media nulla} with $\nu=0$ and $F\equiv 0$, then the ones for the Navier-Stokes equation follow immediately by our scheme. To our knowledge, the only result of existence of time quasi-periodic flows for the free Euler equations on $\T^2$ is given by Crouseilles \& Faou \cite{Faou}, where the solutions are searched based on a prescribed stationary shear flow, locally constant around finitely many points, and propagate in time in the orthogonal direction to the shear flow. Due to the nature of their solutions, the non-resonant frequencies are prescribed as well and therefore there are no small divisors issues involved.

\subsection{Strategy and main ideas of the proof}
In order to prove Theorem \ref{teorema limite singolare}, we have to construct a solution of the Navier-Stokes equation \eqref{equazione vorticita media nulla} which is a correction of order $O(\nu)$ of the solution $v_e$ of the Euler equation (provided in Theorem \ref{main theorem 2} of \cite{BaldiMontalto}). Roughly speaking, the difficult point is the following. There are two smallness parameters which are $\e$, the size of the Euler solution, and $\nu$, the size of the viscosity. If one tries to construct small solutions of the Navier-Stokes equation by using a standard fixed point argument, one immediately notes that a smallness condition of the form $\e \nu^{- 1} \ll 1$ is needed and clearly this is not enough to pass to the inviscid limit as $\nu \to 0$. The key point is to have a smallness condition on $\e$ which is independent of $\nu$ in such a way that one can pass to the limit as the viscosity $\nu \to 0$. We can summarize the construction into three main steps:
\begin{enumerate}
	\item Analysis of the linearized Navier-Stokes equation at the Euler solution and estimates for the inverse operators;
	\item Construction of the first order approximation for the viscous solution up to errors of order $O(\nu^2)$;
	\item A fixed point argument around the approximated viscous solution leading to the desired full solution of the Navier-Stokes equation.
\end{enumerate}

\noindent
{\bf Inversion of the linearized operator at the Euler solution.}
The essential ingredient is to analyse the linearized Navier-Stokes operator at the Euler solution $u_e$, namely one has to linearize \eqref{equazione vorticita media nulla} at the Euler solution $u_e(\vphi, x)$.  This leads to study a linear operator of the form 
\begin{equation}\label{linearizzato intro}
\begin{aligned}
{\mathcal L}_\nu & : = {\mathcal L}_e - \nu \Delta, \quad \,, \\
{\mathcal L}_e & :=  \omega \cdot \partial_\vphi + \big( \zeta + \e a(\vphi, x)\big) \cdot \nabla + \e {\mathcal R}\,, \quad a(\vphi, x)  := u_e(\vphi, x)\,, \\
 & \text{where} \quad {\mathcal R} : h(\vphi, x) \mapsto \nabla_\bot (- \Delta)^{- 1} h(\vphi, x) \cdot \nabla u_e(\vphi, x)
\end{aligned}
\end{equation}
is a pseudo-differential operator of order $- 1$. Note that the linear operator ${\mathcal L}_e$ is obtained by linearizing the Euler equation at the solution  with vorticity $v_e(\vphi, x)$. If one tries to implement a naive approach by directly using Neumann series to invert the linear operator ${\mathcal L}_\nu$, 
one has to require that $\e \nu^{- 1} \ll 1$, which is not enough to pass to the limit as $\nu \to 0$. To overcome this issue, we first implement the normal form procedure developed in \cite{BLM}, \cite{BaldiMontalto} to reduce to a diagonal, constant coefficients operator the Euler operator $\mL_{e}$, generating an unbounded correction to the viscous term $-\nu\Delta$ of size $O(\e\nu)$. More precisely, for most values of the parameters $(\omega, \zeta)$ and for $\e \ll 1$ small enough and \emph{independent of $\nu$}, we construct a bounded, invertible transformation $\Phi : H^s_0 \to H^s_0$  close to the identity such that 
\begin{equation}\label{op forma normale intro}
{\mathcal L}_{\infty, \nu} := \Phi^{- 1} {\mathcal L}_\nu \Phi = {\mathcal D}_\infty - \nu \Delta + {\mathcal R}_{\infty, \nu}
\end{equation}
where ${\mathcal D}_\infty$ and ${\mathcal R}_{\infty, \nu}$ have the following properties. ${\mathcal D}_\infty$ is a diagonal operator of the form 
\begin{equation}\label{operatore diagonale intro}
\begin{aligned}
&{\mathcal D}_\infty   := {\rm diag}_{(\ell, j) \in \Z^d \times (\Z^2 \setminus \{ 0 \})}\, \mu_\infty(\ell, j)\,, \\
&\mu_\infty(\ell, j)  :=  \ii (\omega \cdot \ell + \zeta \cdot j + r_j^\infty)\,, \quad \text{with }\quad |r_j^\infty|  \lesssim \e |j|^{- 1} \quad \forall \, j \in \Z^2 \setminus \{ 0 \}\,.
\end{aligned} 
\end{equation}
The remainder term ${\mathcal R}_{\infty, \nu}$ is an unbounded operator of order two and it satisfies an estimate of the form 
\begin{equation}\label{stima resto forma normale intro}
\| (- \Delta)^{- 1} {\mathcal R}_{\infty, \nu}\|_{{\mathcal B}(H^s)} \lesssim_s \e \nu
\end{equation}
where we denote by ${\mathcal B}(H^s)$, the space of bounded linear operators on $H^s$. The estimate \eqref{stima resto forma normale intro} is the key ingredient to invert the operator ${\mathcal L}_{\infty, \nu}$ in \eqref{op forma normale intro} with a smallness condition on $\e$  uniform with respect to the viscosity parameter $\nu > 0$. It is also crucial to exploit the reversibility structure which is a consequence of the fact that the solutions $v_e(\vphi, x)$ of the Euler equation are odd with respect to $(\vphi, x)$. This ensures that, for any $\ell \in \Z^d$, $j \in \Z^2 \setminus \{ 0 \}$ , the eigenvalues $\mu_\infty(\ell, j)$ of the diagonal operator ${\mathcal D}_\infty$ in \eqref{operatore diagonale intro} are purely imaginary (namely, the corrections $r_j^\infty$ are real). An important consequence is that the diagonal operator ${\mathcal D}_\infty - \nu \Delta$ is invertible and {\it gains two space derivatives} with an estimate for its inverse of order $O(\nu^{- 1})$. Indeed, the eingenvalues of ${\mathcal D}_\infty - \nu \Delta$ are $\ii (\omega \cdot \ell + \zeta \cdot j + r_j^\infty) + \nu |j|^2$, with $ (\ell, j) \in \Z^d \times (\Z^2 \setminus \{ 0 \})$,
and, since $\omega \cdot \ell + \zeta \cdot j + r_j^\infty$ is real, one gets a lower bound 
\begin{equation}\label{bound chiave autovalori intro}
|\ii (\omega \cdot \ell + \zeta \cdot j + r_j^\infty) + \nu |j|^2| \geq \nu |j|^2,
\end{equation} implying that ${\mathcal D}_\infty - \nu \Delta$ is invertible with inverse which gain two space derivatives, namely $$\| ({\mathcal D}_\infty - \nu \Delta)^{- 1}  (- \Delta)\|_{{\mathcal B}(H^s)} \lesssim \nu^{- 1}\,.$$ Thus, on one hand, $({\mathcal D}_\infty - \nu \Delta)^{- 1}$ gains two space derivatives, compensating the loss of two space derivatives of the remainder ${\mathcal R}_{\infty, \nu}$. On the other hand, the norm of $({\mathcal D}_\infty - \nu \Delta)^{- 1}$ explodes as $\nu^{- 1}$ as $\nu \to 0$, but this is compensated by the fact that ${\mathcal R}_{\infty, \nu}$ is of order $O(\e \nu)$. Therefore, recalling \eqref{stima resto forma normale intro}, one gets a bound 
$$
\| ({\mathcal D}_\infty - \nu \Delta)^{- 1} {\mathcal R}_{\infty, \nu}\|_{{\mathcal B}(H^s)} \lesssim \nu^{- 1} (\e \nu) \lesssim \e\,.
$$
Hence, by Neumann series, for $\e \ll 1$ small enough and independent of $\nu$, the operator ${\mathcal L}_{\infty, \nu}$ is invertible and gains two space derivatives, with estimate $\|{\mathcal L}_{\infty, \nu}^{- 1} (- \Delta) \|_{{\mathcal B}(H^s)} \lesssim_s \nu^{- 1}$. By \eqref{op forma normale intro}, we deduce that ${\mathcal L}_\nu$ is invertible as well and satisfies, for $\e\ll 1$ and for any $\nu > 0$,
\begin{equation}\label{stima inverso eul intro}
\| {\mathcal L}_\nu^{- 1} (- \Delta) \|_{{\mathcal B}(H^s)} \lesssim_s \nu^{- 1}\,.
\end{equation}

\noindent
{\bf First order approximation for the viscosity quasi-periodic solution and fixed point argument.} Once we have a good knowledge for properly inverting the operators $\mL_{\nu}$ and $\mL_{\e}$, we are ready to construct quasi-periodic solutions of the Navier Stokes equation converging to the Euler solution $u_e$ as $\nu \to 0$. First, we define an approximate solution $v_{app} = v_e + \nu v_1$ which solves the equation \eqref{equazione vorticita media nulla} up to order $O(\nu^2)$. By making a formal expansion with respect to the viscosity parameter $\nu$, we ask $v_{e}$ to solve the equation at the zeroth order $O(\nu^0)$, namely the Euler equation, whose existence is provided by Theorem \ref{main theorem 2}, and $v_1$ to solve the linear equation at the first order $O(\nu)$, that is $\mL_{e} v_1 = \Delta v_{e}$. This procedure leads to a loss of regularity due to the presence of \emph{ small divisors}, appearing in the inversion of the linearized Euler operator ${\mathcal L}_e$ in \eqref{linearizzato intro}, which satisfies an estimate of the form $\| {\mathcal L}_e^{- 1} h \|_s \lesssim_s \| h \|_{s + \tau}$ for some $\tau \gg 0$ large enough. On the other hand, this is not a problem in our scheme since it appears only twice: first, in the construction of the quasi-periodic solution $v_e$, but it has already been dealt in Theorem \ref{main theorem 2}; second, in the definition indeed of $v_1$. We overcome this issue  by requiring  $v_e$ to be sufficiently regular. The final step to prove Theorem \ref{teorema limite singolare} is to implement a fixed point argument for constructing solutions of the form $v = v_e + \nu v_1 + \psi$, where the quasi-periodic correction $\psi$ lies in the ball $\| \psi \|_s \leq \nu$. It is crucial  here that $v_e + \nu v_1$ is an approximate solution up to order $O(\nu^2)$: indeed, the fixed point iteration asks to invert linearized operator at the Euler solution $\mL_{\nu}$, which has a bound  of order $O(\nu^{- 1})$ (recall \eqref{stima inverso eul intro}), and in this way the new term ends up to be of order $O(\nu)$ as desired. The good news here is that, at this stage, no small divisors are involved and, consequently, no losses of derivatives, which would have made the fixed point argument not applicable otherwise.


\noindent
 {\it 2D vs. 3D.} It is worth to conclude this introduction by making some comments on the 3D case, that it is not covered by the method developed in this paper. 
 In the present paper, we construct global in time quasi-periodic solutions for the \emph{two dimensional} Navier-Stokes equations converging uniformly in time to global quasi-periodic solutions of the two-dimensional forced Euler equation. The three dimensional is much harder. The biggest obstacle is that the reversible structure is not enough to deduce that the spectrum of the linearized Euler operator after the KAM reducibility scheme is purely imaginary. Indeed, as in \cite{BaldiMontalto}, the reduced Euler operator ${\cal D}_\infty$ is a $3 \times 3$ block diagonal operator ${\cal D}_\infty = {\rm diag}_{j \in \Z^3 \setminus \{ 0 \}} D_\infty(j)$ where the $3 \times 3$ matrix $D_\infty(j)$ has the form $D_\infty(j) = \ii \zeta \cdot j {\rm Id} + \e R_\infty(j)$ for $j \in \Z^3 \setminus \{ 0 \}$. This block matrix could have eigenvalues $\mu_1(j), \mu_2(j), \mu_3(j)$ of the form $\mu_i(j) = \ii \zeta \cdot j + \e r_i(j)$, $i = 1,2,3$, with real part different from zero, in particular with ${\rm Re}(r_i(j)) \neq 0$ for some $i = 1,2,3$. This seems to be an obstruction to get a lower bound like \eqref{bound chiave autovalori intro} with a gain of two space derivatives, which holds uniformly in $\e$ and for any value of the viscosity parameter. More precisely, one gets a lower bound on the eigenvalues of the form
 $$
 |\ii \omega \cdot \ell + \mu_i(j) + \nu |j|^2| \geq  |\e {\rm Re}(r_i(j)) + \nu |j|^2 |\,. 
 $$
 It is therefore not clear how to bound the latter quantity by $C \nu |j|^2$ without linking $\e$ and $\nu$ which prevent to pass to the limit as $\nu \to 0$ (independently of $\e$). 
\\[1mm]

\noindent
{\bf Outline of the paper.} The rest of this paper is organized as follows. In Section \ref{sez generale norme e operatori} we introduce the functional setting and some general lemmata that we will employ in the other sections.  In Section \ref{sez generale L} we formulate the nonlinear functional $\mF_{\nu}$ in \eqref{equazione cal F vorticita}, whose zeroes correspond to quasi-periodic solutions of the equation \eqref{equazione vorticita media nulla}, together with the linearized operators that we have to study. In Section \ref{sez riducibilita} we implement the normal form method on of the linearized Euler and Navier Stokes operators operator $\mL_{\e}, \mL_\nu$ in \eqref{linearizzato intro}: first, we regularize to constant coefficients the highest and lower orders, in Sections \ref{sezione riduzione ordine alto} and \ref{sez riduzione ordini bassi} up to sufficiently smoothing orders; then, in Sections \ref{sezione riducibilita a blocchi}-\ref{sez convergenza KAM} we prove the full KAM reducibility scheme. We shall prove that the normal form transformations conjugate the linearized Navier Stokes operator to a diagonal one plus a remainder which is unbounded of order two and has size $O(\e \nu)$. This normal form procedure is uniform w.r. to the viscosity parameter since it requires a smallness condition on $\e$ which is independent of the viscosity $\nu > 0$. Then, in Section \ref{sez inversioni} we show the invertibility of the operator $\mL_{\nu}$ (and also $\mL_e$) that will be used in Section \ref{sez soluzione approx} for the construction of the first order approximate solution and in Section \ref{sez fixed point} for the fixed point argument. Finally, the proof of Theorem \ref{teorema limite singolare} is provided in Section \ref{sezione teoremi principali}, together with the measure estimates proved in Section \ref{sezione stime di misura}.

\noindent

\medskip

\bigskip

{\sc Acknowledgements.} The authors warmly thank Gennaro Ciampa, Nader Masmoudi and Eugene Wayne for many stimulating discussions on the topic and on related results. The work of the author R. M. is supported by the ERC STARTING GRANT "Hamiltonian Dynamics, Normal Forms and Water Waves" (HamDyWWa), project number: 101039762. R. M. is also supported by INDAM-GNFM. The work of the author 
L.F. is supported by Tamkeen under the NYU Abu Dhabi Research Institute grant CG002.

\section{Norms and linear operators}\label{sez generale norme e operatori}
In this section we collect some general definitions and properties concerning norms and matrix representation of operators which are used in the whole paper. 

\medskip

\noindent
{\bf Notations.} In the whole paper, the notation $ A \lesssim_{s, m} B $ means
that $A \leq C(s, m) B$ for some constant $C(s, m) > 0$ depending on 
the Sobolev index $ s $  and a generic constant $ m $. We always omit to write the dependence on $d$, which is the number of frequencies, and $\tau$, which is the constant appearing in the non-resonance conditions (see for instance \eqref{insiemi di cantor rid}). We often write $u = {\rm even}(\vphi, x)$ if $u \in X$ and $u = {\rm odd}(\vphi, x)$ if $u \in Y$ (recall \eqref{X even Y odd}).  For a given Banach space $Z$, we recall that $\mB(Z)$ denotes the space of bounded operators from $Z$ into itself.
\subsection{Function spaces}
\label{subsec:function spaces}
Let $a : \T^d \times \T^2 \to \C$, $a = a(\ph,x)$, 
be a function. then, for $s \in \R$, its Sobolev norm $\| a \|_s$ is defined as 
\begin{equation*} \label{def Sobolev norm generale}
\| a \|_s^2 := \sum_{(\ell, j) \in \Z^d \times \Z^2} 
\langle \ell, j \rangle^{2s} | \widehat a(\ell,j) |^2 ,
\quad \ 
\langle \ell, j \rangle := \max \{ 1, |\ell|, |j| \},
\end{equation*}
where $\widehat a(\ell,j)$ (which are scalars, or vectors, or matrices) 
are the Fourier coefficients of $a(\ph,x)$, namely
\[
\widehat a(\ell,j) := \frac{1}{(2\pi)^{d+2}} \int_{\T^{d+2}} 
a(\ph,x) e^{- \ii (\ell \cdot \ph + j \cdot x)} \, d\ph dx. 
\]
We denote, for $E = \C^n$ or $\R^n$,
\begin{equation}	\label{def sobolev}
	\begin{aligned}
		H^s 
		&
		:= H^s_{\ph,x} 
		:= H^s(\T^{d} \times \T^2) \\
		&
		:= H^s(\T^{d} \times \T^2, E) 
		:= \{ u : \T^{d} \times \T^2 \to E, \ \| u \|_s < \infty \},
	\end{aligned}
\end{equation}

In the paper we use Sobolev norms for 
(real or complex, scalar- or vector- or matrix-valued) functions $u( \ph, x; \om, \zeta)$, 
$(\ph,x) \in \T^d \times \T^2$, being Lipschitz continuous with respect to the parameters $\lambda:=(\om,\zeta) \in \R^{d+2}$.
We fix 
\begin{equation}\label{definizione s0}
s_0 \geq  \lfloor\tfrac{d+2}{2}\rfloor + 2
\end{equation}
once and for all, 
and define the weighted Sobolev norms in the following way. 

\begin{definition} 
{\bf (Weighted Sobolev norms).} 
\label{def:Lip F uniform} 
Let   $\g \in (0,1]$, $\Lambda \subseteq \R^{d + 2}$ 
and $s \geq s_0$. 
Given a function $u : \Lambda \to H^s(\T^d \times \T^2)$, 
$\lm \mapsto u(\lm) = u(\ph,x; \lm)$ 
that is Lipschitz continuous with respect to $\lm$, 
we define its weighted Sobolev norm by
$$
\| u \|_{s}^{\Lip(\g)} := \| u\|_{s}^{\sup} + \g \,\| u\|_{s-1}^{\rm lip}\,,
$$
where
\begin{equation*}
	\| u\|_{s}^{\sup} := \sup_{\lambda\in \Lambda} \| u(\lambda)\|_{s}\,, \quad \| u\|_{s}^{\rm lip}:= \sup_{\lambda_1,\lambda_2\in \Lambda  \atop \lambda_1\neq \lambda_2} \frac{\| u(\lambda_1)-u(\lambda_2)\|_{s}}{| \lambda_1-\lambda_2|}\,.
\end{equation*}
For $u$ independent of $(\ph,x)$, we simply denote by 
$| u |^{k_0,\g}:= | u|^{\sup} + \g \, | u|^{\rm lip} $.
\end{definition}

For any $N>0$, we define the smoothing operators (Fourier truncation)
\begin{equation}\label{def:smoothings}
(\Pi_N u)(\ph,x) := \sum_{\la \ell,j \ra \leq N} \hat u(\ell, j) e^{\ii (\ell\cdot\ph + j \cdot x)}, \qquad
\Pi^\perp_N := {\rm Id} - \Pi_N.
\end{equation}

\begin{lemma} {\bf (Smoothing).} \label{lemma:smoothing}
The smoothing operators $\Pi_N, \Pi_N^\perp$ satisfy 
the smoothing estimates
\begin{align}
\| \Pi_N u \|_{s}^{\Lip(\g)} 
& \leq N^a \| u \|_{s-a}^{\Lip(\g)}\, , \quad 0 \leq a \leq s, 
\label{p2-proi} \\
\| \Pi_N^\bot u \|_{s}^{\Lip(\g)}
& \leq N^{-a} \| u \|_{s + a}^{\Lip(\g)}\, , \quad  a \geq 0.
\label{p3-proi}
\end{align}
\end{lemma}

\begin{lemma}{\bf (Product).}
\label{lemma:LS norms}
 For all $ s \geq s_0$, 
\begin{align}
\| uv \|_{s}^{\Lip(\g)}
& \lesssim_s C(s)  \| u \|_{s}^{\Lip(\g)} \| v \|_{s_0}^{\Lip(\g)}+ C(s_0)  \| u \|_{s_0}^{\Lip(\g)} \| v \|_{s}^{\Lip(\g)}\,.
\label{p1-pr}
\end{align}
\end{lemma}

\subsection{Matrix representation of linear operators}
Let
${\mathcal R}  : L^2(\T^2) \to L^2(\T^2)$ be a linear operator. Such an operator can be represented as
\begin{equation}\label{matriciale 1}
{\mathcal R} u (x) := \sum_{j, j' \in \Z^2} {\mathcal R}_j^{j'}\widehat u(j') e^{\ii j \cdot x}, 
\quad  \text{for} \quad  u (x) = \sum_{j \in \Z^2} \widehat u(j) e^{\ii j \cdot x}, 
\end{equation}
where, for $j, j' \in \Z^2$, the matrix element ${\mathcal R}_j^{j'}$ is defined by 
\begin{equation}\label{rappresentazione blocchi 3 per 3}
{\mathcal R}_j^{j'} :=  \frac{1}{(2\pi)^2} \int_{\T^2} 
{\mathcal R}[e^{\ii j' \cdot x}] e^{- \ii j \cdot x}\, d x\,. 
\end{equation}

We also consider smooth $\vphi$-dependent families of linear operators 
$\T^d \to {\mathcal B} (L^2(\T^2))$, $\vphi \mapsto {\mathcal R}(\vphi)$, 
which we write in Fourier series with respect to $\vphi$ as 
\begin{equation}\label{matrix representation 1}
{\mathcal R}(\vphi) = \sum_{\ell \in \Z^d} \widehat{\mathcal R}(\ell) e^{\ii \ell \cdot \vphi}, \quad \widehat{\mathcal R}(\ell) := \frac{1}{(2 \pi)^d} \int_{\T^d} {\mathcal R}(\vphi) e^{- \ii \ell \cdot \vphi}\, d \vphi, \quad \ell \in \Z^d\,. 
\end{equation}
According to \eqref{rappresentazione blocchi 3 per 3}, for any $\ell \in \Z^d$, the linear operator $\widehat{\mathcal R}(\ell) \in {\mathcal B} (L^2(\T^2))$ is identified 
with the matrix $(\widehat{\mathcal R}(\ell)_j^{j'})_{j, j' \in \Z^2}$.
A map $\T^d \to {\mathcal B} (L^2(\T^2))$, $\vphi \mapsto {\mathcal R}(\vphi)$ 
can be also regarded as a linear operator 
$L^2(\T^{d + 2}) \to L^2(\T^{d + 2})$ by 
\begin{equation} \label{amatriciana}
{\mathcal R} u(\vphi, x) := \sum_{\begin{subarray}{c}
\ell, \ell' \in \Z^d \\
j, j' \in \Z^2
\end{subarray}} \widehat{\mathcal R}(\ell - \ell')_j^{j '} \widehat u(\ell', j') e^{\ii (\ell \cdot \vphi + j \cdot x)}, \quad \forall u \in L^2(\T^{d + 2})\,. 
\end{equation}
If the operator ${\mathcal R}$ is invariant on the space of functions with zero average in $x$, we identify ${\mathcal R}$ with the matrix 
$$
\Big( \widehat{\mathcal R}(\ell)_j^{j'} \Big)_{\begin{subarray}{c}
j , j' \in \Z^2 \setminus \{ 0 \} \\
\ell \in \Z^d
\end{subarray}}
$$
\begin{definition}
	{\bf (Diagonal operators).}
	\label{def block-diagonal op}
	Let ${\mathcal R}$ be a linear operator as in 
	\eqref{matriciale 1}-\eqref{amatriciana}. We define ${\mathcal D}_{\mathcal R}$ as the operator defined by 
	\begin{equation*}
		{\mathcal D}_{\mathcal R} := {\rm diag}_{j \in \Z^2} \widehat{\mathcal R}_j^j(0)\,, \quad (\mD_{\mR})_j^{j'}(\ell) := \begin{cases}
			\widehat{\mathcal R}_j^{j}(\ell) & j=j'\,, \ \ell=0\,, \\
			0 & \text{otherwise}\,.
		\end{cases}\,.
	\end{equation*}
	In particular, we say that $\mR$
	is a \emph{diagonal operator} if $\mR \equiv \mD_{\mR}$.
\end{definition}

For the purpose of the  Normal form method for the linearized operator in Section \ref{sez riducibilita}, it is convenient to introduce the following norms that take into account the order and the off-diagonal decay of the matrix elements representing any linear operator on $L^2(\T^{d+2})$.
\begin{definition} \label{block norm}
{\bf (Matrix decay norm and the class ${\mathcal O}{\mathcal P}{\mathcal M}^m_s$).}
Let $m \in \R$, $s \geq s_0$ and $\mR$ be an operator represented by the matrix in \eqref{amatriciana}. We say that ${\mathcal R}$ belongs to the class ${\mathcal O}{\mathcal P}{\mathcal M}^m_s$ if 
\begin{equation} \label{def decay norm}
| \mR |_{m, s} := \sup_{j' \in \Z^2} 
\Big( \sum_{(\ell,j) \in \Z^{d+2}} 
\langle \ell, j-j' \rangle^{2s} | \widehat \mR(\ell)_j^{j'}|^2  \Big)^{\frac12} \langle j' \rangle^{- m} < \infty
\end{equation}
If the operator $\mR = \mR(\lm)$ is Lipschitz with respect to the parameter $\lambda \in \Lambda \subseteq  \R^{\nu+2}$, 
we define 
\begin{equation} \label{def decay norm parametri}
\begin{aligned}
& | {\mathcal R} |_{m, s}^{{\rm Lip}(\gamma)} := |{\mathcal R}|_{m, s}^{\rm sup} + \gamma |{\mathcal R}|^{\rm lip}_{m, s - 1}\,, \\
& |{\mathcal R}|_{m, s}^{\rm sup} := \sup_{\lambda \in \Lambda} |{\mathcal R}(\lambda)|_{m, s}, \quad |{\mathcal R}|_{m, s - 1}^{\rm lip} := \sup_{\begin{subarray}{c}
\lambda_1, \lambda_2 \in \Lambda \\
\lambda_1 \neq \lambda_2
\end{subarray}} \dfrac{|{\mathcal R}(\lambda_1) - {\mathcal R}(\lambda_2)|_{m, s - 1}}{|\lambda_1 - \lambda_2|}
\end{aligned}
\end{equation}
\end{definition}

Directly from the latter definition, it follows that 
\begin{equation*}\label{prop elementari}
\begin{aligned}
& m  \leq m' \Longrightarrow {\mathcal O}{\mathcal P}{\mathcal M}^m_s \subseteq {\mathcal O}{\mathcal P}{\mathcal M}^{m'}_s \quad \text{and} \quad |\cdot |_{m', s}^{{\rm Lip}(\gamma)} \leq |\cdot |_{m, s}^{{\rm Lip}(\gamma)}, \\
& s \leq s' \Longrightarrow {\mathcal O}{\mathcal P}{\mathcal M}^m_{s'} \subseteq {\mathcal O}{\mathcal P}{\mathcal M}^m_s \quad \text{and} \quad | \cdot |_{m, s}^{{\rm Lip}(\gamma)} \leq |\cdot|_{m, s'}^{{\rm Lip}(\gamma)}\,. 
\end{aligned}
\end{equation*}
We now state some 
standard properties of the decay norms 
that are needed for the reducibility scheme 
of Section \ref{sezione riducibilita a blocchi}. If $a \in H^s$, $s \geq s_0$, then the multiplication operator ${\mathcal M}_a : u \mapsto a u$ satisfies 
\begin{equation}\label{prop multiplication decay}
{\mathcal M}_a \in {\mathcal O}{\mathcal P}{\mathcal M}_s^0 \quad \text{and} \quad |{\mathcal M}_a|_{0, s}^{{\rm Lip}(\gamma)} \lesssim \| a \|_s^{{\rm Lip}(\gamma)}\,. 
\end{equation}

\begin{lemma}\label{proprieta standard norma decay}
$(i)$ Let $s \geq s_0$ and ${\mathcal R} \in {\mathcal O}{\mathcal P}{\mathcal M}^0_s$. 
If $\| u \|_s^{{\rm Lip}(\gamma)} < \infty$, then 
$$
\| {\mathcal R} u \|_s^{{\rm Lip}(\gamma)} \lesssim_{s} |{\mathcal R}|_{0, s}^{{\rm Lip}(\gamma)} \| u \|_s^{{\rm Lip}(\gamma)}\,. 
$$
$(ii)$ Let $s \geq s_0$, $m, m' \in \R$, and let ${\mathcal R} \in {\mathcal O}{\mathcal P}{\mathcal M}^m_s$, ${\mathcal Q} \in {\mathcal O}{\mathcal P}{\mathcal M}^{m'}_{s + |m|}$. 
Then ${\mathcal R} {\mathcal Q} \in {\mathcal O}{\mathcal P}{\mathcal M}^{m + m'}_s$ and 
$$
|{\mathcal R}{\mathcal Q}|_{m + m', s}^{{\rm Lip}(\gamma)} \lesssim_{s, m} |{\mathcal R}|_{m, s}^{{\rm Lip}(\gamma)} |{\mathcal Q}|_{m', s_0 + |m|}^{{\rm Lip}(\gamma)} + |{\mathcal R}|_{m, s_0}^{{\rm Lip}(\gamma)} |{\mathcal Q}|_{m', s + |m|}^{{\rm Lip}(\gamma)}\ \,. 
$$
$(iii)$ Let $s \geq s_0$ and ${\mathcal R} \in {\mathcal O}{\mathcal P}{\mathcal M}^0_s$. 
Then, for any integer $n \geq 1$, ${\mathcal R}^n \in {\mathcal O}{\mathcal P}{\mathcal M}^0_s$ and there exist constants $C(s_0),C(s) > 0$, independent of $n$, such that 
\begin{equation*}
	\begin{aligned}
		& |{\mathcal R}^n|_{0, s_0}^{{\rm Lip}(\gamma)} \leq C(s_0)^{n - 1} \big(|{\mathcal R}|_{0, s_0}^{{\rm Lip}(\gamma)}\big)^{n} \,, \\
		& |{\mathcal R}^n|_{0, s}^{{\rm Lip}(\gamma)} \leq n\,C(s)^{n - 1} \big(C(s_0)|{\mathcal R}|_{0, s_0}^{{\rm Lip}(\gamma)}\big)^{n - 1} |{\mathcal R}|_{0, s}^{{\rm Lip}(\gamma)}\,;
	\end{aligned}
\end{equation*}
$(iv)$ Let $s \geq s_0$, $m \geq 0$ and ${\mathcal R} \in {\mathcal O}{\mathcal P}{\mathcal M}^{- m}_s$. 
Then there exists $\delta(s) \in (0, 1)$ small enough such that, 
if $|{\mathcal R}|_{- m, s_0}^{{\rm Lip}(\gamma)} \leq \delta(s)$, 
then the map $\Phi = {\rm Id} + {\mathcal R}$ is invertible 
and the inverse satisfies the estimate 
\[
|\Phi^{- 1} - {\rm Id}|_{- m, s}^{{\rm Lip}(\gamma)} \lesssim_{s, m} |{\mathcal R}|_{- m, s}^{{\rm Lip}(\gamma)}. 
\]
\noindent
$(v)$ Let $s \geq s_0$, $m \in \R$ and ${\mathcal R} \in {\mathcal O}{\mathcal P}{\mathcal M}^m_s$. Let ${\mathcal D}_{\mathcal R}$ be the diagonal operator as in Definition \ref{def block-diagonal op}.
Then ${\mathcal D}_{\mathcal R} \in {\mathcal O}{\mathcal P}{\mathcal M}^m_s$ and $|{\mathcal D}_{\mathcal R}|_{m, s}^{{\rm Lip}(\gamma)} \lesssim |{\mathcal R}|_{m, s}^{{\rm Lip}(\gamma)}$. 
As a consequence, 
\[
| \widehat{\mathcal R}_j^j(0) |^{{\rm Lip}(\gamma)} \lesssim \langle j \rangle^m|{\mathcal R}|_{s_0}^{{\rm Lip}(\gamma)}\,.
\] 
\end{lemma}

\begin{proof}
$(i), (ii)$ The proofs of the first two items use similar arguments. We only prove item $(ii)$. 
%
We start by assuming that both $\mR$ and $\mQ$ do not depend on the parameter $\lambda$. The matrix elements for the composition operator $\mR\mQ$ follow the rule
\begin{equation*}
	\widehat{\mR\mQ}(\ell)_j^{j'} = \sum_{(k, i) \in \Z^{d + 2}} \widehat{\mR}(\ell- k)_{j}^{i} \widehat{\mQ}(k)_{i}^{j'}\,.
\end{equation*}
Using that $\langle \ell, j - j' \rangle^s \lesssim_s \langle \ell - k, j - i \rangle^s + \langle k, i - j' \rangle^s$, one gets
\begin{align}
\sum_{\begin{subarray}{c}
\ell \in \Z^d \\
j\in \Z^2 
\end{subarray}} \langle \ell, j - j' \rangle^{2 s} |	\widehat{\mR\mQ}(\ell)_j^{j'}|^2 \langle j' \rangle^{ - 2 (m + m')} & \lesssim_s (A) + (B) \,, \label{paris 0} 
\end{align}
where 
\begin{align}
&
(A) := \sum_{\begin{subarray}{c}
\ell \in \Z^d \\
j \in \Z^2
\end{subarray}} \Big( \sum_{\begin{subarray}{c}
k \in \Z^d \\
i \in \Z^2
\end{subarray}} \langle \ell - k, j - i \rangle^{s} |\widehat{\mathcal R}_j^{i}(\ell - k) | | \widehat{\mathcal Q}_{i}^{j'}(k) | \Big)^2 \langle j' \rangle^{- 2(m + m')}\,,  \nonumber	
\\
&
(B) := \sum_{\begin{subarray}{c}
\ell \in \Z^d \\
j \in \Z^2
\end{subarray}} \Big( \sum_{\begin{subarray}{c}
k \in \Z^d \\
i  \in \Z^2
\end{subarray}} \langle k ,  i - j'  \rangle^{s} |\widehat{\mathcal R}_j^{i}(\ell - k) | | \widehat{\mathcal Q}_{i}^{j'}(k) | \Big)^2\langle j' \rangle^{- 2(m + m')} \,.\nonumber 
\end{align}
We start with estimating $(A)$. 
By the elementary inequality 
$
\langle i \rangle^m\langle j' \rangle^{- m} \lesssim_m  \langle j' - i \rangle^{|m|}
$,
the Cauchy-Schwartz inequality and having the series $\sum_{k \in \Z^d, i \in \Z^2} \langle k, i - j'\rangle^{- 2 s_0}  = C(s_0)<\infty$,
one has
\begin{equation*}
	\begin{footnotesize}
		\begin{aligned}
			(A) &  \lesssim_{s,m} \sum_{\begin{subarray}{c}
					k \in \Z^d \\
					i \in \Z^2
			\end{subarray}} \langle k, i - j' \rangle^{2 (s_0 + |m|)}| \widehat {\mathcal Q}_{i}^{j'}(k) |^2 \langle j' \rangle^{- 2 m'} \sum_{\begin{subarray}{c}
					\ell \in \Z^d \\
					j \in \Z^2
			\end{subarray}} \langle \ell - k, j - i \rangle^{ 2s} |\widehat{\mathcal R}_j^{i}(\ell - k) |^2 \langle i \rangle^{- 2 m} 
			\\
			& \stackrel{\eqref{def decay norm}}{\lesssim_{s,m}} |{\mathcal Q}|_{s_0 + |m|, m'}^2 |{\mathcal R}|_{s, m}^2  
		\end{aligned}
	\end{footnotesize}
\end{equation*}
By similar arguments, one gets $(B) \lesssim_m |{\mathcal Q}|_{s + |m|, m'}^2 |{\mathcal R}|_{s_0, m}^2 $ and hence the claimed estimate follows by taking the supremum over $j' \in \Z^2$ in \eqref{paris 0}. If we reintroduce the dependence on the parameter $\lambda$, the  estimate for the Lipschitz seminorm follows as usual by taking two parameters $\lambda_1, \lambda_2$ and writing ${\mathcal R}(\lambda_1) {\mathcal Q}(\lambda_1) - {\mathcal R}(\lambda_2) {\mathcal Q}(\lambda_2) = ({\mathcal R}(\lambda_1) - {\mathcal R}(\lambda_2)) {\mathcal Q}(\lambda_1) + {\mathcal R}(\lambda_2) ({\mathcal Q}(\lambda_1) -  {\mathcal Q}(\lambda_2))$. 

$iii)$ The claim follows by an induction argument and item $(ii)$.

$iv)$ The claim follows by a Neumann series argument, together with item $(iii)$.

$v)$ The claims are a direct consequence of the definition of the matrix decay norm in Definition \ref{block norm}.
\end{proof}

We recall the definition of the set of the Diophantine vectors in a bounded, measurable set $\Lambda \subset \R^{d + 2}$. Given $\gamma, \tau > 0$, we define 
\begin{equation}\label{set DC trasporto a}
	\Lambda(\gamma,\tau):= \big\{  (\omega,\zeta)\in \Lambda \,:\, |\omega\cdot \ell+\zeta\cdot j| \geq \frac{\gamma}{|(\ell,j)|^{\tau}}\,, \ \forall\,(\ell,j)\in\Z^{d+2}\setminus\{0\}  \big\}\,,
\end{equation}
where $|(\ell,j)|:=|\ell|+|j|$ for any $\ell\in\Z^d$, $j\in\Z^2$.

\begin{lemma}\label{lemma eq omologica descent}
	{\bf (Homological equation).}
Let $\Lambda\ni \lambda=(\omega,\zeta)\mapsto{\mathcal R}(\lambda)$ be a Lipschitz family of linear operators in ${\mathcal O}{\mathcal P}{\mathcal M}^m_{s + 2 \tau + 1}$. Then, for any $\lambda \in \Lambda(\gamma, \tau)$, there exists a solution $\Psi = \Psi(\lambda) \in {\mathcal O}{\mathcal P}{\mathcal M}^m_{s}$ of the equation 
\begin{equation}\label{eq omologica descent astratta}
\omega \cdot \partial_\vphi \Psi + [\zeta \cdot \nabla, \Psi] + {\mathcal R} = {\mathcal D}_{\mathcal R}
\end{equation}
satisfying the estimate $|\Psi|_{m, s}^{{\rm Lip}(\gamma)} \lesssim \gamma^{- 1} |{\mathcal R}|_{m, s + 2 \tau + 1}^{{\rm Lip}(\gamma)}$. Moreover, if ${\mathcal R}$ is invariant on the space of zero average functions, also $\Psi$ is invariant on the space of zero average functions. 
\end{lemma}
\begin{proof}
By the matrix representation \eqref{matrix representation 1}, \eqref{amatriciana}, the equation \eqref{eq omologica descent astratta} is equivalent to 
$$
\ii \big(\omega \cdot \ell + \zeta \cdot (j - j') \big) \widehat \Psi_j^{j'}(\ell) + \widehat{\mathcal R}(\ell)_j^{j'} = 0
$$
for any $(\ell, j, j') \in \Z^d \times \Z^2 \times \Z^2$ with $(\ell, j, j') \neq (0, j, j)$. We then define $\Psi$ as 
\begin{equation*}\label{def Psi astratto descent}
\widehat \Psi(\ell)_j^{j'} := \begin{cases}
- \dfrac{\widehat{\mathcal R}(\ell)_j^{j'}}{\ii \big(\omega \cdot \ell + \zeta \cdot (j - j') \big)} \quad \text{if} \quad (\ell, j, j') \neq (0, j, j)\,, \\
0 \quad \text{otherwise.}
\end{cases}
\end{equation*}
Since $\lambda = (\omega, \zeta) \in \Lambda(\gamma, \tau)$, by \eqref{set DC trasporto a}, one has that 
$$
|\widehat \Psi(\ell)_j^{j'}| \leq \gamma^{- 1} \langle \ell, j - j' \rangle^\tau |\widehat{\mathcal R}(\ell)_j^{j'}|\,.
$$
The latter estimate, together with the Definition \ref{def decay norm}, implies that 
\begin{equation}\label{Psi cal R sup}
|\Psi|_{m, s} \lesssim \gamma^{- 1} |{\mathcal R}|_{m, s + \tau}.
\end{equation} We prove now the Lipschitz estimate. Let 
$\lambda_1, \lambda_2  \in \Lambda(\gamma, \tau)$. 
A direct computation shows that 
$$
\begin{aligned}
|\widehat \Psi(\ell)_j^{j'}(\lambda_1) - \widehat \Psi(\ell)_j^{j'}(\lambda_2)| & \lesssim \gamma^{- 1} \langle \ell, j - j' \rangle^\tau |\widehat{\mathcal R}(\ell)_j^{j'} (\lambda_1) - \widehat{\mathcal R}(\ell)_j^{j'}(\lambda_2)| \\
& \quad + \gamma^{- 2} \langle \ell, j - j' \rangle^{2 \tau + 1} |\widehat{\mathcal R}(\ell)_j^{j'}(\lambda_2)| |\lambda_1 - \lambda_2|\,.
\end{aligned}
$$
Hence by recalling \eqref{def decay norm}, \eqref{def decay norm parametri} one obtains that 
\begin{equation}\label{Psi cal R lip}
|\Psi|_{m, s}^{\rm lip} \lesssim \gamma^{- 1} |{\mathcal R}|_{m, s + \tau}^{\rm lip} + \gamma^{- 2} |{\mathcal R}|^{\rm sup}_{m, s + 2 \tau + 1}\,.
\end{equation}
The estimates \eqref{Psi cal R sup}, \eqref{Psi cal R lip} imply the claimed bound $|\Psi|_{m, s}^{{\rm Lip}(\gamma)} \lesssim \gamma^{- 1} |{\mathcal R}|_{m, s + 2 \tau + 1}^{{\rm Lip}(\gamma)}$. 
\end{proof}

For $N > 0$, we define the operators $\Pi_N {\mathcal R}$ and $\Pi_N^\perp \mR$ by means of their matrix representation as follows: 
\begin{equation}\label{def proiettore operatori matrici}
(\widehat{\Pi_N {\mathcal R}})_{j}^{j'}(\ell) := \begin{cases}
\widehat{\mathcal R}_j^{j'}(\ell) & \text{if } |\ell|, |j - j'| \leq N, \\
0 & \text{otherwise}\,, 
\end{cases} \qquad   \ \Pi_N^\bot {\mathcal R} := {\mathcal R} - \Pi_N {\mathcal R}\,. 
\end{equation}


\begin{lemma}\label{lemma proiettori decadimento}
For all $s, \a \geq 0$, $m \in \R$, one has 
$|\Pi_N {\mathcal R}|_{m, s + \alpha}^{{\rm Lip}(\gamma)} \leq N^\alpha |{\mathcal R}|_{m, s}^{{\rm Lip}(\gamma)}$ and $|\Pi_N^\bot {\mathcal R}|_{m, s}^{{\rm Lip}(\gamma)} \leq N^{- \alpha} |{\mathcal R}|_{m, s + \alpha}^{{\rm Lip}(\gamma)}$. 
\end{lemma}
\begin{proof}
	The claims follow directly from \eqref{def decay norm} and \eqref{def proiettore operatori matrici}.
\end{proof}

We also define the projection $\Pi_0$ on the space of zero average functions as 
\begin{equation*}\label{definizione proiettore media spazio tempo}
	\Pi_0 h := \frac{1}{(2 \pi)^{ 2}} \int_{\T^{ 2}} h(\vphi, x)\, d x, 
	\qquad 
	\Pi_0^\bot := {\rm Id} - \Pi_0\,.
\end{equation*}
%
%
In particular,
for any $m, s \geq 0$,
\begin{equation}\label{stima Pi 0}
|\Pi_0^\bot|_{0, s} \leq 1\,, \quad |\Pi_0|_{- m, s} \lesssim_{m} 1\,. 
\end{equation}
We finally mention the elementary properties of the Laplacian operator $- \Delta$ and its inverse $(- \Delta)^{- 1}$ acting on functions with zero average in $x$:
$$
- \Delta u(x) = \sum_{\xi \neq 0} |\xi|^2 \widehat u(\xi) e^{\ii x \cdot \xi}, \quad (- \Delta)^{- 1} u(x) = \sum_{\xi \in \Z^2 \setminus \{ 0 \}} \frac{1}{|\xi|^2} \widehat u(\xi) e^{\ii x \cdot \xi}\,. 
$$
By Definition \ref{block norm}, one easily verifies, for any $s\geq 0$,
\begin{equation}\label{decay laplace}
|- \Delta|_{2, s} \leq 1\,, \quad |(- \Delta)^{- 1}|_{- 2, s} \leq 1\,. 
\end{equation}

\subsection{Real and reversible operators}\label{Reversible operators}

We recall the notation introduced in  \eqref{X even Y odd}, that is, for any function $u(\ph,x)$, 
 we write $u \in X$ when $u = \even(\ph,x)$
and $u \in Y$ when $u = \odd(\ph,x)$. 

\begin{definition}
$(i)$ We say that a linear operator $\Phi$ is \emph{reversible} 
if $\Phi : X \to Y$ and $\Phi : Y \to X$. 
We say that $\Phi$ is \emph{reversibility preserving} 
if $\Phi : X \to X$ and $\Phi : Y \to Y$. 

\noindent
$(ii)$ We say that an operator $\Phi : L^2(\T^2) \to L^2(\T^2)$ is \emph{real} if $\Phi(u)$ is real valued for any $u$ real valued. 
\end{definition}
It is convenient to reformulate real and reversibility properties of linear operators in terms of their matrix representations.
\begin{lemma}\label{lemma real rev matrici}
A linear operator ${\mathcal R}$ is :

\noindent
$(i)$ real if and only if 
$\widehat{\mathcal R}_{j}^{j'}(\ell) = \overline{\widehat{\mathcal R}_{- j}^{- j'}(- \ell)}$ 
for all $\ell \in \Z^d$, $j, j' \in \Z^2$;

\noindent
$(ii)$ reversible if and only if 
$\widehat{\mathcal R}_j^{j'}(\ell) = - \widehat{\mathcal R}_{- j}^{- j'}(- \ell)$ 
for all $\ell \in \Z^d$, $j, j' \in \Z^2$;

\noindent
$(iii)$ reversibility-preserving if and only if 
$\widehat{\mathcal R}_j^{j'}(\ell) =  \widehat{\mathcal R}_{- j}^{- j'}(- \ell)$ 
for all $\ell \in \Z^d$, $j, j' \in \Z^2$.
\end{lemma}

\section{The nonlinear functional and the  linearized Navier Stokes operator at the Euler solution}\label{sez generale L}

We shall show the existence of solutions of \eqref{equazione vorticita media nulla} 
by finding zeroes of the nonlinear operator 
$
{\mathcal F}_\nu : H^{s + 2}_0(\T^{d + 2}) \to H^s_0(\T^{d + 2}) 
$
defined by 
\begin{equation}\label{equazione cal F vorticita}
	\begin{aligned}
		{\mathcal F}_\nu(v) & := 
		\omega \cdot \partial_\vphi v + \zeta \cdot \nabla v  - \nu \Delta v + \e \big(\Pi_0^\bot 
		\big[  \nabla_\bot (- \Delta)^{- 1} v \cdot \nabla v \big] - F(\vphi, x) \big), 
	\end{aligned}
\end{equation}
with $\nabla_\bot$ as in \eqref{equazione vorticita media nulla},
and, without loss of generality, $F = \nabla \times f$ has zero average in space, namely 
\begin{equation*}\label{ipotesi media nulla f}
	\int_{\T^2} F(\vphi, x)\, dx = 0\,, \quad \forall \,\vphi \in \T^d\,.
\end{equation*}
We 
consider parameters $(\omega, \zeta)$ in a bounded open set 
$\Omega \subset \R^d \times \R^2$; 
we will use such parameters along the proof 
in order to impose appropriate non resonance conditions.

In this section and Section \ref{sez riducibilita} 
we assume the following ansatz, which is implied by Theorem \ref{main theorem 2}:
there exists $S \gg 0$ large enough such that $v_e(\cdot; \lambda) \in H^{S}_0(\T^d \times \T^2)$, $\lambda \in \Omega_\e$, is a solution of the Euler equation satisfying   
\begin{equation}\label{ansatz}
	\| v_e \|_{S}^{{\rm Lip}(\gamma)} \lesssim_S \e^{\mathtt a} \ll 1\,, \qquad \mathtt a \in (0, 1)\,, \quad S > \overline S
\end{equation}
where $\overline S := \overline S(d)$ is the minimal regularity threshold for the existence of quasi-periodic solutions of the Euler equation provided by Theorem \ref{main theorem 2}. 
\noindent
We want to study the linearized operator ${\mathcal L }_\nu := \di {\mathcal F}_\nu(v_e)$ at the solution of the Euler equation $v_e$. 
where $\mF_\nu(v)$ is defined in \eqref{equazione cal F vorticita}. The linearized operator has the form 
\begin{equation}\label{operatore linearizzato} 
\begin{aligned}
{\mathcal L}_{\nu}&  =  {\mathcal L}_e - \nu \Delta\,, \\
{\mathcal L}_e &:= \omega \cdot \partial_\vphi  + 
\big(\zeta  + \e a(\vphi, x) \big)\cdot \nabla + \e {\mathcal R}(\vphi) 
\end{aligned}
\end{equation}
where $a(\ph,x)$ is the function defined by 
\begin{equation}\label{def a}
a(\vphi, x) :=
  \nabla_\bot  (- \Delta)^{- 1} v_e\,, 
\end{equation}
with $\nabla_\bot$ as in \eqref{equazione vorticita media nulla} and  ${\mathcal R}(\vphi)$ is a 
pseudo-differential operator of order $- 1$, given by 
\begin{equation} \label{definizione cal R}
\begin{aligned}
& {\mathcal R}(\vphi)  h :=   \nabla v_e(\vphi,x) \cdot   \nabla_\bot  (- \Delta)^{- 1}h  \,. 
\end{aligned}
\end{equation}
Using that ${\rm div} (\nabla_\bot h) = 0$ for any $h$, the operators $a \cdot \nabla$, ${\mathcal R}$ ${\mathcal L}_\nu$ and ${\mathcal L}_e$ leave invariant the subspace of zero average function, with
\begin{equation}\label{invarianze cal L}
\begin{aligned}
& [\Pi_0^\bot\,,\, a \cdot \nabla ] = 0\,, \, [\Pi_0^\bot \,,\, {\mathcal R}] = 0\,, \\
&  a \cdot \nabla \Pi_0 = \Pi_0 a \cdot \nabla = 0\,,\, {\mathcal R} \Pi_0 = \Pi_0 {\mathcal R} = 0 \\
&[\Pi_0^\bot , {\mathcal L}_{\nu}] = 0 = [\Pi_0^\bot , {\mathcal L}_e] , \quad \Pi_0 {\mathcal L}_{\nu} = {\mathcal L}_{\nu} \Pi_0 = 0\,, \quad  \Pi_0 {\mathcal L}_{e} = {\mathcal L}_{e} \Pi_0 = 0 \,,
\end{aligned}
\end{equation}
implying that
\begin{equation*}
a \cdot \nabla = \Pi_0^\bot a \cdot \nabla \Pi_0^\bot\,, \quad {\mathcal R} = \Pi_0^\bot {\mathcal R} \Pi_0^\bot\,, \quad 	{\mathcal L}_{\nu} = \Pi_0^\bot {\mathcal L}_{\nu} \Pi_0^\bot \,, \quad {\mathcal L}_e = \Pi_0^\bot {\mathcal L}_e \Pi_0^\bot\,. 
\end{equation*}
We always work on the space of zero average functions and we shall preserve this invariance along the whole paper.

The goal of next two sections is to invert the whole linearized Navier Stokes operator $\mL_\nu$ obtained by linearizing the nonlinear functional $\mF_\nu(v)$ in \eqref{equazione cal F vorticita} at any quasi-periodic solution $v_e(\vphi,x)|_{\vphi=\omega t}$ provided by Theorem \ref{Eulero1} by requiring a smallness condition on $\e$ which is independent of the viscosity parameter $\nu > 0$. This is achieved in two steps. First, we fully reduce to a constant coefficient, diagonal operator the  linearized Euler operator $\mL_{e}$ in \eqref{operatore linearizzato}. This is done in Section \ref{sez riducibilita}, in the spirit of \cite{BaldiMontalto}, \cite{BLM}, by combining a reduction to constant coefficients up to an arbitrarily regularizing remainder with a KAM reducibility scheme. We check step by step that this normal form procedure, when applied to the full operator $\mL_{\nu}$ in \eqref{operatore linearizzato}, just perturbs the unbounded viscous term $-\nu \Delta$ by an unbounded pseudo differential operator of order two that "gain smallness", namely it is of size $O(\nu \e)$, see \eqref{forma cal L infty viscosity}-\eqref{stima R infty n}. In Section \ref{sez inversioni}, we use this normal form procedure in order to infer the invertibility of the operator $\mL_{\nu}$ uniformly with respect to the viscosity parameter, namely by imposing a smallness condition on $\e$ that is independent of $\nu$. The inverse of the Navier Stokes operator is bounded from $H^s_0$,  {\it gains two space derivatives} and it has size $O(\nu^{- 1})$ (see Proposition \ref{inversione linearized}), whereas the inverse of the linearized Euler operator ${\cal L}_e$ loses $\tau$ derivatives, due to the small divisors (see Proposition \ref{inversione linearized no viscosity}). The invertibility of the linearized Euler operator ${\cal L}_e$ is used to construct the approximate solution in Section \ref{sez soluzione approx} and the invertibility of the linearized Navier Stokes operator ${\cal L}_\nu$ is used to implement the fixed point argument of Section \ref{sez fixed point}.

\section{Normal form reduction of the operator $\mathcal{L}_{\nu}$}\label{sez riducibilita}

In this section we  reduce to a constant coefficients, diagonal operator the operator $\mL_{\nu}$ in \eqref{operatore linearizzato} up to an unbounded remainder of order two which is of size $O(\e \nu)$. First, we deal with the conjugation of the transport operator in Section \ref{sezione riduzione ordine alto}, which is the highest order term in the operator $\mL_{e}$. In Section \ref{sez riduzione ordini bassi}, the lower order terms after the previous conjugation are regularized to constant coefficients up to a remainder of arbitrary smoothing matrix decay and up to an unbounded remainder of order two and size $O(\e \nu)$. Then, in Sections \ref{sezione riducibilita a blocchi}-\ref{sez convergenza KAM} we perform the full KAM reducibility for the regularized version of the operator $\mL_{e}$. In particular, in Section \ref{sezione riducibilita a blocchi} the $n$-th iterative step of the reduction is performed and in Section \ref{sez convergenza KAM} the convergence of the scheme is proved via Nash-Moser estimates to overcome the loss of derivatives coming from the small divisors.  The linearized Navier Stokes operator is then reduced to a diagonal operator plus an unbounded operator of order two and size $O(\e \nu)$ in \eqref{forma cal L infty viscosity}-\eqref{stima R infty n}. This is the starting point for its inversion in Section \ref{sez inversioni}

From now on, the parameters $\gamma\in(0,1)$ and $\tau>0$, characterizing the set $\Lambda(\gamma,\tau)$ in \eqref{set DC trasporto a} of the Diophantine frequencies in a given measurable set $\Lambda$, are considered as fixed and $\tau$ is chosen in  \eqref{definizione finale tau} and $\gamma$  at the end of Section \ref{sezione teoremi principali}, see \eqref{scelta gamma}. Therefore we omit to recall them each time. Moreover, from now on, we denote by $DC(\gamma, \tau)$, the set  of Diophantine frequencies in $\Omega_\e$, where the set $\Omega_\e$ is provided in Theorem \ref{main theorem 2}, namely $\Lambda(\gamma, \tau)$ with $\Lambda = \Omega_\e$. We repeat the definition for clarity of the reader: 
\begin{equation}\label{set DC trasporto}
DC(\gamma, \tau) := \big\{  (\omega,\zeta)\in \Omega_\e \,:\, |\omega\cdot \ell+\zeta\cdot j| \geq \frac{\gamma}{|(\ell,j)|^{\tau}}\,, \ \forall\,(\ell,j)\in\Z^{d+2}\setminus\{0\}  \big\}\,. 
\end{equation}

\subsection{Reduction of the highest order term}\label{sezione riduzione ordine alto}

First, we state the proposition that allows to reduce to constant coefficients the highest order operator 
\begin{equation*}\label{def operatore cal T}
{\mathcal T} := \omega \cdot \partial_\vphi 
+ \big( \zeta + \e a(\vphi, x) \big) \cdot \nabla 
\end{equation*}
where we recall, by \eqref{def a}, that $\Pi_0 a = 0 $ and $ {\rm div}(a) = 0 $.
The result has been proved in Proposition 4.1 in \cite{BaldiMontalto} (see also \cite{FGMP} for a more general result of this kind). We restate it with clear adaptation to our case and we refer to the former for the proof. 
\begin{proposition}\label{proposizione trasporto}
{\bf (Straightening of the transport operator $\mathcal T$).}
There exist $\sigma := \sigma (\tau, d) > 0$ large enough such that for any 
$S > s_0 + \sigma$ there exists $\delta := \delta(S, \tau, d) \in (0, 1)$ small enough such that if \eqref{ansatz} holds  and
\begin{equation}\label{condizione piccolezza rid trasporto}
 \e \gamma^{- 1} \leq \delta,
\end{equation} 
are fulfilled, then the following holds. 
There exists an invertible diffeomorphism 
$\T^2 \to \T^2$, $x \mapsto x + \alpha(\vphi, x; \omega, \zeta)$ 
with inverse $y \mapsto y + \breve \alpha(\vphi, y; \omega, \zeta)$,
defined for all $(\omega, \zeta) \in DC(\gamma, \tau)$, with the set given in \eqref{set DC trasporto},
 satisfying, for any $s_0\leq s \leq S - \sigma$,
\begin{equation}\label{stima alpha trasporto}
\| \alpha \|_s^{{\rm Lip}(\gamma)}, \| \breve \alpha\|_{s}^{{\rm Lip}(\gamma)} \lesssim_{s} \e \gamma^{- 1} 
\end{equation}
such that, by defining 
\begin{equation}\label{def mappa trasporto}
{\mathcal A} h (\vphi, x) := h(\vphi, x + \alpha(\vphi, x))\,, \quad \text{with } \quad {\mathcal A}^{- 1} h(\vphi, y) = h(\vphi, y + \breve \alpha(\vphi, y))\,,
\end{equation}
one gets the conjugation 
\begin{equation*}\label{coniugazione nel teo trasporto}
{\mathcal A}^{- 1} {\mathcal T} {\mathcal A} = \omega \cdot \partial_\vphi + \zeta \cdot \nabla
\end{equation*}
Furthermore, $\alpha,\breve{\alpha}$ are $\odd(\vphi,x)$ and the maps ${\mathcal A}, {\mathcal A}^{- 1}$ are reversibility preserving, satisfying the estimates 
\begin{equation}\label{stima tame cambio variabile rid trasporto}
\begin{aligned}
& \| {\mathcal A}^{\pm 1}  h\|_s^{{\rm Lip}(\gamma)} \lesssim_{s} \| h \|_s^{{\rm Lip}(\gamma)}\,. 
\end{aligned}
\end{equation}
\end{proposition} 

We remark that the assumptions of Proposition \ref{proposizione trasporto} are satisfied by the ansatz \eqref{ansatz} and by the choice of the parameter $\gamma>0$  at the end of Section \ref{sezione teoremi principali} in \eqref{scelta gamma}. In particular, the smallness condition \eqref{condizione piccolezza rid trasporto} becomes $\e \gamma^{-1}= \e^{1-\frac{\mathtt a}{2}} \ll 1$, which is clearly satisfied since $\mathtt a \in (0,1)$ and for $\e$ sufficiently small.

In order to study the conjugation of the operator ${\mathcal L}_{\nu} : H^{s+2}_0 \to H^{s}_0$ in \eqref{operatore linearizzato} under the transformation $\mA$, we need the following auxiliary Lemma. 
\begin{lemma}\label{coniugazione inverso laplaciano cambio variabile}
Let $S > s_0 + \sigma + 2$ (where $\sigma$ is the constant appearing in Proposition \ref{proposizione trasporto}). Then there exists $\delta := \delta(S, \tau, d) \in (0, 1)$ small enough such that if \eqref{ansatz}, \eqref{condizione piccolezza rid trasporto} are fulfilled, the following hold: 
\\[1mm]
\noindent
$(i)$ Let ${\mathcal A}_\bot := \Pi_0^\bot {\mathcal A} \Pi_0^\bot$. Then, for any $s_0\leq s\leq S - \sigma$, the operator ${\mathcal A}_\bot : H^s_0 \to H^s_0$ is invertible with bounded inverse given by ${\mathcal A}_\bot^{- 1} = \Pi_0^\bot {\mathcal A}^{- 1} \Pi_0^\bot : H^s_0 \to H^s_0$;
\\[1mm]
\noindent
$(ii)$ Let $s_0\leq s\leq S - \sigma - 1$, $a (\cdot; \lambda) \in H^{s + 1}(\T^{d + 2})$ and let ${\mathcal R}_a$ be the linear operator defined by 
$$
{\mathcal R}_a : h(\vphi, x) \mapsto \nabla a(\vphi, x) \cdot \nabla^\bot h(\vphi, x).
$$ Then ${\mathcal A}_\bot^{- 1} {\mathcal R}_a {\mathcal A}_\bot \in {\mathcal O}{\mathcal P}{\mathcal M}^1_s$ and $|{\mathcal A}_\bot^{- 1} {\mathcal R}_a {\mathcal A}_\bot |_{1, s}^{{\rm Lip}(\gamma)} \lesssim_s \| a \|_{s + 1}^{{\rm Lip}(\gamma)}$;
\\[1mm]
\noindent
$(iii)$ For any $s_0\leq s\leq S- \sigma - 2$, the operator ${\mathcal P}_\Delta := {\mathcal A}_\bot^{-1} (- \Delta) {\mathcal A}_\bot =- \Delta +\mR_{\Delta} : H^{s+2}_0 \to H_{0}^{s}$ is in  ${\mathcal{O}{\mathcal P}{\mathcal M}}^2_s$, with estimates 
\begin{equation}\label{stima R Delta}
 |\mR_{\Delta}|_{2,s}^{{{\rm Lip}(\gamma)}}   \lesssim_s \e \gamma^{- 1}\,. 
\end{equation}
\noindent
$(iv)$ For any $s_0\leq s \leq S- \sigma - 2$, the operator ${\mathcal P}_\Delta$ is invertible, with  inverse of the form 
\[
{\mathcal P}_\Delta^{- 1} 
= {\mathcal A}_\bot^{- 1} (- \Delta)^{- 1}{\mathcal A}_\bot \in {\mathcal O}{\mathcal P}{\mathcal M}_s^{- 2}
\]
satisfying the estimates
$$
|{\mathcal P}_\Delta^{- 1}|_{- 2, s}^{{\rm Lip}(\gamma)}  \lesssim_{s} 1\,.
$$
\end{lemma}

\begin{proof}
{\sc Proof of $(i)$.} For any $u \in H^s(\T^{d + 2})$, 
we split $u = \Pi_0^\bot u + \Pi_0 u$, 
where $\Pi_0^\bot u \in H^s_0 = H^s_0(\T^{d + 2})$ 
and $\Pi_0 u \in H^s_\vphi := H^s(\T^d)$.
 Since ${\mathcal A}$ is an operator of the form \eqref{def mappa trasporto},
one has that ${\mathcal A} h = h$ if $h(\vphi)$ does not depend on $x$. This implies that ${\mathcal A} \Pi_0 = \Pi_0$. Similarly, one can show that ${\mathcal A}^{- 1} \Pi_0 = \Pi_0$ and therefore
\begin{equation}\label{prop bella Pi0 bot A Pi0}
\Pi_0^\bot {\mathcal A} \Pi_0 = 0 \quad \text{and} \quad \Pi_0^\bot {\mathcal A}^{- 1} \Pi_0 = 0 \,. 
\end{equation}
We now show that the operator $\Pi_0^\bot {\mathcal A}^{- 1} \Pi_0^\bot : H^s_0 \to H^s_0$ is the inverse of the operator $\mA_\perp :=\Pi_0^\bot {\mathcal A} \Pi_0^\bot : H^s_0 \to H^s_0$. By \eqref{prop bella Pi0 bot A Pi0}, one has 
$$
\begin{aligned}
(\Pi_0^\bot {\mathcal A}^{- 1} \Pi_0^\bot)(\Pi_0^\bot {\mathcal A} \Pi_0^\bot) & = \Pi_0^\bot {\mathcal A}^{- 1} \big( {\rm Id} - \Pi_0 \big) {\mathcal A} \Pi_0^\bot   \\
& =  \Pi_0^\bot {\mathcal A}^{- 1} {\mathcal A} \Pi_0^\bot - \Pi_0^\bot {\mathcal A}^{- 1} \Pi_0 {\mathcal A} \Pi_0^\bot = \Pi_0^\bot
\end{aligned}
$$
and similarly $(\Pi_0^\bot {\mathcal A} \Pi_0^\bot) (\Pi_0^\bot {\mathcal A}^{- 1} \Pi_0^\bot)=\Pi_0^\perp$. The claimed statement then follows. 

\noindent
{\sc Proof of $(ii)$.} First, we note that, given a function $h(\vphi, x)$ and integrating by parts
$$
\begin{aligned}
\Pi_0[ {\mathcal R}_a h] & = \frac{1}{(2 \pi)^2} \int_{\T^2} \nabla a(\vphi, x) \cdot \nabla^\bot h(\vphi, x)\, d x \\
&  = - \frac{1}{(2 \pi)^2} \int_{\T^2}  a(\vphi, x) {\rm div}\big( \nabla^\bot h(\vphi, x) \big)\, d x = 0
\end{aligned}
$$
since ${\rm div}(\nabla^\bot h) = 0$ for any function $h$. Moreover it is easy to see that ${\mathcal R}_a \Pi_0 = 0$. This implies that the linear operator ${\mathcal R}_a$ is invariant on the space of zero average functions. Then, using also item $(i)$, one has that 
\begin{equation}\label{A inv M nabla a A c}
\begin{aligned}
	{\mathcal A}_\bot^{- 1} {\mathcal R}_a {\mathcal A}_\bot & = \Pi_0^\bot {\mathcal A}^{- 1} {\mathcal R}_a {\mathcal A} \Pi_0^\bot = \Pi_0^\bot ({\mathcal A}^{- 1} {\mathcal M}_{\nabla a} {\mathcal A}  ) \cdot ({\mathcal A}^{- 1} \nabla^\bot {\mathcal A}) \Pi_0^\bot\,,
\end{aligned}
\end{equation}
where ${\mathcal M}_{\nabla a}$ denotes the multiplication operator by $\nabla a$. A direct calculation shows that the operator ${\mathcal A}^{- 1}{\mathcal M}_{\nabla a} {\mathcal A} = {\mathcal M}_g$, where the function $g (\vphi, y) := \nabla a (\vphi, y + \breve \alpha(\vphi, y)) = \{{\mathcal A}^{- 1} \nabla a\}(\vphi, y)$. The estimates \eqref{prop multiplication decay}, \eqref{stima tame cambio variabile rid trasporto} imply that 
\begin{equation}\label{A inv M nabla a A}
|{\mathcal A}^{- 1} {\mathcal M}_{\nabla a} {\mathcal A}|_{0, s} = |{\mathcal M}_g|_{0, s} \lesssim_s \| a \|_{s + 1}\,.
\end{equation} 
Moreover, one computes  ${\mathcal A}^{- 1} \partial_{x_i} {\mathcal A} \,[h]= \partial_{y_i} h + {\mathcal A}^{- 1}[\partial_{x_i} \alpha] \cdot \nabla h $, for $i=1,2$.
Using that $|\partial_{x_1}|_{1, s}, |\partial_{x_2}|_{1, s} \leq 1$ and by applying the estimates \eqref{prop multiplication decay}, \eqref{stima alpha trasporto}, \eqref{stima tame cambio variabile rid trasporto}, \eqref{condizione piccolezza rid trasporto}, one obtains that 
\begin{equation}\label{A inv M nabla a A b}
|{\mathcal A}^{- 1} \nabla^\bot {\mathcal A}|_{1, s} \lesssim_s 1 + \| \alpha \|_{s + 1} \lesssim_s 1\,. 
\end{equation}
Using the trivial fact that $|\Pi_0^\bot|_{0, s} \leq 1$, the formula \eqref{A inv M nabla a A c}, the estimates \eqref{A inv M nabla a A}, \eqref{A inv M nabla a A b}, together with the composition Lemma \ref{proprieta standard norma decay}-$(ii)$, imply the claimed bound. 

\noindent
{\sc Proof of $(iii)$.} Since $\Pi_0 \Delta = \Delta \Pi_0 = 0$, one computes 
$$
{\mathcal P}_\Delta =  {\mathcal A}_\bot^{- 1} (- \Delta) {\mathcal A}_\bot = \Pi_0^\bot {\mathcal A}^{- 1} (- \Delta) {\mathcal A} \Pi_0^\bot\,. 
$$
and hence, a direct calculation shows that 
\begin{equation}\label{RDelta}
\begin{aligned}
& {\mathcal A}_\bot^{- 1} (- \Delta) {\mathcal A}_\bot = - \Delta + {\mathcal R}_\Delta,  \\
& {\mathcal R}_\Delta := \Pi_0^\bot \Big(  \sum_{i, j = 1}^2 \mA^{-1} [a_{i j}]\, \partial_{y_i y_j} + \sum_{k = 1}^2 \mA^{-1} [b_k ] \, \partial_{y_k} \Big)\Pi_0^\bot\,,
\end{aligned}
\end{equation}
where
\begin{equation*}
	\begin{footnotesize}
		\begin{aligned}
			&a_{11}(\vphi,x):=-\alpha_{x_1}(\vphi,x)(2+\alpha_{x_1}(\vphi,x))\,,\quad a_{22}(\vphi,x):=-\alpha_{x_2}(\vphi,x)(2+\alpha_{x_2}(\vphi,x))\,, \\
			& a_{12}(\vphi,x)=a_{21}(\vphi,x)=-\alpha_{x_1}(\vphi,x)\alpha_{x_2}(\vphi,x) -\tfrac12(\alpha_{x_1}(\vphi,x)+\alpha_{x_2}(\vphi,x))\,, \\
			& b_{1}(\vphi,x):=-\alpha_{x_1 x_1}(\vphi,x)-\alpha_{x_1 x_2}(\vphi,x)\,, \quad b_{2}(\vphi,x):=-\alpha_{x_1 x_2}(\vphi,x)-\alpha_{x_2 x_2}(\vphi,x)\,.
		\end{aligned}
	\end{footnotesize}
\end{equation*}
By Lemma \ref{lemma:LS norms} and the estimate \eqref{stima alpha trasporto}, we have, for any $s_0\leq s \leq S- \sigma - 2$
\begin{equation}\label{stime a ij bk Delta}
\| \mA^{-1} [a_{i j}] \|_s^{{\rm Lip}(\gamma)}\,,\,  \| \mA^{-1}[b_k] \|_s^{{\rm Lip}(\gamma)} \lesssim_s \| \alpha \|_{s + 2}^{{\rm Lip}(\gamma)} + (\| \alpha \|_{s + 2}^{{\rm Lip}(\gamma)})^2 \lesssim_s \e \gamma^{- 1}
\end{equation}
By \eqref{RDelta}, Lemma \ref{lemma:LS norms}, estimates \eqref{prop multiplication decay}, \eqref{stima Pi 0}, \eqref{stime a ij bk Delta}, Lemma  \ref{proprieta standard norma decay}-$(ii)$ together with the trivial fact that $|\partial_{x_j}|_{1, s}\,,\, |\partial_{x_i x_j}|_{2, s} \lesssim 1$, we conclude that $\mR_{\Delta}$ satisfies the claimed bound.  

\medskip

\noindent
{\sc Proof of $(iv)$.} We write 
$$
{\mathcal P}_\Delta = - \Delta  + {\mathcal R}_\Delta = (- \Delta) \big( {\rm Id}_0 + (- \Delta)^{- 1} {\mathcal R}_\Delta \big)
$$
where ${\rm Id}_0$ is the identity on the space of the $L^2$ zero average functions. By  Lemma \ref{proprieta standard norma decay}-$(ii)$, estimates \eqref{decay laplace}, \eqref{stima R Delta}, one obtains that $|(- \Delta)^{- 1} {\mathcal R}_\Delta|_{0, s}^{{\rm Lip}(\gamma)} \lesssim_s \e \gamma^{- 1}$. Hence, by the smallness condition in \eqref{condizione piccolezza rid trasporto} and by Lemma \ref{proprieta standard norma decay}-$(iv)$, one gets that $ {\rm Id}_0 + (- \Delta)^{- 1} {\mathcal R}_\Delta : H^s_0 \to H^s_0 $ is invertible, with $|\big( {\rm Id}_0 + (- \Delta)^{- 1} {\mathcal R}_\Delta \big)^{- 1}|_{0, s}^{{\rm Lip}(\gamma)} \lesssim_s 1 $. The claimed statement then follows since we have
$$
{\mathcal P}_\Delta^{- 1} = \big( {\rm Id}_0 + (- \Delta)^{- 1} {\mathcal R}_\Delta \big)^{- 1}(- \Delta)^{- 1}
$$
using again Lemma \ref{proprieta standard norma decay}-$(ii)$. 
\end{proof}

We conclude this section by conjugating the whole operator ${\mathcal L}$ defined in \eqref{operatore linearizzato} by means of the map ${\mathcal A}$ constructed in Proposition \ref{proposizione trasporto}. 
\begin{proposition}\label{lemma coniugio cal L (0)}
Let  $ S > s_0 + \sigma + 2$ (where $\sigma$ is the constant appearing in Proposition \ref{proposizione trasporto}). Then there exists $\delta := \delta(S, \tau, d) \in (0, 1)$ small enough such that, if \eqref{ansatz} and \eqref{condizione piccolezza rid trasporto} are fulfilled, the following holds.  
For any $(\omega, \zeta ) \in DC(\gamma, \tau)$ defined in \eqref{set DC trasporto}, one has
\begin{equation}\label{cal L (1)}
\begin{aligned}
{\mathcal L}^{(1)}_\nu & := {\mathcal A}_\bot^{- 1} {\mathcal L} {\mathcal A}_\bot = {\mathcal L}^{(1)}_e - \nu \Delta + {\mathcal R}_\nu^{(1)} \,, \\
{\mathcal L}^{(1)}_e &:=  {\mathcal A}^{- 1}_\bot {\mathcal L}_e {\mathcal A}_\bot =  \omega \cdot \partial_\vphi + \zeta \cdot \nabla +{\mathcal R}^{(1)}  
\end{aligned}
\end{equation}
where the map $\mA_\perp$ is defined as in Proposition \ref{proposizione trasporto} and Lemma \ref{coniugazione inverso laplaciano cambio variabile}-$(i)$, whereas, for any $s_0 \leq s \leq S - \sigma - 2$, the operators ${\mathcal R}^{(1)} \in {\mathcal O}{\mathcal P}{\mathcal M}^{- 1}_{s}$ and ${\mathcal R}^{(1)}_\nu  \in {\mathcal O}{\mathcal P}{\mathcal M}^2_s$ satisfy the estimates
\begin{equation}\label{stime cal R (1)}
\begin{aligned}
& |{\mathcal R}^{(1)}|_{- 1, s}^{{\rm Lip}(\gamma)} \lesssim_{s} \e \,, \qquad  |{\mathcal R}_\nu^{(1)}|_{2, s}^{{\rm Lip}(\gamma)} \lesssim_s  \e \gamma^{- 1} \,\nu \,.
\end{aligned}
\end{equation}
Moreover, the operators ${\mathcal L}^{(1)}_e$ and ${\mathcal R}^{(1)}$ are real and reversible and ${\mathcal R}^{(1)}, {\mathcal R}_\nu^{(1)}$ leave invariant the space of functions with zero average in $x$. 
\end{proposition}

\begin{proof}
By recalling \eqref{operatore linearizzato} and by Lemma \ref{coniugazione inverso laplaciano cambio variabile}-$(i)$, one gets ${\mathcal L}^{(1)}_\nu  = {\mathcal L}_e^{(1)} - \nu {\mathcal A}_\bot^{- 1} \Delta {\mathcal A}_\bot$, where, by \eqref{invarianze cal L},
\begin{equation*}\label{cal L bla bla}
\begin{aligned} 
{\mathcal L}_e^{(1)} & :=  {\mathcal A}_\bot^{- 1} {\mathcal L}_e {\mathcal A}_\bot  = {\mathcal A}_\bot^{- 1} \big( \omega \cdot \partial_\vphi + (\zeta + \e a(\vphi, x)) \cdot \nabla \big) {\mathcal A}_\bot + \e {\mathcal A}_\bot^{- 1} {\mathcal R} {\mathcal A}_\bot \\
& :=  \Pi_0^\bot {\mathcal A}^{- 1}\Pi_0^\bot \big( \omega \cdot \partial_\vphi + (\zeta + \e a(\vphi, x)) \cdot \nabla \big) \Pi_0^\bot  {\mathcal A} \Pi_0^\bot + \e {\mathcal A}_\bot^{- 1} {\mathcal R} {\mathcal A}_\bot \\ 
& =  \Pi_0^\bot {\mathcal A}^{- 1} \big( \omega \cdot \partial_\vphi + (\zeta + \e a(\vphi, x)) \cdot \nabla \big)  {\mathcal A} \Pi_0^\bot + \e {\mathcal A}_\bot^{- 1} {\mathcal R} {\mathcal A}_\bot\,. 
\end{aligned}
\end{equation*}
By Proposition \ref{proposizione trasporto}, using the formula \eqref{operatore linearizzato},  one has that 
\begin{equation*}\label{cal L1 e bla}
{\mathcal L}^{(1)}_e  = \omega \cdot \partial_\vphi + \zeta \cdot \nabla + {\mathcal R}^{(1)}, \quad {\mathcal R}^{(1)} :=  \e  {\mathcal A}_\bot^{- 1} {\mathcal R} {\mathcal A}_\bot\,. 
\end{equation*}

 By \eqref{definizione cal R}, one writes,
 according to the notation of Lemma \ref{coniugazione inverso laplaciano cambio variabile}-$(ii)$, 
 $$
 {\mathcal R} = {\mathcal R}_{v_e} \circ (- \Delta)^{- 1}\,, \quad {\mathcal R}_{v_e} : h(\vphi, x) \mapsto \nabla v_e \cdot \nabla^\bot h(\vphi, x)\,.
 $$
 Therefore, 
$$
\begin{aligned}
{\mathcal R}^{(1)} & = \e {\mathcal A}_\bot^{- 1} {\mathcal R} {\mathcal A}_\bot = \e \big( {\mathcal A}_\bot^{- 1} {\mathcal R}_{v_e} {\mathcal A}_\bot \big)  \circ  \big({\mathcal A}_\bot^{- 1} (- \Delta)^{- 1} {\mathcal A}_\bot\big)\,. 
\end{aligned}
$$
Then, by Lemma \ref{coniugazione inverso laplaciano cambio variabile}, $(ii)$, $(iv)$
Lemma \ref{proprieta standard norma decay}-$(ii)$, the ansatz \eqref{ansatz} and $\e \gamma^{- 1} \leq \delta< 1$, one gets the claimed bound \eqref{stime cal R (1)} for ${\mathcal R}^{(1)}$.  By Lemma \ref{coniugazione inverso laplaciano cambio variabile}-$(iii)$, one has $- \nu {\mathcal A}_\bot^{- 1} \Delta {\mathcal A}_\bot = - \nu \Delta + {\mathcal R}_\nu^{(1)}$, with ${\mathcal R}_\nu^{(1)} := \nu \Pi_0^{\bot} \big( \Delta - {\mathcal A}^{- 1} \Delta {\mathcal A} \big) \Pi_0^\bot $ satisfying the estimate $ |{\mathcal R}_\nu^{(1)}|_s^{{\rm Lip}(\gamma)} \lesssim_s \e \gamma^{- 1} \nu $,
which is the second estimate in \eqref{stime cal R (1)}. Moreover, since ${\mathcal A}, {\mathcal A}^{- 1}$ are real and reversibility preserving and ${\mathcal R}$ is real and reversible, then ${\mathcal L}_e^{(1)}, {\mathcal R}^{(1)}$ are real and reversible. This concludes the proof.
\end{proof}

\subsection{Reduction to constant coefficients of the lower order terms}\label{sez riduzione ordini bassi}
In this section we diagonalize the operator ${\mathcal L}^{(1)}_\nu$ in \eqref{cal L (1)} up to a remainder of size $O(\e)$ and arbitrarily smoothing matrix decay and up to an unbounded operator of order $2$ of size $O(\e \gamma^{- 1} \nu)$. More precisely, we prove the following Proposition. 

\begin{proposition} \label{proposizione regolarizzazione ordini bassi}
Let $M\in\N$ be fixed. There exists $\sigma_{M - 1} := \sigma_{M - 1}( \tau, d) > \sigma + 2$ large enough (where $\sigma$ is the constant appearing in Proposition \ref{proposizione trasporto}) such that for any $S > s_0 + \sigma_{M - 1}$ there exists $\delta := \delta(S, M, \tau, d) \in (0, 1)$ small enough such that, if \eqref{ansatz}, \eqref{condizione piccolezza rid trasporto} are fulfilled, the following holds. For any $(\omega, \zeta) \in DC(\gamma, \tau)$ defined in \eqref{set DC trasporto}, there exists a real and reversibility preserving, invertible map ${\mathcal B}$ satisfying, for any $s_0\leq s\leq S-\sigma_{M - 1}$,
\begin{equation}\label{stima Phi M reg ordini bassi}
{\mathcal B}^{\pm 1} : H^s_0 \to H^s_0, \quad |{\mathcal B}^{\pm 1}- {\rm Id} |_{0, s}^{{\rm Lip}(\gamma)}  \lesssim_{s,M}\e \gamma^{- 1} 
\end{equation}
such that
\begin{equation}\label{def cal L (3)}
\begin{aligned}
{\mathcal L}^{(2)}_\nu &:= {\mathcal B}^{- 1}{\mathcal L}^{(1)}_\nu {\mathcal B} = {\mathcal L}_e^{(2)} - \nu \Delta + {\mathcal R}_\nu^{(2)}\,, \\
 {\mathcal L}_e^{(2)} & :=  {\mathcal B}^{- 1}{\mathcal L}^{(1)}_e {\mathcal B} =  \omega \cdot \partial_\vphi + \zeta \cdot \nabla + {\mathcal Q} + {\mathcal R}^{(2)} 
\end{aligned}
\end{equation}
where ${\mathcal Q}  \in {\mathcal O}{\mathcal P}{\mathcal M}^{- 1}_{s}$ is a diagonal operator, 
${\mathcal R}^{(2)}$ belongs to ${\mathcal O}{\mathcal P}{\mathcal M}^{- M}_{s}$ and ${\mathcal R}^{(2)}_\nu \in {\mathcal O}{\mathcal P}{\mathcal M}^2_{s}$. Moreover, for any $s_0\leq s \leq S-\sigma_{M - 1}$,
\begin{equation}\label{stima cal Z cal R (3)}
\begin{aligned}
& |{\mathcal Q}|_{- 1, s}^{{\rm Lip}(\gamma)} \, , \, 
|{\mathcal R}^{(2)}|_{- M, s}^{{\rm Lip}(\gamma)} 
\lesssim_{s, M} \e, 
\quad |{\mathcal R}^{(2)}_\nu|_{2, s}^{{\rm Lip}(\gamma)} \lesssim_{s,M}  \e \gamma^{- 1} \,\nu\,.  
\end{aligned}
\end{equation}
The operators ${\mathcal Q}$, ${\mathcal R}^{(2)}$ are real, reversible and ${\mathcal Q}$, ${\mathcal R}^{(2)}, {\mathcal R}^{(2)}_\nu$ leave invariant the space of functions with zero average in $x$.  
\end{proposition}

Proposition \ref{proposizione regolarizzazione ordini bassi} 
follows by the following iterative lemma on the operator obtained  by neglecting the viscosity term $- \nu \Delta + {\mathcal R}^{(2)}_\nu$ in \eqref{cal L (1)}, namely, in the next proposition, we only consider 
\begin{equation}\label{def cal L (1) 0}
{\mathcal L}^{(1)}_0 \equiv {\mathcal L}^{(1)}_e := \omega \cdot \partial_\vphi + \zeta \cdot \nabla + {\mathcal R}^{(1)}\,. 
\end{equation} 

\begin{lemma}\label{lemma iterativo ordini bassi}
Let $M \in \N$. There exist $0 < \sigma_0 < \sigma_1 < \ldots < \sigma_{M - 1}$ large enough such that for any $S > s_0 + \sigma_{M - 1}$ there exists $\delta := \delta(S, \tau, d) \in (0, 1)$ small enough such that, if \eqref{ansatz}, \eqref{condizione piccolezza rid trasporto} are fulfilled, the following holds.  For any $n = 0, \ldots, M -1 $ and any $(\omega, \zeta) \in DC(\gamma, \tau)$,  there exists a real, reversibility preserving, invertible map ${\mathcal T}_n$ satisfying, for any $s_0 \leq s \leq S-\sigma_{n}$
\begin{equation}\label{stima cal Tn reg ordini bassi}
{\mathcal T}_n^{\pm 1} : H^s_0 \to H^s_0, \quad |{\mathcal T}_n^{\pm 1} - {\rm Id}|_{- n,s}^{{\rm Lip}(\gamma)} \lesssim_{s,n} \e \gamma^{- 1} 
\end{equation}
and, for any $n = 1, \ldots, M - 1$ and for any  $(\omega, \zeta) \in DC(\gamma, \tau)$,
\begin{equation}\label{coniugazione lemma iterativo ordini bassi}
{\mathcal L}^{(1)}_n = {\mathcal T}_n^{- 1}{\mathcal L}_{n - 1}^{(1)} {\mathcal T}_n 
\,,
\end{equation}
where ${\mathcal L}^{(1)}_n$ has the form 
\begin{equation}\label{def cal L (2) n}
{\mathcal L}^{(1)}_n := \omega \cdot \partial_\vphi + \zeta \cdot \nabla + {\mathcal Z}_n + {\mathcal R}_n^{(1)}
\end{equation}
where ${\mathcal Z}_n \in {\mathcal O}{\mathcal P}{\mathcal M}^{- 1}_{s}$  is diagonal,
 ${\mathcal R}_n^{(1)} \in {\mathcal O}{\mathcal P}{\mathcal M}^{- (n + 1) }_{s}$ and they
 satisfy
\begin{equation}\label{stima induttiva cal Zn cal Rn (2)}
\begin{aligned}
& |{\mathcal Z}_n|_{- 1, s}^{{\rm Lip}(\gamma)} \lesssim_{s,n} \e \quad |{\mathcal R}_n^{(1)}|_{- (n + 1), s}^{{\rm Lip}(\gamma)} \lesssim_{s,n} \e \quad \ \forall s_0 \leq s \leq S-\sigma_{n}\,. 
\end{aligned}
\end{equation} 
The operators ${\mathcal L}_n^{(1)}, {\mathcal Z}_n$ and $ {\mathcal R}_n^{(1)}$ are real, reversible and leave invariant the space of zero average functions in $x$. 
%
\end{lemma}

\begin{proof}
We prove the lemma arguing by induction. For $n = 0$ the desired properties follow by Proposition \ref{lemma coniugio cal L (0)}, by defining ${\mathcal L}^{(1)}_0$ as in \eqref{def cal L (1) 0}, ${\mathcal Z}_0 = 0$, ${\mathcal R}^{(1)}_0 := {\mathcal R}^{(1)}$ and $\sigma_0 = \sigma$ given in Proposition \ref{lemma coniugio cal L (0)}. 

We assume that the claimed statement holds for some $n \in \{0, \ldots, M - 2\}$ 
and we prove it at the step $n+1$. Let us consider a transformation ${\mathcal T}_{n+1} = {\rm Id} +{\mathcal K}_{n+1}$ where ${\mathcal K}_{n+1} $ is an operator of order $- (n+1)$ which has to be determined. One computes 
\begin{equation}\label{prima formula cal L n (2)}
\begin{aligned}
{\mathcal L}_n^{(1)} {\mathcal T}_{n + 1} 
& = {\mathcal T}_{n + 1} \big( \omega \cdot \partial_\vphi + \zeta \cdot \nabla \big) + {\mathcal Z}_n 
+ (\omega \cdot \partial_\vphi {\mathcal K}_{n + 1}) + [\zeta \cdot \nabla, {\mathcal K}_{n + 1}] + {\mathcal R}_n^{(1)}
 \\
 & \qquad + {\mathcal Z}_n {\mathcal K}_{n + 1} + {\mathcal R}_n^{(1)}{\mathcal K}_{n + 1}\,. 
\end{aligned}
\end{equation}
By the induction hypothesis ${\mathcal R}_n^{(1)} \in {\mathcal O}{\mathcal P}{\mathcal M}_{s}^{- (n + 1)}$, with $|{\mathcal R}_n^{(1)} |_{- (n + 1), s} \lesssim_{s,n+1} \e$ for any $s_0\leq s \leq S- \sigma_{n}$. Therefore, by Lemma \ref{lemma eq omologica descent}, there exists ${\mathcal K}_{n + 1} \in {\mathcal O}{\mathcal P}{\mathcal M}_{s}^{- (n + 1)}$ solving the homological equation
\begin{equation}\label{M n + 1 R n + 1 (1)}
\begin{aligned}
& \omega \cdot \partial_\vphi \,{{\mathcal K}}_{n + 1} + [\zeta \cdot \nabla, {{\mathcal K}}_{n + 1}] + {\mathcal R}_n^{(1)} = {\mathcal D}_{{\mathcal R}_n^{(1)}}\,, 
\end{aligned}
\end{equation} 
with ${\mathcal D}_{{\mathcal R}_n^{(1)}}$ as in Definition \ref{def block-diagonal op}, satisfying, for any $s_0\leq s \leq S-\sigma_{n + 1}$ (for an arbitrary $\sigma_{n + 1} > \sigma_n + 2 \tau + 1$),
\begin{equation}\label{M n + 1 R n + 1 (2)}
	 |{\mathcal K}_{n + 1}|_{- (n + 1), s}^{{\rm Lip}(\gamma)} \lesssim_{s,n+1} \gamma^{- 1}|{\mathcal R}_n^{(1)} |^{{\rm Lip}(\gamma)}_{- (n + 1), s + 2 \tau + 1} 
	  \lesssim_{s,n+1}  \e \gamma^{- 1}. 
\end{equation}
By Lemma \ref{proprieta standard norma decay}-$(iv)$, we obtain that ${\mathcal T}_{n + 1} = {\rm Id} + {\mathcal K}_{n + 1}$ is invertible with inverse ${\mathcal T}_{n + 1}^{- 1}$ satisfying the estimate, for $s_0\leq s\leq S-\sigma_{n + 1}$,
\begin{equation}\label{stima cal T n + 1 - 1 nel proof}
|{\mathcal T}_{n + 1}^{- 1} - {\rm Id}|_{- (n + 1), s}^{{\rm Lip}(\gamma)} 
\lesssim_{s,n+1} \e \gamma^{- 1}\,. 
\end{equation}
Hence the estimate \eqref{stima cal Tn reg ordini bassi} at the step $n + 1$ holds. 
By \eqref{prima formula cal L n (2)}, \eqref{M n + 1 R n + 1 (1)} we get, 
for any $(\omega, \zeta) \in DC(\gamma, \tau)$, 
the conjugation \eqref{coniugazione lemma iterativo ordini bassi} at the step $n + 1$ 
where ${\mathcal L}^{(1)}_{n + 1}$ has the form \eqref{def cal L (2) n}, with 
\begin{equation*}\label{cal Z cal R2 n + 1}
\begin{aligned}
{\mathcal Z}_{n + 1} & := {\mathcal Z}_n + {\mathcal D}_{{\mathcal R}_n^{(1)}}\,, \\
{\mathcal R}_{n + 1}^{(1)} & := ({\mathcal T}_{n + 1}^{- 1} - {\rm Id}){\mathcal Z}_{n + 1} 
+ {\mathcal T}_{n + 1}^{- 1} 
\big( {\mathcal Z}_n{{\mathcal K}}_{n + 1} + {\mathcal R}_n^{(1)}{{\mathcal K}}_{n + 1} \big)  \,. 
\end{aligned}
\end{equation*}
By Lemma \ref{proprieta standard norma decay}-$(v)$, since ${\mathcal R}_n^{(1)}$ and ${\mathcal Z}_n$ satisfy \eqref{stima induttiva cal Zn cal Rn (2)}, we deduce that ${\mathcal Z}_{n + 1} \in {\mathcal O}{\mathcal P}{\mathcal M}_{s}^{- 1}$, with estimates $|{\mathcal Z}_{n + 1}|_{- 1,s}^{{\rm Lip}(\gamma)} \lesssim_{s,n+1} \e$.   Moreover, by \eqref{M n + 1 R n + 1 (2)}, \eqref{stima cal T n + 1 - 1 nel proof}, Lemma \ref{proprieta standard norma decay}-$(ii)$ and the condition in \eqref{ansatz}, we obtain that ${\mathcal R}_{n + 1}^{(1)} \in {\mathcal OPM}_{s}^{- (n + 2)}$, with $|{\mathcal R}_{n + 1}^{(1)}|_{- (n + 2), s}^{{\rm Lip}(\gamma)} \lesssim_{s,n+1} \e$ for any $s_0\leq s\leq S-\sigma_{n + 1}$, by fixing $\sigma_{n+1} := \sigma_{n} + 2 \tau + 1 + n + 1$. This concludes the induction argument and therefore the claimed statement is proved.
\end{proof}

\begin{proof}[{\sc Proof of Proposition \ref{proposizione regolarizzazione ordini bassi}}] 
By \eqref{cal L (1)}, \eqref{def cal L (1) 0}, Lemma \ref{lemma iterativo ordini bassi} and by defining
$$
\begin{aligned}
{\mathcal B} & := {\mathcal T}_1 \circ \ldots \circ {\mathcal T}_{M - 1} \,, \quad 
{\mathcal Q}  := {\mathcal Z}_{M - 1}, \quad 
{\mathcal R}^{(2)} := {\mathcal R}_{M - 1}^{(1)}, 
\end{aligned}
$$
one obtains that 
\begin{equation*}\label{final conjugation cal L 0 (2)}
{\mathcal L}^{(2)}_e := {\mathcal B}^{- 1} {\mathcal L}^{(1)}_e {\mathcal B} = \omega \cdot \partial_\vphi + \zeta \cdot \nabla + {\mathcal Q} + {\mathcal R}^{(2)}
\end{equation*}
To deduce the claimed properties of ${\mathcal Q}$ and ${\mathcal R}^{(2)}$, 
it suffices to apply Lemma \ref{lemma iterativo ordini bassi} for $n = M - 1$.
 The estimate \eqref{stima Phi M reg ordini bassi} then follows by the estimate \eqref{stima cal Tn reg ordini bassi} on ${\mathcal T}_n$, $n = 1, \ldots, M - 1$, using the composition property stated in Lemma \ref{proprieta standard norma decay}-$(ii)$. We now study the conjugation ${\mathcal B}^{- 1}( - \nu \Delta + {\mathcal R}_\nu^{(1)}){\mathcal B}$. One has 
$$
\begin{aligned}
{\mathcal B}^{- 1} \big( - \nu \Delta + {\mathcal R}_\nu^{(1)} \big){\mathcal B} & = - \nu {\mathcal B}^{- 1} \Delta {\mathcal B} + {\mathcal B}^{- 1} {\mathcal R}_\nu^{(1)} {\mathcal B } = - \nu \Delta + {\mathcal R}_\nu^{(2)}\,,
\end{aligned}
$$
where 
$$
{\mathcal R}_\nu^{(2)} := - \nu \big( \Delta ({\mathcal B} - {\rm Id}) + ({\mathcal B}^{- 1} - {\rm Id}) \Delta {\mathcal B} \big) + {\mathcal B}^{- 1} {\mathcal R}_\nu^{(1)} {\mathcal B}\,.
$$
By applying again Lemma \ref{proprieta standard norma decay}-$(ii)$ and the estimates \eqref{stime cal R (1)}, \eqref{stima Phi M reg ordini bassi} (using also that $|\Delta|_{2, s} \leq 1$ for any $s$), one gets that ${\mathcal R}_\nu^{(2)}$ satisfies \eqref{stima cal Z cal R (3)} for any $s_0\leq s\leq S-\sigma_{M - 1}$. The proof of the claimed statement is then concluded. 
\end{proof}

\subsection{KAM reducibility} \label{sezione riducibilita a blocchi}
In this section we perform the KAM reducibility scheme for the operator obtained by neglecting the small viscosity term $- \nu \Delta + {\mathcal R}_{ \nu}^{(2)}$ from the operator ${\mathcal L}^{(2)}_\nu$ in \eqref{def cal L (3)}, namely we only consider the operator ${\mathcal L}^{(2)}_e$. More precisely we consider the operator 
\begin{equation}\label{op inizio riducibilita}
\begin{aligned}
& {\mathcal L}_0 := {\mathcal L}^{(2)}_e =  \omega \cdot \partial_\vphi + {\mathcal D}_0 + {\mathcal R}_0\,, \\
& {\mathcal D}_0 := \zeta \cdot \nabla + {\mathcal Q}\,, \quad {\mathcal R}_0 := {\mathcal R}^{(2)}\,. 
\end{aligned}
\end{equation}
where the diagonal operator $\mQ$ and the smoothing operator $\mR^{(2)}$ are as in Proposition \ref{proposizione regolarizzazione ordini bassi}.
Given $\tau, N_0 > 0$, we fix the constants 
\begin{equation}\label{definizione alpha beta}
\begin{aligned}
& 
M := [4 \tau] + 1,  \quad 
\alpha := (1+\chi^{-1})\tau_1+1 \,, \quad 
\beta:= \alpha + 1 \,, \\
& \tau_1:= 4\tau+2+M\,, \quad \Sigma(\beta):= \sigma_{M - 1} +\beta\,, \\
&   
N_{- 1} := 1, \quad 
N_n := N_0^{\chi^n}, \quad 
n \geq 0, \quad 
\chi := 3/2 
\end{aligned}
\end{equation}
where $[4\tau]$ is the integer part of $4\tau$ and $M$, $\sigma_{M - 1}$ are introduced in Proposition \ref{proposizione regolarizzazione ordini bassi}. 

By Proposition \ref{proposizione regolarizzazione ordini bassi}, replacing $s$ by $s + \beta$ in  \eqref{stima cal Z cal R (3)} and having ${\mathcal Q} = {\rm diag}_{j \in \Z^2 \setminus \{ 0 \}} q_0(j)$  diagonal, one gets the initialization conditions for the KAM reducibility, for any $s_0\leq s\leq S-\Sigma(\beta)$,
\begin{equation}\label{stime pre rid}
\sup_{j \in \Z^2 \setminus \{ 0 \}} |j| |q_0(j)|^{{\rm Lip}(\gamma)}, |{\mathcal R}_0|_{- M, s + \beta}^{{\rm Lip}(\gamma)} \leq C(S) \e\,.
\end{equation}

\begin{proposition}[\bf Reducibility]\label{prop riducibilita}
Let $S > s_0 + \Sigma(\beta)$. There exist $N_0 := N_0(S, \tau, d) > 0$ large enough and $\delta := \delta(S, \tau, d) \in (0, 1)$ small enough such that, if \eqref{ansatz} holds and
\begin{equation}\label{KAM smallness condition}
N_0^{\tau_1}\e \gamma^{-1} 
\leq \delta, 
\end{equation} 
then the following statements hold for any integer $n \geq 0$.

\smallskip
\noindent 
${\bf (S1)}_n$ There exists a real and reversible operator 
\begin{equation}\label{def.calLn calDn calQn}
	\begin{aligned}
		& {\mathcal L}_n := \omega \cdot \partial_\vphi + {\mathcal D}_n +   {\mathcal R}_n : H^{s + 1}_0 \to H^s_0, \\
		& {\mathcal D}_n := \zeta \cdot \nabla + {\mathcal Q}_n =  {\rm diag}_{j \in \Z^2 \setminus \{ 0 \}} \mu_n(j) \,, \\ 
		& {\mathcal Q}_n =  {\rm diag}_{j \in \Z^2 \setminus \{ 0 \}} q_n(j)\,, \quad \mu_n(j) :=  \ii\, \zeta \cdot j + q_n(j)\,,
	\end{aligned}
\end{equation}
defined for any $\lambda \in \Lambda_n^\gamma$, where we define $\Lambda_0^\gamma := DC(\gamma, \tau)$ for $n=0$ and, for $n \geq 1$, 
\begin{equation}\label{insiemi di cantor rid}
\begin{aligned}
\Lambda_n^\gamma &:= \Big\{ \lambda = (\omega, \zeta) \in \Lambda_{n - 1}^\gamma : |\ii \,\omega \cdot \ell + \mu_{n - 1}(j) - \mu_{n - 1}(j') | \geq \frac{\gamma}{\langle \ell \rangle^\tau |j|^\tau |j'|^\tau},  \\
& \forall \,\ell \in \Z^d, \  
j,j' \in \Z^2 \setminus \{ 0 \}, \ \ 
(\ell, j, j') \neq (0, j, j), \ \ 
|\ell|, |j-j'| \leq N_{n-1} \Big\}\,. 
\end{aligned}
\end{equation}
For any $j \in \Z^2 \setminus \{ 0 \}$, the eigenvalues $\mu_{n}(j)= \mu_{n}(j;\lambda)$ are purely imaginary and satisfy the conditions
\begin{equation}\label{reversibility reality auto}
\begin{aligned}
& \mu_n(j) = - \mu_n(- j) = \overline{\mu_n(- j)}, \ \ \text{or equivalently}   \\
& q_n(j) = - q_n(- j)= \overline{q_n(- j)}\,, 
\end{aligned}
\end{equation}
and the estimates
\begin{align}
	& |q_n(j)|^{{\rm Lip}(\gamma)} \lesssim \e |j|^{- 1}\,, \quad | q_n(j)- q_0(j) |^{{\rm Lip}(\gamma)} \lesssim
	\e  |j|^{- M}, \quad \forall j \in \Z^2 \setminus \{ 0 \}\,, \label{stime qn} \\
	&| q_n(j) - q_{n - 1}(j)|^{{\rm Lip}(\gamma)}
	\lesssim \e N_{n - 2}^{- \alpha} |j|^{- M} \quad \ \text{when } \quad n\geq 1 \,. \label{cal Nn - N n - 1}
\end{align}
The operator ${\mathcal R}_n$ is real and reversible, satisfying, for any $s_0\leq s\leq S-\Sigma(\beta)$,
\begin{equation}\label{stime cal Rn rid}
\begin{aligned}
& |{\mathcal R}_n  |_{- M, s}^{{\rm Lip}(\gamma)} 
 \leq C_*(s,\beta) \e N_{n - 1}^{- \alpha},  \quad 
|{\mathcal R}_n  |_{-M, s + \beta}^{{\rm Lip}(\gamma)}  \leq C_*(s,\beta) \e N_{n - 1}
\end{aligned}
\end{equation}
for some constant $C_* (s) = C_*(s, \tau) > 0$. 

When $n \geq 1$, there exists an invertible, real and reversibility preserving map 
$\Phi_{n -1} = {\rm Id} + \Psi_{n - 1}$, such that, 
for any $\lambda = (\omega, \zeta) \in \Lambda_n^\gamma$, 
\begin{equation}\label{coniugazione rid}
{\mathcal L}_n = \Phi_{n - 1}^{- 1} {\mathcal L}_{n - 1} \Phi_{n - 1}\,.
\end{equation}
Moreover, for any $s_0\leq s\leq S-\Sigma(\beta)$, the map $\Psi_{n - 1} : H^s_0 \to H^s_0$ satisfies
\begin{equation}\label{stime Psi n rid}
\begin{aligned}
|\Psi_{n - 1}|_{0, s}^{{\rm Lip}(\gamma)}
& \leq C(s,\beta) \e \gamma^{- 1} N_{n - 1}^{4 \tau + 2} N_{n - 2}^{- \alpha}\,, \\
|\Psi_{n - 1} |_{0, s + \beta}^{{\rm Lip}(\gamma)} & \leq C(s,\beta) \e \gamma^{- 1} N_{n - 1}^{4 \tau + 2}  N_{n - 2}\,,
\end{aligned}
\end{equation}
for some constant $C(s,\beta)>0$.
\\[1mm]
\noindent
${\bf (S2)}_n$ For all $ j \in \Z^2 \setminus \{ 0 \}$,  there exist a Lipschitz extension  of the eigenvalues $\mu_n(j;\,\cdot\,) :\Lambda_n^\gamma \to \ii\, \R$ to the set $DC(\gamma, \tau)$, denoted by
$ \widetilde \mu_n(j;\,\cdot\,): DC(\gamma, \tau) \to \ii \,\R$,
satisfying,  for $n \geq 1$, 
\begin{equation}\label{lambdaestesi}  
|\widetilde \mu_n(j) -  \widetilde \mu_{n - 1}(j) |^{{\rm Lip}(\gamma)}  \lesssim  |j|^{- M} |{\mathcal R}_{n - 1}|_{-M,s_0}^{{\rm Lip}(\gamma)}  \lesssim   |j|^{- M} \e N_{n - 2}^{- \alpha} \,.
\end{equation}

\end{proposition}

\begin{proof}
{\sc Proof of ${\bf (S1)}_0,{\bf (S2)}_0$.} 
The claimed properties follow directly 
from Proposition \ref{proposizione regolarizzazione ordini bassi}, recalling \eqref{op inizio riducibilita}, \eqref{stime pre rid} 
and the definition of $\Lambda_0^\g := DC(\gamma, \tau)$. 

\smallskip

\noindent
{\sc Proof of ${\bf (S1)}_{n+1}$.} 
By induction, we assume the the claimed properties ${\bf (S1)}_n$, ${\bf (S2)}_n$ hold for some $n \geq 0$ and we prove them at the step $n + 1$. 
Let $\Phi_n = {\rm Id} + \Psi_n$ where $\Psi_n$ is an operator 
to be determined. 
We compute 
\begin{equation}\label{primo coniugio Ln Psin}
\begin{aligned}
{\mathcal L}_n \Phi_n & = \Phi_n ( \omega \cdot \partial_\vphi + {\mathcal D}_n )  \\
& \quad + \omega \cdot \partial_\vphi \Psi_n + [{\mathcal D}_n, \Psi_n] + \Pi_{N_n} {\mathcal R}_n  \\
& \quad + \Pi_{N_n}^\bot {\mathcal R}_n + {\mathcal R}_n \Psi_n
\end{aligned}
\end{equation}
where ${\mathcal D}_n := \zeta \cdot \nabla+ {\mathcal Q}_n$ and the projectors $\Pi_N$, $\Pi_N^\bot$ are defined in  \eqref{def proiettore operatori matrici}. 
Our purpose is to find a map $\Psi_n$ solving the {\it homological equation} 
\begin{equation}\label{equazione omologica KAM}
\omega \cdot \partial_\vphi \Psi_n + [{\mathcal D}_n, \Psi_n] + \Pi_{N_n} {\mathcal R}_n ={\mathcal D}_{{\mathcal R}_n}
\end{equation}
where ${\mathcal D}_{{\mathcal R}_n}$ is the diagonal operator as per Definition \ref{def block-diagonal op}.
By  \eqref{matrix representation 1} and \eqref{def.calLn calDn calQn}, the homological equation \eqref{equazione omologica KAM} is equivalent to 
\begin{equation}\label{eq omologica matrici}
	\big( \ii \,\omega \cdot \ell + \mu_{n}(j) - \mu_{n}(j') \big) \widehat \Psi_{n}(\ell)_j^{j'} + \widehat{\mathcal R}_{n}(\ell)_j^{j'} = \widehat{{\mathcal D}_{{\mathcal R}_{n}}}(\ell)_j^{j'}
\end{equation}
for $\ell \in \Z^d$, $j, j' \in \Z^2 \setminus \{ 0 \}$. Therefore, we define the linear operator $\Psi_{n}$ by
\begin{equation}\label{def Psi eq omo KAM}
	\begin{footnotesize}
		\widehat\Psi_{n} (\ell)_j^{j'} := \begin{cases}
			- \dfrac{\widehat{\mathcal R}_{n}(\ell)_j^{j'} }{ \ii \, \omega \cdot \ell + \mu_{n}(j) - \mu_{n}(j')  }, \quad \forall (\ell , j, j') \neq (0, j, j), \quad |\ell|, |j - j'| \leq N_{n}\,, \\
			0 \quad \text{otherwise}\,,
		\end{cases}
	\end{footnotesize}
\end{equation}
which  is the solution of \eqref{eq omologica matrici}. 

\begin{lemma}\label{Lemma eq omologica riducibilita KAM}
The operator $\Psi_n$ in \eqref{def Psi eq omo KAM}, defined for any $(\omega, \zeta) \in \Lambda_{n + 1}^\gamma$, satisfies, for any $s_0\leq s\leq S-\Sigma(\beta)$,
\begin{equation}\label{stime eq omologica}
\begin{aligned}
& |\Psi_n|_{0, s}^{{\rm Lip}(\gamma)}
\lesssim_{s} N_n^{4 \tau +2} \gamma^{- 1} |{\mathcal R}_n|_{- M, s}^{{\rm Lip}(\gamma)}\,, \\
& |\Psi_n|_{0, s + M}^{{\rm Lip}(\gamma)} \lesssim_{s} N_n^{\tau_1} \gamma^{- 1} |{\mathcal R}_n|_{- M, s}\,,
\end{aligned}
\end{equation}
where $\tau_1>1$ is given in \eqref{definizione alpha beta}. Moreover, $\Psi_n$ is real and reversibility preserving.
\end{lemma}

\begin{proof}
To simplify notations, in this proof 
we drop the index $n$.
Since $\lambda = (\omega, \zeta) \in \Lambda_{n + 1}^\gamma$ (see \eqref{insiemi di cantor rid}), one immediately gets the estimate
\begin{equation}\label{stima eq omo KAM 1}
|\widehat\Psi (\ell)_j^{j'}| \lesssim \gamma^{- 1} \langle \ell \rangle^\tau | j |^\tau | j' |^\tau |{\mathcal R}(\ell)_j^{j'}|\,.
\end{equation}
For any $\lambda_1 = (\omega_1, \zeta_1), \lambda_2 = (\omega_2, \zeta_2) \in \Lambda_{n + 1}^\gamma$, we define $\delta_{\ell j j'} :=\ii\, \omega \cdot \ell +\mu(j) - \mu(j')$. By \eqref{stime qn}, \eqref{insiemi di cantor rid} one has 
\begin{equation*}
\begin{aligned}
\Big| \frac{1}{ \delta_{\ell j j'}(\lambda_1)} - \frac{1}{\delta_{\ell j j'}(\lambda_2)} \Big| & \leq \dfrac{|\delta_{\ell j j'}(\lambda_1) - \delta_{\ell j j'}(\lambda_2)|}{|\delta_{\ell j j'}(\lambda_1)| |\delta_{\ell j j'}(\lambda_2)|} \\
& \lesssim \g^{-2} \langle \ell \rangle^{2 \tau + 1} \langle j - j' \rangle |j|^{2 \tau} |j'|^{2 \tau} |\lambda_1 - \lambda_2|\,.
\end{aligned}
\end{equation*}
The latter estimate (recall \eqref{def Psi eq omo KAM}) implies also that 
\begin{equation}\label{Psi lam 12 KAM}
\begin{aligned}
|\widehat\Psi (\ell)_j^{j'}(\lambda_1) &- \widehat\Psi (\ell)_j^{j'}(\lambda_2)|  \lesssim  \langle \ell \rangle^\tau |j|^\tau |j'|^\tau\gamma^{- 1} |\widehat{\mathcal R}(\ell)_j^{j'}(\lambda_1) - \widehat{\mathcal R}(\ell)_j^{j'}(\lambda_2)| \\
& + \langle \ell \rangle^{2 \tau + 1} \langle j - j' \rangle |j|^{2 \tau} |j'|^{2 \tau} \gamma^{- 2} |\widehat{\mathcal R}(\ell)_j^{j'}(\lambda_2)| |\lambda_1 - \lambda_2 |\,. 
\end{aligned}
\end{equation}
Using that $  \langle \ell, j - j' \rangle \leq N $  and the elementary chain of inequalities $ |j| \lesssim |j - j'| + |j'| \lesssim N + |j'| \lesssim N |j'|$,
the estimates \eqref{stima eq omo KAM 1}, \eqref{Psi lam 12 KAM} take the form 
\begin{equation*}
\begin{aligned}
 |\widehat\Psi (\ell)_j^{j'}| & \lesssim  N^{2 \tau}\gamma^{- 1} |j'|^{2 \tau}  |{\mathcal R}(\ell)_j^{j'}|\,, \\
 |\widehat\Psi (\ell)_j^{j'}(\lambda_1) - \widehat\Psi (\ell)_j^{j'}(\lambda_2)| & \lesssim N^{2 \tau} \gamma^{- 1}  |j'|^{2 \tau} |\widehat{\mathcal R}(\ell)_j^{j'}(\lambda_1) - \widehat{\mathcal R}(\ell)_j^{j'}(\lambda_2)| \\
& \qquad + N^{4 \tau + 2}\gamma^{- 2}  |j'|^{ 4 \tau} |\widehat{\mathcal R}(\ell)_j^{j'}(\lambda_2)| |\lambda_1 - \lambda_2 |\,. 
\end{aligned}
\end{equation*}
Since $M > 4 \tau$ by \eqref{definizione alpha beta}, recalling Definition \ref{block norm}, the latter estimates imply that 
$$
\begin{aligned}
& |\Psi|_{0, s}^{\rm sup} \lesssim N^{2 \tau} \gamma^{- 1} |{\mathcal R}|_{- M, s}^{\rm sup}\,, \\
& |\Psi|_{0, s}^{\rm lip} \lesssim  N^{2 \tau} \gamma^{- 1}  |{\mathcal R}|_{- M, s}^{\rm lip} + N^{4 \tau + 2} \gamma^{- 2} |{\mathcal R}|_{- M, s}^{\rm sup}
\end{aligned}
$$
and similarly, using also that $\langle \ell, j - j' \rangle^{M} \lesssim N^M$, 
$$
\begin{aligned}
& |\Psi|_{0, s + M}^{\rm sup} \lesssim N^{2 \tau + M} \gamma^{- 1} |{\mathcal R}|_{- M, s}^{\rm sup}\,, \\
& |\Psi|_{0, s+ M}^{\rm lip} \lesssim  N^{2 \tau + M} \gamma^{- 1}  |{\mathcal R}|_{- M, s}^{\rm lip} + N^{4 \tau +M + 2} \gamma^{- 2} |{\mathcal R}|_{- M, s}^{\rm sup}\,. 
\end{aligned}
$$
Hence, we conclude the claimed bounds in \eqref{stime eq omologica}. Finally, since ${\mathcal R}$ is real and reversible, by Lemma \ref{lemma real rev matrici} and the properties \eqref{reversibility reality auto} for $\mu(j)$, we deduce that $\Psi$ is real and reversibility preserving. 
\end{proof}

By Lemma \ref{Lemma eq omologica riducibilita KAM} and the estimate in \eqref{stime cal Rn rid}, we obtain, for any $s_0\leq s\leq S-\Sigma(\beta)$,
\begin{equation}\label{stime Psin neumann}
\begin{aligned}
|\Psi_n|_{0, s}^{{\rm Lip}(\gamma)} 
& \lesssim_{s}  N_n^{4 \tau + 2} \gamma^{- 1} |{\mathcal R}_n |_{- M, s}^{{\rm Lip}(\gamma)}
\lesssim_{s } N_n^{4 \tau + 2} N_{n - 1}^{- \alpha} 
\e \gamma^{- 1} , \\
|\Psi_n|_{0, s + M}^{{\rm Lip}(\gamma)} \,,\, 
& \lesssim_{s}  N_n^{\tau_1} \gamma^{- 1} |{\mathcal R}_n |_{- M, s}^{{\rm Lip}(\gamma)}
\lesssim_{s } N_n^{\tau_1} N_{n - 1}^{- \alpha} 
\e \gamma^{- 1} , \\
|\Psi_n|_{0, s + \beta}^{{\rm Lip}(\gamma)} & \lesssim_{s,\beta} N_n^{4 \tau + 2} \gamma^{- 1} |{\mathcal R}_n |_{-M, s + \beta}^{{\rm Lip}(\gamma)} \lesssim_{s } N_n^{4 \tau + 2} N_{n - 1} \e \gamma^{- 1}, \\
|\Psi_n|_{0, s + \beta + M}^{{\rm Lip}(\gamma)} & \lesssim_{s,\beta} N_n^{\tau_1} \gamma^{- 1} |{\mathcal R}_n |_{-M, s + \beta}^{{\rm Lip}(\gamma)} \lesssim_{s } N_n^{\tau_1} N_{n - 1} \e \gamma^{- 1},
\end{aligned}
\end{equation} 
which are the estimates \eqref{stime Psi n rid} at the step $n + 1$. By \eqref{definizione alpha beta} and by the smallness condition \eqref{KAM smallness condition}, one has, for any $s_0\leq s\leq S-\Sigma(\beta)$ and for $N_0>0$ large enough,
\begin{equation} \label{2103.1}
|\Psi_n|_{0, s}^{{\rm Lip}(\gamma)} 
\lesssim_s N_n^{4 \tau + 2} N_{n - 1}^{- \alpha} \leq C(s) N_0^{4 \tau + 2} \e \gamma^{- 1} \leq \delta < 1
\end{equation}
Therefore, by Lemma \ref{proprieta standard norma decay}-$(iv)$, 
$\Phi_n = {\rm Id} + \Psi_n$ is invertible 
and 
\begin{equation}\label{stime Phi n inv - Id}
|\Phi_n^{- 1} - {\rm Id}|_{0, s}^{{\rm Lip}(\gamma)} 
\lesssim_{s}  |\Psi_n|_{0, s}^{{\rm Lip}(\gamma)}\,,
\quad 
|\Phi_n^{- 1} - {\rm Id}|_{0, s + \beta}^{{\rm Lip}(\gamma)} 
\lesssim_{s}  |\Psi_n|_{0, s +\beta}^{{\rm Lip}(\gamma)}. 
\end{equation}
Then, we define 
\begin{equation}\label{2003.1} 
\begin{aligned}
{\mathcal L}_{n + 1} & := \omega \cdot \partial_\vphi + {\mathcal D}_{n + 1} + {\mathcal R}_{n + 1}\,,  \\
{\mathcal D}_{n + 1} & :=  \zeta \cdot \nabla  + {\mathcal Q}_{n + 1}\,, \qquad {\mathcal Q}_{n + 1}  : = {\mathcal Q}_n + {\mathcal D}_{{\mathcal R}_n}\,,  \\
{\mathcal R}_{n + 1}& :=\Pi_{N_n}^\bot {\mathcal R}_n +  ( \Phi_n^{- 1} - {\rm Id} )  \big( {\mathcal D}_{{\mathcal R}_n} + \Pi_{N_n}^\bot {\mathcal R}_n\big)+ \Phi_n^{- 1}  {\mathcal R}_n \Psi_n . 
\end{aligned}
\end{equation}
All the operators in \eqref{2003.1}
are defined for any $\lm = (\om,\zeta) \in \Lambda_{n + 1}^\gamma$. 
Since $\Psi_n, \Phi_n, \Phi_n^{- 1}$ are real and reversibility preserving and ${\mathcal D}_n, {\mathcal R}_n$ are real and reversible operators, one gets that ${\mathcal D}_{n + 1}$, ${\mathcal R}_{n + 1}$ are real and reversible operators. 
Moreover, by \eqref{primo coniugio Ln Psin}, \eqref{equazione omologica KAM},
for $(\om,\zeta) \in \Lambda^\g_{n+1}$ one has the 
identity
$\Phi_n^{-1} \mL_n \Phi_n = \mL_{n+1}$, 
which is \eqref{coniugazione rid} at the step $n+1$. 
By Definition \ref{def block-diagonal op} applied to ${\mathcal D}_{{\mathcal R}_n}$, one has that 
\begin{equation}\label{frittata di maccheroni}
\begin{aligned}
{\mathcal Q}_{n + 1} & := {\mathcal Q}_n + {\mathcal D}_{{\mathcal R}_n} = {\rm diag}_{j \in \Z^2 \setminus \{ 0 \}} q_{n + 1}(j)\,, \\
q_{n + 1}(j) & := q_n(j) + \widehat{\mathcal R}_n(0)_j^j\,, \\
{\mathcal D}_{n + 1} &:= \zeta \cdot \nabla + {\mathcal Q}_{n + 1} = {\rm diag}_{j \in \Z^2 \setminus \{ 0 \}} \mu_{n + 1}(j)\,, \\
\mu_{n + 1}(j) & := \ii \,\zeta \cdot j + q_{n + 1}(j)\,.
\end{aligned}
\end{equation}
The reality and the reversibility of ${\mathcal D}_{n + 1}$ and ${\mathcal Q}_{n + 1}$ imply that \eqref{reversibility reality auto} is verified at the step $n + 1$. Moreover, by Lemma \ref{proprieta standard norma decay}-(v)
$$
\begin{aligned}
| \mu_{n + 1}(j) - \mu_n(j)|^{{\rm Lip}(\gamma)} & = | q_{n + 1}(j) - q_n(j)|^{{\rm Lip}(\gamma)}   \\
& \leq | \widehat{\mathcal R}_n(0)_j^j |^{{\rm Lip}(\gamma)} \lesssim |{\mathcal R}_n|_{- M, s_0}^{{\rm Lip}(\gamma)} \langle j \rangle^{- M}\,.
\end{aligned}
$$
Then, the estimate \eqref{stime cal Rn rid}
implies  \eqref{cal Nn - N n - 1} at the step $n + 1$. 
The estimate \eqref{stime qn} at the step $n + 1$ follows, as usual, 
by a telescoping argument, 
using the fact that $\sum_{n \geq 0} N_{n - 1}^{- \alpha}$ is convergent 
since $\alpha > 0$ (see \eqref{definizione alpha beta}). 
Now we prove the estimates \eqref{stime cal Rn rid} at the step $n + 1$. 
By \eqref{2003.1}, estimates 
\eqref{stime Psin neumann}, 
\eqref{2103.1},
\eqref{stime Phi n inv - Id}, 
 Lemma \ref{proprieta standard norma decay}-$(ii),(v)$ and 
Lemma \ref{lemma proiettori decadimento}, 
we get, for any $s_0\leq s \leq S-\Sigma(\beta)$,
\begin{equation*}
\begin{footnotesize}
	\begin{aligned}
		& |{\mathcal R}_{n + 1}|_{- M, s}^{{\rm Lip}(\gamma)} 
		\lesssim_{s} N_n^{- \beta} |{\mathcal R}_n|_{- M, s + \beta}^{{\rm Lip}(\gamma)} + N_n^{\tau_1} \gamma^{- 1} |{\mathcal R}_n|_{- M, s_0}^{{\rm Lip}(\gamma)} |{\mathcal R}_n|_{- M, s}^{{\rm Lip}(\gamma)} \,, \\
		& |{\mathcal R}_{n + 1}|_{- M, s + \beta}^{{\rm Lip}(\gamma)} 
		\lesssim_{s} |{\mathcal R}_n|_{- M, s +\beta}^{{\rm Lip}(\gamma)} + N_{n}^{\tau_1} \gamma^{-1}\big( |{\mathcal R}_n|_{- M, s_0}^{{\rm Lip}(\gamma)} |{\mathcal R}_n|_{- M, s+\beta}^{{\rm Lip}(\gamma)} +  |{\mathcal R}_n|_{- M, s}^{{\rm Lip}(\gamma)} |{\mathcal R}_n|_{- M, s_0+\beta}^{{\rm Lip}(\gamma)} \big) \,. 
	\end{aligned}
\end{footnotesize}
\end{equation*}
By the induction estimate \eqref{stime cal Rn rid}, 
the definition of the constants in \eqref{definizione alpha beta} and the smallness condition in  \eqref{KAM smallness condition}, 
taking $N_0 = N_0(S,  \tau)> 0$ large enough, 
we obtain the estimates \eqref{stime cal Rn rid} at the step $n + 1$. 
\\[1mm]
\noindent
{\sc Proof of ${\bf (S2)}_{n + 1}$.} It remains to construct a Lipschitz extension for the eigenvalues $\mu_{n + 1}(j,\,\cdot\,) : \Lambda_{n + 1}^\gamma \to \ii \,\R$. By the induction hypothesis, there exists a Lipschitz extension of $\mu_n(j,\lambda)$,  denoted by $\widetilde \mu_{n}(j,\lambda)$ to the whole set $DC(\gamma, \tau)$ that satisfies ${\bf (S2)}_n$. By \eqref{frittata di maccheroni}, we have $\mu_{n + 1}(j) = \mu_n(j) + r_n(j)$ where $r_n(j)=r_n(j,\lambda) := \widehat{\mathcal R}_n(0;\lambda)_j^j$ satisfies $|r_n(j)|^{{\rm Lip}(\gamma)} \lesssim  N_{n - 1}^{- \alpha} |j|^{- M} \e$. By the reversibility and the reality of ${\mathcal R}_n$, we have $r_n(j) = - r_n(- j)=\overline{r_n(-j)} $, implying that $r_n(j) \in  \ii\, \R$. Hence by the Kirszbraun Theorem (see Lemma M.5 \cite{KukPo}) there exists a Lipschitz extension $\widetilde r_n(j,\,\cdot\,) : DC(\gamma, \tau) \to \ii \R$ of $r_n(j,\,\cdot\,) : \Lambda_{n + 1}^\gamma \to \ii \,\R$ satisfying $|\widetilde r_n(j)|^{{\rm Lip}(\gamma)} \lesssim |r_n(j)|^{{\rm Lip}(\gamma)} \lesssim N_{n - 1}^{- \alpha} |j|^{- M} \e$. The claimed statement then follows by defining $\widetilde \mu_{n + 1}(j) := \widetilde  \mu_n(j) + \widetilde r_n(j)$. 
\end{proof}

\subsection{KAM reducibility: convergence}\label{sez convergenza KAM}

In this section we prove that the KAM reducibility scheme for the operator $\mL_{e}^{(2)}$, whose iterative step is described in Proposition \ref{prop riducibilita}, is convergent under the smallness condition \eqref{KAM smallness condition} with the final operator being diagonal with purely imaginary eigenvalues.

\begin{lemma}\label{lemma blocchi finali}
For any $j \in \Z^2 \setminus \{ 0 \}$, the sequence 
$\{ \widetilde \mu_n(j)= \ii\, \zeta \cdot j + q_n(j) \}_{n\in\N}$ 
converges to some limit
\begin{equation*}\label{def cal N infty nel lemma}
\mu_\infty(j) = \ii \,\zeta \cdot j + q_\infty(j)\,, \quad \mu_\infty(j)= \mu_\infty(j,\,\cdot\,):DC(\gamma,\tau)\to \ii\,\R\,,
\end{equation*}
satisfying the following estimates 
\begin{equation}\label{stime forma normale limite}
\begin{aligned}
& |  \mu_\infty(j) - \widetilde \mu_n(j) |^{{\rm Lip}(\gamma)} = | q_\infty(j) - \widetilde q_n(j) |^{{\rm Lip}(\gamma)} \lesssim  N_{n - 1}^{- \alpha} |j|^{- M} \e  \,, \\
& | q_\infty(j)|^{{\rm Lip}(\gamma)} \lesssim   |j|^{- 1} \e \,. 
\end{aligned}
\end{equation}
\end{lemma}

\begin{proof}
By Proposition \ref{prop riducibilita}-${\bf (S2)}_n$, we have that the sequence $\{\widetilde \mu_n(j,\lambda)\}_{n\in\N}\subset \ii\,\R$ is Cauchy on the closed set $DC(\gamma,\tau)$,  therefore it is convergent for any $ \lambda \in DC(\gamma,\tau)$. The estimate \eqref{stime forma normale limite} follows by a telescoping argument with the estimate \eqref{lambdaestesi}. 
\end{proof}
We define the set $\Lambda_\infty^\gamma$ of the non-resonance conditions for the final eigenvalues as 
\begin{equation}\label{cantor finale ridu}
\begin{aligned}
\Lambda_\infty^\gamma 
:= \Big\{ \lambda \in DC(\g,\t) 
& : |\ii \,\omega \cdot \ell + \mu_\infty(j) -  \mu_\infty(j')| \geq \frac{2\gamma}{\langle \ell \rangle^\tau | j |^\tau | j' |^\tau}, \\
&  \ \ \forall \ell \in \Z^d, \ \ j, j' \in \Z^2 \setminus \{ 0 \}, 
\ \ (\ell, j, j') \neq (0, j, j)\Big\}\,. 
\end{aligned}
\end{equation}

\begin{lemma}\label{prima inclusione cantor}
We have $\Lambda_\infty^\gamma \subseteq \cap_{n \geq 0} \, \Lambda_n^\gamma$.
\end{lemma}

\begin{proof}
We prove by induction that $\Lambda_\infty^\gamma \subseteq \Lambda_n^\gamma$ 
for any integer $n \geq 0$. 
The statement is trivial for $n=0$,
since $\Lambda_0^\gamma := DC(\g,\t)$ 
(see Proposition \ref{prop riducibilita}).
We now assume by induction that $\Lambda_\infty^\gamma \subseteq \Lambda_n^\gamma$ for some $n \geq 0$ and we  show that $\Lambda_\infty^\gamma \subseteq \Lambda_{n + 1}^\gamma$. Let $\lambda \in \Lambda_\infty^\gamma$, $\ell \in \Z^d$, $j, j' \in \Z^2 \setminus \{ 0 \}$, with $(\ell, j, j') \neq (0, j, j)$ and $|\ell|, |j - j'| \leq N_n$. By \eqref{stime forma normale limite}, \eqref{cantor finale ridu}, we compute
$$
\begin{aligned}
|\ii \,\omega \cdot \ell + \mu_n(j) - \mu_n(j')|  & \geq |\ii \,\omega \cdot \ell + \mu_\infty(j) - \mu_\infty(j')| - |\mu_\infty(j) - \mu_n(j)| \\
& \qquad  - |\mu_\infty(j') - \mu_n(j')| \\
& \geq \frac{2\gamma}{\langle \ell \rangle^\tau | j |^\tau | j' |^\tau} - C N_{n - 1}^{- \alpha} \e \Big( |j|^{- M} + |j'|^{- M} \Big)   \\
& \geq \frac{\gamma}{\langle \ell \rangle^\tau | j |^\tau | j' |^\tau}
\end{aligned}
$$
for some positive constant $C>0$, provided
\begin{equation}\label{frittata di maccheroni 10}
C \e \gamma^{- 1}\langle \ell \rangle^\tau | j |^\tau | j' |^\tau \Big( |j|^{- M} + |j'|^{- M} \Big) \leq 1\,.
\end{equation}
Using that $ |\ell|, |j - j'| \leq N_n$ and the chain of inequalities
\begin{equation*}
	\begin{aligned}
		& |j| \leq |j'| + |j - j'| \leq |j'| + N_n \lesssim N_n |j'|\,, 
	\end{aligned}
\end{equation*}
we deduce that, for some $C_0>0$ and recalling that $M > 4 \tau$ by \eqref{definizione alpha beta},
$$
\langle \ell \rangle^\tau | j |^\tau | j' |^\tau \Big( |j|^{- M} + |j'|^{- M} \Big) \leq C_0 N_n^{2 \tau}\,.
$$
Therefore, \eqref{frittata di maccheroni 10} is verified provided $C C_0 N_n^{2 \tau} \e \gamma^{- 1} \leq 1\,.$ The latter inequality is implied by the smallness condition \eqref{KAM smallness condition}. We conclude that $\lambda = (\omega, \zeta) \in \Lambda_{n + 1}^\gamma$. 
\end{proof}

Now we define the sequence of invertible maps
\begin{equation}\label{trasformazioni tilde ridu}
\widetilde \Phi_n := \Phi_0 \circ \Phi_1 \circ \ldots \circ \Phi_n, \quad n \in \N\,. 
\end{equation}


\begin{proposition}\label{lemma coniugio finale}
Let $S > s_0 + \Sigma(\beta)$. There exists $\delta := \delta (S, \tau, d) > 0$ such that, if  \eqref{ansatz} \eqref{KAM smallness condition} are verified, then the following holds. 
For any $\lambda = (\omega, \zeta) \in \Lambda_\infty^\gamma$, the sequence $(\widetilde \Phi_n)_{n\in\N}$
converges in norm $| \cdot |_{0, s}^{{\rm Lip}(\gamma)}$ 
to an invertible map $\Phi_\infty$, 
satisfying, for any $s_0\leq s\leq S-\Sigma(\beta)$,
\begin{equation}\label{stima Phi infty}
\begin{aligned}
& |\Phi_\infty^{\pm 1} - \widetilde \Phi_n^{\pm 1}|_{0, s}^{{\rm Lip}(\gamma)} 
\lesssim_{s} N_{n + 1}^{4 \tau + 2} N_n^{- \alpha} \e \gamma^{- 1} \,,
\quad 
|\Phi_\infty^{\pm 1} - {\rm Id}|_{0, s}^{{\rm Lip}(\gamma)} 
\lesssim_{s}  \e \gamma^{- 1} \,.
\end{aligned}
\end{equation}
The operators $\Phi_\infty^{\pm 1} : H^s_0 \to H^s_0$ are real and reversibility preserving. Moreover, for any $\lambda \in \Lambda_\infty^\gamma$, one has 
\begin{equation}\label{cal L infty e}
	{\mathcal L}_{ e}^{(\infty)} := \Phi_\infty^{- 1} {\mathcal L}_e^{(2)} \Phi_\infty = \omega \cdot \partial_\vphi + {\mathcal D}_\infty\,, \quad {\mathcal D}_\infty := {\rm diag}_{j \in \Z^2 \setminus \{ 0 \}} \mu_\infty(j) 
\end{equation} 
where the operator ${\mathcal L}_e^{(2)}$ is given in \eqref{def cal L (3)}-\eqref{op inizio riducibilita} 
and  the final eigenvalues $\mu_\infty(j)$ are given in Lemma \ref{lemma blocchi finali}.
\end{proposition}

\begin{proof}
The existence of the invertible map $\Phi_\infty^{\pm 1}$ and the estimates \eqref{stima Phi infty} follow by 
\eqref{stime Psi n rid}, 
\eqref{trasformazioni tilde ridu}, 
arguing as in Corollary 4.1 in \cite{BBM-Airy}. 
By \eqref{trasformazioni tilde ridu}, Lemma \ref{prima inclusione cantor} 
and Proposition \ref{prop riducibilita}, one has 
$\widetilde \Phi_n^{- 1} {\mathcal L}_0 \widetilde \Phi_n 
= \omega \cdot \partial_\vphi + {\mathcal D}_n + {\mathcal R}_n$ for all $n \geq 0$. 
The claimed statement then follows by passing to the limit as $n \to \infty$, 
by using \eqref{stime cal Rn rid}, \eqref{stima Phi infty} and Lemma \ref{lemma blocchi finali}. 
\end{proof}

\section{Inversion of the linearized Navier Stokes operator $\mL_\nu$}\label{sez inversioni}
The main purpose of this section is to prove the invertibility of the operators $\mL_\nu$ in \eqref{operatore linearizzato}, for any value of the viscosity $\nu > 0$ with a smallness condition on $\e$ which is independent of $\nu$. We also provide the invertibility of $\mL_{e}$, which we shall to construct an approximate solution up to order $O(\nu^2)$ in Section \ref{sez soluzione approx}. We use the normal form reduction implemented in Section \ref{sez riducibilita}. First, we recollect all the terms.
By \eqref{def cal L (3)} and by Lemma \ref{lemma coniugio finale}, for any $\lambda \in \Lambda_\infty^\gamma$, the operator $\mL^{(2)}_\nu$ in \eqref{def cal L (3)} is conjugated to
\begin{equation}\label{forma cal L infty viscosity}
\begin{aligned}
{\mathcal L}^{(\infty)}_\nu & := \Phi_\infty^{- 1} {\mathcal L}^{(2)}_\nu \Phi_\infty = \Phi_\infty^{- 1} {\mathcal L}_e^{(2)} \Phi_\infty - \nu \Phi_\infty^{- 1} \Delta \Phi_\infty + \Phi_\infty^{- 1} {\mathcal R}_\nu^{(2)} \Phi_\infty \\
& = {\mathcal L}_{e}^{(\infty)}   - \nu \Delta + {\mathcal R}_{\nu}^{(\infty)}\,, \\
 {\mathcal R}_{\nu}^{(\infty)}& := - \nu \big( \Delta (\Phi_\infty - {\rm Id}) + (\Phi_\infty^{- 1} - {\rm Id}) \Delta \Phi_\infty  \big) + \Phi_\infty^{- 1} {\mathcal R}_{\nu}^{(2)} \Phi_\infty 
\end{aligned}
\end{equation}
By the estimates \eqref{stima cal Z cal R (3)}, \eqref{stima Phi infty}, using that $|\Delta|_{2, s} \leq 1$, for any $s > s_0$ and for $S > s_0 + \Sigma(\beta) + 2$, there exists $\delta := \delta(S, \tau, d) \in (0, 1)$ such that if \eqref{ansatz}, \eqref{condizione piccolezza rid trasporto} hold, by applying Lemma \ref{proprieta standard norma decay}-$(ii)$, one gets 
\begin{equation}\label{stima R infty n}
|{\mathcal R}_{\nu}^{(\infty)}|^{{\rm Lip}(\gamma)}_{2, s} \lesssim_s \e  \gamma^{- 1}\,\nu\,, \quad \forall s_0 \leq s \leq S - \Sigma(\beta) - 2\,. 
\end{equation}

\begin{lemma}[{\textbf{Inversion of ${\mathcal L}_\nu^{(\infty)}$}}]
\label{lemma inversione cal L infty} 
For any $S > s_0 + \Sigma(\beta) + 4$, there exists 
$\delta := \delta (s, \tau, d) \in (0,1)$ small enough such that, if \eqref{ansatz} holds and $\e \gamma^{- 1} \leq \delta$, the operator ${\mathcal L}_\nu^{(\infty)}$ is invertible for any $\nu > 0$ with bounded inverse $({\mathcal L}_\nu^{(\infty)})^{- 1} \in {\mathcal B}(H^s_0)$ for any $s_0\leq s\leq S-\Sigma(\beta) - 4$. Moreover, the inverse operator $(\mL_{\nu}^{(\infty)})^{-1}$ is smoothing of order two, that is $({\mathcal L}_\nu^{(\infty)})^{- 1} (- \Delta) \in {\mathcal B}(H^s_0)$, with $\| ({\mathcal L}_\nu^{(\infty)})^{- 1} (- \Delta)\|_{{\mathcal B}(H^s_0)} \lesssim_s \nu^{- 1}$. 
\end{lemma}
\begin{proof}
By \eqref{forma cal L infty viscosity} and \eqref{cal L infty e}, we write 
\begin{equation}\label{splitting cal L infty}
\begin{aligned}
& {\mathcal L}_\nu^{(\infty)} = {\mathtt L}_\nu^{(\infty)} + {\mathcal R}_{ \nu}^{(\infty)}\,, \qquad  {\mathtt L}_\nu^{(\infty)} := \mL_{e}^{(\infty)} -\nu\Delta =  \omega \cdot \partial_\vphi + {\mathcal D}_\infty - \nu \Delta\,. 
\end{aligned}
\end{equation}
By Lemma \ref{lemma coniugio finale}, one has that ${\mathtt L}_\nu^{(\infty)}$ is a diagonal operator with eigenvalues $\ii \,\omega \cdot \ell + \mu_\infty(j) + \nu |j|^2$,  for $\ell \in \Z^d, j \in \Z^2 \setminus \{ 0 \}$. By Lemma \ref{lemma blocchi finali}, $\mu_\infty(j) \in \ii \R$ is purely imaginary. Therefore, for any $\ell \in \Z^d$, $j \in \Z^2 \setminus \{ 0 \}$, 
we obtain a lower bound of the eigenvalues for any value of the parameters $\lambda = (\omega, \zeta)$ as follows: 
\begin{equation}\label{lower bound autov viscosity}
\begin{aligned}
|\ii \,\omega \cdot \ell + \mu_\infty(j) + \nu |j|^2| & \geq \big|{\rm Re}\big( \ii \,\omega \cdot \ell + \mu_\infty(j) + \nu |j|^2\big) \big| = \nu |j|^2\,.
\end{aligned}
\end{equation}
The latter bound implies that  the operator ${\mathtt L}_\nu^{(\infty)}$ in \eqref{splitting cal L infty} is invertible for any $\nu>0$,  with inverse given by 
$$
({\mathtt L}_\nu^{(\infty)})^{- 1} h(\vphi, x) = \sum_{\begin{subarray}{c}
\ell \in \Z^d \\
j \in \Z^2 \setminus \{ 0 \}
\end{subarray}} \dfrac{\widehat h(\ell, j)}{\ii \,\omega \cdot \ell + \mu_\infty(j) +\nu |j|^2} \,e^{\ii \ell \cdot \vphi } e^{\ii j \cdot x}\,, \quad h \in H^s_0\,.  
$$
Recalling that  $\langle j \rangle = {\rm max}\{ 1, |j| \}=|j|$ for $j \neq 0$, by the smoothing term from the small divisor \eqref{lower bound autov viscosity}, one has, for any $s_0\leq s \leq S-\Sigma(\beta)$,
$$
\begin{aligned}
\| ( {\mathtt L}_\nu^{(\infty)})^{- 1} (- \Delta) h \|_s^2 & = \sum_{\begin{subarray}{c}
\ell \in \Z^d \\
j \in \Z^2 \setminus \{ 0 \}
\end{subarray}} \langle \ell , j \rangle^{2 s} \frac{| j |^{4}}{|\ii \,\omega \cdot \ell + \mu_\infty(j) + \nu |j|^2|^2} |\widehat h(\ell, j)|^2  \\
& \leq \nu^{- 2}  \sum_{\begin{subarray}{c}
\ell \in \Z^d \\
j \in \Z^2 \setminus \{ 0 \}
\end{subarray}} \langle \ell, j \rangle^{2 s} |\widehat h(\ell, j)|^2 = \nu^{- 2} \| h \|_s^2 
\end{aligned}
$$
implying that 
\begin{equation}\label{stima L infty in}
\| ({\mathtt L}_\nu^{(\infty)})^{- 1} (- \Delta) \|_{{\mathcal B}(H^s_0)} \leq \nu^{- 1} \quad \text{and } \quad \| ({\mathtt L}_\nu^{(\infty)})^{- 1} \|_{{\mathcal B}(H^s_0)} \leq \nu^{- 1}\,. 
\end{equation}
We write the operator ${\mathcal L}_\nu^{(\infty)}$ in \eqref{splitting cal L infty} as
$
{\mathcal L}_\nu^{(\infty)} = {\mathtt L}_\nu^{(\infty)} \big( {\rm Id} + ({\mathtt L}_\nu^{(\infty)})^{- 1} {\mathcal R}_{ \nu}^{(\infty)} \big)
$.
By the estimates  \eqref{stima L infty in}, \eqref{stima R infty n} 
and Lemma \ref{proprieta standard norma decay}-$(i)$,$(ii)$, one gets, for any $s_0\leq s \leq S-\Sigma(\beta) - 4$,
$$
\begin{aligned}
\| ({\mathtt L}_\nu^{(\infty)})^{- 1} {\mathcal R}_{ \nu}^{(\infty)} \|_{{\mathcal B}(H^s_0)} & \leq \| ({\mathtt L}_\nu^{(\infty)})^{- 1} (- \Delta) \|_{{\mathcal B}(H^s_0)} \| (- \Delta)^{- 1} {\mathcal R}_{ \nu}^{(\infty)} \|_{{\mathcal B}(H^s_0)} \\
& \leq \nu^{- 1} |(- \Delta)^{- 1} {\mathcal R}_{ \nu}^{(\infty)} |_{0, s} \lesssim \nu^{- 1} |{\mathcal R}_{ \nu}^{(\infty)}|_{2, s + 2} \\
& \lesssim_s \nu^{- 1} \nu \e \gamma^{- 1} \lesssim_s \e \gamma^{- 1}\,.
\end{aligned}
$$
Hence, having $\e \gamma^{- 1} \leq \delta < 1$ small enough (independent of the viscosity $\nu > 0$), the operator ${\rm Id} + ({\mathtt L}_\nu^{(\infty)})^{- 1} {\mathcal R}_{ \nu}^{(\infty)}$ is invertible by Neumann series, uniformly with respect to the viscosity parameter $\nu > 0$ and $\big\| \big( {\rm Id} + {\mathtt L}_\infty^{- 1} {\mathcal R}_{\infty, \nu} \big)^{- 1} \big\|_{{\mathcal B}(H^s_0)} \lesssim_s 1$. Together with the estimate \eqref{stima L infty in}, one deduces that
$
({\mathcal L}_\nu^{(\infty)})^{- 1} = \big( {\rm Id} + ({\mathtt L}_\nu^{(\infty)})^{- 1} {\mathcal R}_{\nu}^{(\infty)} \big)^{- 1} ({\mathtt L}_\nu^{(\infty)})^{- 1}
$
, with
$$
\begin{aligned}
\|({\mathcal L}_\nu^{(\infty)})^{- 1}  (- \Delta) \|_{{\mathcal B}(H^s_0)} & 
\lesssim_s \nu^{- 1}\,.
\end{aligned}
$$
The claimed statement has then been proved. 
\end{proof}


We now deal with the inversion of the linearized operator $\mL_{\nu}$ in \eqref{operatore linearizzato}.
By Propositions \ref{proposizione trasporto}, \ref{proposizione regolarizzazione ordini bassi}, Lemma \ref{lemma coniugio finale} and recalling the definition of the set $\Lambda_\infty^\gamma$ in \eqref{cantor finale ridu}, one has that, for any $\lambda \in \Lambda_\infty^\gamma$,
\begin{equation}\label{coniugio finale cal L}
\begin{aligned}
& {\mathcal L}_\nu = {\mathcal W}_\infty {\mathcal L}_\nu^{(\infty)} {\mathcal W}_\infty^{- 1}, \quad {\mathcal L}_e = {\mathcal W}_\infty {\mathcal L}_{e}^{(\infty)} {\mathcal W}_\infty^{- 1} \,, \quad  {\mathcal W}_\infty := {\mathcal A}_\bot {\mathcal B} \Phi_\infty\,,
\end{aligned}
\end{equation} 
where the invertible maps $\mA_\perp$, $\mB$ and $\Phi_\infty$ are provided in Propositions \ref{lemma coniugio cal L (0)},  \ref{proposizione regolarizzazione ordini bassi} and \ref{lemma coniugio finale}, respectively.
Furthermore, by Lemma \ref{lemma inversione cal L infty}, the inverse of the operator ${\mathcal L}_\nu$ is given by ${\mathcal L}_\nu^{- 1} = {\mathcal W}_\infty ({\mathcal L}_\nu^{(\infty)})^{- 1} {\mathcal W}_\infty^{- 1}$. We need to estimate $\| {\mathcal L}_\nu^{- 1} (- \Delta) \|_{{\mathcal B}(H^s_0)}$. First, need some auxiliary lemmata. Let us denote
\begin{equation*}\label{laplaciano esteso}
	\langle D \rangle^{2} := {\rm Id} - \Delta\,, \quad \langle D \rangle^{-2} := ({\rm Id} - \Delta)^{-1} \,.
\end{equation*}
\begin{lemma}\label{stima strana cambio di variabile}
There exists $\Sigma_1(\beta) > \Sigma(\beta)$ (where $\Sigma(\beta)$ is given in \eqref{definizione alpha beta}) large enough such that, for any $S > s_0 + \Sigma_1(\beta)$, there exists $\delta := \delta (S, \tau, d) \in (0, 1)$ such that, if \eqref{ansatz} is fulfilled and $\e \gamma^{- 1} \leq \delta$, the invertible map ${\mathcal A}$ given in Proposition \ref{proposizione trasporto} satisfies $\langle D \rangle^{- 2} {\mathcal A}^{\pm 1} \langle D \rangle^2 \in {\mathcal B}(H^s)$ for any $s_0\leq s \leq S-\Sigma_1(\beta)$, with estimate $\| \langle D \rangle^{- 2} {\mathcal A}^{\pm 1} \langle D \rangle^2 \|_{{\mathcal B}(H^s)} \lesssim_s 1$.  
\end{lemma}
\begin{proof}
We prove the claim for $\langle D \rangle^{- 2} {\mathcal A} \langle D \rangle^2$. The proof for $\langle D \rangle^{- 2} {\mathcal A}^{-1} \langle D \rangle^2$ is analogous and we omit it. For any $\tau \in [- 1, 1]$, we define the operator ${\mathcal A}(\tau)$ by
$$
{\mathcal A}(\tau) : u_0(\vphi, x) \mapsto u_0(\vphi, x + \tau \alpha(\vphi, x))\,.
$$
By Proposition \eqref{proposizione trasporto}, for any $\tau \in [- 1, 1]$, the map ${\mathcal A}(\tau)$ is invertible on $H^s(\T^{d + 2})$ and 
\begin{equation}\label{stima A tau inv}
\| {\mathcal A}(\tau)^{\pm 1} \|_{{\mathcal B}(H^s)} \lesssim_s 1\,, \quad \forall \tau \in [- 1 \,,\, 1]\,. 
\end{equation}
For some $\Sigma_1(\beta) \gg 0$ sufficiently large and for any $S > s_0 + \Sigma_1(\beta)$, we can apply the estimate \eqref{stima alpha trasporto} with $s + 4$ instead of $s$, obtaining, for any $s_0\leq s \leq S-\Sigma_1(\beta)$,  
$$
\| \alpha \|_{s + 4}, \| \breve \alpha \|_{s + 4} \lesssim_s \e \gamma^{- 1}\,. 
$$
A direct calculation shows that ${\mathcal A}(\tau)$ solves 
$$
\begin{cases}
\partial_\tau {\mathcal A}(\tau) = b(\tau, \vphi, x) \cdot \nabla {\mathcal A}(\tau)\,, \\
{\mathcal A}(0) = {\rm Id}\,,
\end{cases}
 \text{where} \quad b(\tau, \vphi, x) := ({\rm Id} + \tau D_{x} \alpha(\vphi, x))^{- 1} \alpha(\vphi, x)\,.
$$
Note that, since $\| \alpha \|_{s + 1} \lesssim_s \e \gamma^{- 1} \leq \delta < 1$, the $2 \times 2$ matrix ${\rm Id} + D_{x} \alpha(\vphi, x)$ is invertible by Neumann series and by using the tame estimate in \eqref{p1-pr}, one has that  $\| ({\rm Id} + \tau D_{x} \alpha(\vphi, x))^{- 1}  \|_s \lesssim_s  1 $. Hence 
\begin{equation}\label{stima b}
\| b(\tau, \cdot) \|_s \lesssim_s \| \alpha \|_{s + 1} \lesssim_s \e \gamma^{- 1}\,, \quad \text{uniformly in} \quad \tau \in [- 1, 1]\,. 
\end{equation}
Define 
$$
\Phi(\tau) := \langle D \rangle^{- 2} {\mathcal A}(\tau) \langle D \rangle^2, \quad \tau \in [- 1, 1]\,.
$$
Clearly $\Phi(0, \vphi) = {\rm Id}$ and, by a direct computation,
\begin{equation*}
	\begin{aligned}
		&\partial_\tau \Phi(\tau) = b (\tau, \vphi, x) \cdot \nabla \Phi(\tau) + {\mathcal R}(\tau)  \Phi(\tau) \,,\\ 
		& {\mathcal R}(\tau) :=  [\langle D \rangle^{- 2},b (\tau, \vphi, x) \cdot \nabla] \langle D \rangle^{2} \,.
	\end{aligned}
\end{equation*}
By variation of the constants, we write $\Phi(\tau) = {\mathcal A}(\tau) {\mathcal M}(\tau)$, where ${\mathcal M}(\tau)$ solves 
\begin{equation}\label{disastro totale 10}
\begin{cases}
\partial_\tau {\mathcal M}(\tau) = {\mathcal Q}(\tau) {\mathcal M}(\tau) \,,\\
{\mathcal M}(0) = {\rm Id}\,, 
\end{cases} \quad \text{with} \quad {\mathcal Q}(\tau) := {\mathcal A}(\tau)^{- 1} {\mathcal R}(\tau) {\mathcal A}(\tau)\,. 
\end{equation}
First, we estimate $\mR(\tau)$. We claim that 
\begin{equation}\label{stima cal R tau lemma A}
\sup_{\tau \in [- 1, 1]} \|{\mathcal R}(\tau)\|_{{\mathcal B}(H^s)} \lesssim_s \e \gamma^{- 1}\,. 
\end{equation}
By a direct calculation, the matrix representation of the operator ${\mathcal R}(\tau)$ is given by 
\begin{equation}\label{disastro assoluto zero}
\widehat{\mathcal R}(\tau, \ell)_j^{j'} = \Big(\dfrac{\langle j' \rangle^2 - \langle j \rangle^2}{\langle j \rangle^2 \langle j' \rangle^2}  \Big) \langle j' \rangle^2\, \ii \,j' \cdot \widehat b(\tau, \ell, j - j') , \quad (\ell, j, j') \in \Z^d \times \Z^2 \times \Z^2\,. 
\end{equation}
Note that $\widehat{\mathcal R}(\tau, \ell)_j^{j'} = 0$ when $j'=0$.
Therefore, we can always consider $j' \neq 0$, implying that $\langle j' \rangle = |j'|$. One then has 
$$
\begin{aligned}
\langle j' \rangle^2 - \langle j \rangle^2 & = |j'|^2 - \langle j \rangle^2 = |j + j' - j|^2 - \langle j \rangle^2  \\
&= |j|^2 + 2 j \cdot (j - j') + |j - j'|^2 - \langle j \rangle^2\,, \\
\end{aligned}
$$
which implies that $ |\langle j' \rangle^2 - \langle j \rangle^2 | \lesssim \langle j - j' \rangle^2 \langle j \rangle$.
This implies that the coefficients in \eqref{disastro assoluto zero} are estimated as
$$
\begin{aligned}
|\widehat{\mathcal R}(\tau, \ell)_j^{j'} | & \lesssim 
\dfrac{ \langle j - j' \rangle^2  \langle j' \rangle}{\langle j \rangle }|\widehat b(\tau, \ell, j - j')| \\
& \lesssim \dfrac{ \langle j - j' \rangle^2  (\langle j\rangle + \langle j - j'\rangle )}{\langle j \rangle }|\widehat b(\tau, \ell, j - j')|  \lesssim \langle j - j' \rangle^3 |\widehat b(\tau, \ell, j - j')| 
\end{aligned}
$$
Then, by \eqref{def decay norm}, \eqref{prop multiplication decay} and \eqref{stima b}, one gets
$$
|{\mathcal R}(\tau)|_{0, s} \lesssim \| b(\tau, \cdot) \|_{s + 3} \lesssim_s \| \alpha \|_{s + 4} \lesssim_s \e \gamma^{- 1}\,. 
$$
We conclude the claimed bound \eqref{stima cal R tau lemma A} by Lemma \ref{proprieta standard norma decay}-$(i)$. 

\noindent
We now estimate $\Phi(\tau)$. By the estimates \eqref{stima A tau inv}, \eqref{stima cal R tau lemma A}, one gets that ${\mathcal Q}(\tau)$ satisfies 
$$
\sup_{\tau \in [- 1, 1]} \| {\mathcal Q}(\tau) \|_{{\mathcal B}(H^s)} \lesssim_s \e \gamma^{- 1}\,.
$$ 
Hence, by \eqref{disastro totale 10}, ${\mathcal M}(\tau)$ is the propagator associated to a linear vector field which is bounded on $H^s(\T^{d + 2})$ uniformly with respect to $\tau$. By standard arguments, we get
$$
\sup_{\tau \in [- 1, 1]} \| {\mathcal M}(\tau) - {\rm Id} \|_{{\mathcal B}(H^s)} \lesssim_s \e \gamma^{-1}\,.
$$
Since $\Phi(\tau) = {\mathcal A}(\tau) {\mathcal M}(\tau)$, the latter estimate together with \eqref{stima A tau inv} and $\e \gamma^{- 1} \leq \delta 1$ implies the claimed bound for $\Phi(1)\equiv \langle D \rangle^{-2} \mA \langle D \rangle^{2}$.
\end{proof}


\begin{lemma}\label{tame trasformazioni finali}
There exists $\Sigma_2(\beta) > \Sigma_1(\beta)$ large enough (where $\Sigma_1(\beta)$ is given in Lemma \ref{stima strana cambio di variabile}) such that, for any $S > s_0 + \Sigma_2(\beta)$, there exists $\delta := \delta (S, \tau, d) \in (0, 1)$ such that, if \eqref{ansatz} is fulfilled and $\e \gamma^{- 1} \leq \delta$, the maps ${\mathcal W}_\infty^{\pm 1} : H^s_0 \to H^s_0$ in \eqref{coniugio finale cal L} are bounded for any $s_0\leq s\leq S-\Sigma_2(\beta)$ with estimates $\| {\mathcal W}_\infty^{\pm 1} \|_{{\mathcal B}(H^s_0)} \lesssim_s 1$. Moreover, we have  $\| (- \Delta)^{- 1}{\mathcal W}_\infty^{\pm1} (- \Delta) \|_{{\mathcal B}(H^s_0)} \lesssim_s 1$. 
\end{lemma}

\begin{proof}
We prove the claimed bound for $(- \Delta)^{- 1}{\mathcal W}_\infty (- \Delta)$. The other bounds follow similarly and we omit their proof. By \eqref{coniugio finale cal L}, one has
$$
\begin{aligned}
(- \Delta)^{- 1}{\mathcal W}_\infty (- \Delta) & = (- \Delta)^{- 1}{\mathcal A}_\bot (- \Delta) (- \Delta)^{- 1}{\mathcal B} \Phi_\infty (- \Delta) \\
& = \Pi_0^\bot (\langle D \rangle^{- 2}{\mathcal A} \langle D \rangle^2) \Pi_0^\bot (- \Delta)^{- 1}{\mathcal B} \Phi_\infty (- \Delta) \\
& =  \langle D \rangle^2 (- \Delta)^{- 1} \Pi_0^\bot \langle D \rangle^{- 2} {\cal A} \langle D \rangle^2 \Pi_0^\bot (- \Delta) \langle D \rangle^{- 2} (- \Delta)^{- 1}{\mathcal B} \Phi_\infty (- \Delta)
\end{aligned}
$$ 
Hence, the claimed bound follows by Lemma \ref{stima strana cambio di variabile}, Lemma \ref{proprieta standard norma decay}-$(i)$, $(ii)$, the estimates \eqref{stima alpha trasporto}  \eqref{stima Phi M reg ordini bassi}, \eqref{stima Phi infty} (for $\e \gamma^{- 1} \leq \delta < 1$) and by  the trivial fact that 
$\|\langle D \rangle^2 (- \Delta)^{- 1} \|_{{\mathcal B}(H^s)}$, $ \|(- \Delta) \langle D \rangle^{- 2} \|_{{\mathcal B}(H^s)} \lesssim 1$ for any $s \geq 0$. 
\end{proof}

We are now ready to prove the invertibility of the operator $\mL_{\nu}$.
\begin{proposition}\label{inversione linearized}
	{\bf (Inversion of the operator $\mL_{\nu}$).}
There exists $\Sigma_3(\beta) > \Sigma_2(\beta)$ (where $\Sigma_2(\beta)$ is given in Lemma \ref{tame trasformazioni finali}) such that, for any $S > s_0 + \Sigma_3(\beta)$, there exists $\delta := \delta (S, \tau, d) \in (0, 1)$ such that, if \eqref{ansatz} holds and $\e \gamma^{- 1} \leq \delta$, for any value of the viscosity parameter $\nu > 0$, for any $\lambda = (\omega, \zeta) \in \Lambda_\infty^\gamma $ and for any $s_0\leq s\leq S-\Sigma_3(\beta)$, the operator ${\mathcal L}_\nu$ is invertible with a bounded inverse ${\mathcal L}_\nu^{- 1} \in {\mathcal B}(H^s_0)$, satisfying the estimates
$$\|{\mathcal L}_\nu^{- 1} \|_{{\mathcal B}(H^s_0)}\,,\,  \|{\mathcal L}_\nu^{- 1} (- \Delta) \|_{{\mathcal B}(H^s_0)} \lesssim_s \nu^{- 1}\,.$$
\end{proposition}
\begin{proof}
We write ${\mathcal L}_\nu^{- 1} = {\mathcal W}_\infty ({\mathcal L}_\nu^{(\infty)})^{- 1} {\mathcal W}_\infty^{- 1} $ and, consequently,
$$
\begin{aligned}
{\mathcal L}_\nu^{- 1} (- \Delta)& = {\mathcal W}_\infty ({\mathcal L}_\nu^{(\infty)})^{- 1} {\mathcal W}_\infty^{- 1} (- \Delta) \\
&  = {\mathcal W}_\infty \big( ({\mathcal L}_\nu^{(\infty)})^{- 1} (- \Delta) \big) \big( (- \Delta)^{- 1} {\mathcal W}_\infty^{- 1} (- \Delta) \big)\,.
\end{aligned}
$$
Therefore, by Lemmata \ref{lemma inversione cal L infty}, \ref{tame trasformazioni finali}
$$
\begin{aligned}
\|  {\mathcal L}_\nu^{- 1} (- \Delta)\|_{{\mathcal B}(H^s_0)} & \leq \|  {\mathcal W}_\infty\|_{{\mathcal B}(H^s_0)}\| ({\mathcal L}_\nu^{(\infty)})^{- 1} (- \Delta) \|_{{\mathcal B}(H^s_0)}\| (- \Delta)^{- 1} {\mathcal W}_\infty^{- 1} (- \Delta)\|_{{\mathcal B}(H^s_0)}  \\
& \lesssim_s \nu^{- 1}
\end{aligned}
$$
The bound  $\| {\mathcal L}_\nu^{- 1} \|_{{\mathcal B}(H^s_0)} \lesssim_s \nu^{- 1}$ follows from the latter.
\end{proof}
In order to compute a good approximate solution for the nonlinear equation, we also need to invert the linearized Euler operator ${\mathcal L}_e$ at the Euler solution $v_e$, see \eqref{operatore linearizzato}. 
To this purpose, we then define the set of the \emph{first Melnikov non-resonance  conditions}
\begin{equation}\label{prime di melnikov}
\begin{aligned}
\Gamma_\infty^\gamma & := \Big\{ \lambda = (\omega, \zeta) \in DC(\gamma, \tau) \,: \, |\ii \, \omega \cdot \ell + \mu_\infty(j)| \geq \frac{\gamma}{\langle \ell \rangle^\tau |j|^\tau} \\
&\quad  \quad \forall\, (\ell, j) \in \Z^d \times (\Z^2 \setminus \{ 0 \}) \Big\}\,. 
\end{aligned}
\end{equation}
\begin{proposition}\label{inversione linearized no viscosity}
	{\bf (Inversion of the operator $\mL_{\e}$).}
There exists $\Sigma_4(\beta) > \Sigma_3(\beta)$ large enough (where $\Sigma_3(\beta)$ is given in Proposition \ref{inversione linearized}) such that, for any $S > s_0 + \Sigma_4(\beta)$, there exists $\delta := \delta(S,  \tau, d) \in (0, 1)$ small enough such that, if \eqref{ansatz} holds and $\e \gamma^{- 1} \leq \delta$, for any $\lambda = (\omega, \zeta) \in \Lambda_\infty^\gamma \cap \Gamma_\infty^\gamma$, the operator ${\mathcal L}_e : H^{s + 1}_0 \to H^s_0$ is invertible (with loss of derivatives) with the inverse ${\mathcal L}_e^{- 1} \in {\mathcal B}(H^{s + \tau}_0, H^s_0)$ for any $s_0\leq s \leq S-\Sigma_4(\beta)$, satisfying $$\| {\mathcal L}_e^{- 1} \|_{{\mathcal B}(H^{s + \tau}_0, H^s_0)} \lesssim_s \gamma^{- 1}\,.$$ 
\end{proposition}
\begin{proof}
By \eqref{coniugio finale cal L}, one has that ${\mathcal L}_e = {\mathcal W}_\infty {\mathcal L}_{ e}^{(\infty)} {\mathcal W}_\infty^{- 1}$ where ${\mathcal L}_{e}^{(\infty)}$ 
is defined in \eqref{cal L infty e}. Thus, for any $\lambda = (\omega, \zeta) \in \Gamma_\infty^\gamma$ (see \eqref{prime di melnikov}), we can invert ${\mathcal L}_\infty$, its inverse is given by 
$$
({\mathcal L}_{e}^{(\infty)})^{- 1} = {\rm diag}_{(\ell, j) \in \Z^d \times (\Z^2 \setminus \{ 0 \})} \dfrac{1}{\ii \,\omega \cdot \ell + \mu_\infty(j)}\,,
$$
satisfying the estimate $\| ({\mathcal L}_{e}^{(\infty)})^{- 1} \|_{{\mathcal B}(H^{s + \tau}_0, H^s_0)} \leq \gamma^{- 1}$. The claimed statement then follows by Lemma \ref{tame trasformazioni finali} and using that ${\mathcal L}_e^{- 1} = {\mathcal W}_\infty ({\mathcal L}_{ e}^{(\infty)})^{- 1} {\mathcal W}_\infty^{- 1}$. 
\end{proof}

\section{Approximate solutions}\label{sez soluzione approx}
The purpose of this section is to find an approximate solution up to order $O(\nu^2)$ of the functional equation ${\mathcal F}_\nu(v) = 0$ where ${\mathcal F}_\nu$ is the nonlinear operator defined in \eqref{equazione cal F vorticita}. We actually write 
\begin{equation}\label{altra forma cal F v}
\begin{aligned}
& {\mathcal F}_\nu(v) = {\mathtt L}_0 v +\e  {\mathcal Q}(v) - \e F - \nu \Delta v \,, \\
& \text{where} \quad  {\mathtt L}_0 := \omega \cdot \partial_\vphi + \zeta \cdot \nabla\,, \quad {\mathcal Q}(v) := {\mathcal N}(v, v)\,, \\
& {\mathcal N}(v_1, v_2) :=  \nabla_\bot  (- \Delta)^{- 1} v_1  \cdot \nabla v_2\,,
\end{aligned}
\end{equation}
with $\nabla_\bot$ as in \eqref{equazione vorticita media nulla}.
The map $v \mapsto {\mathcal Q}(v)$ is a quadratic form. Therefore, for any $v_1, v_2$,
\begin{equation}\label{cal Q v1 v2}
\begin{aligned}
{\mathcal Q}(v_1 + v_2) & = {\mathcal Q}(v_1) + \di{\mathcal Q}(v_1)[v_2] + {\mathcal Q}(v_2)\,, \\
\di{\mathcal Q}(v_1)[v_2] & = {\mathcal N}(v_1, v_2) + {\mathcal N}(v_2, v_1)\,. 
\end{aligned}
\end{equation}
By standard Sobolev algebra estimates, one has, for any $v,h \in H^{s+1}_0$, $s\geq s_0$,
\begin{equation}\label{stime euler quadratica}
\begin{aligned}
\| {\mathcal Q}(v) \|_s & \lesssim_s  \| v \|_{s + 1}^{2}\,, \quad  \| \di{\mathcal Q}(v)[h] \|_s & \lesssim_s \| v \|_{s + 1} \| h \|_{s + 1}\,. 
\end{aligned}
\end{equation}
We recall the function $v_e$ solves the Euler equation, i.e. \eqref{altra forma cal F v} with $\nu = 0$, and satisfies  \eqref{ansatz}:
\begin{equation}\label{sol eulero schema approx}
\begin{aligned}
& {\mathtt L}_0 v_e +\e  {\mathcal Q}(v_e) - \e F = 0 \quad \text{and} \\
& \| v_e \|_S \lesssim_S \e^{\mathtt a}, \quad \mathtt a \in (0, 1)\,, \quad S > \overline S
  \end{aligned}
\end{equation}
where $\overline S$ is given in Theorem \ref{main theorem 2}. 
We now prove the following proposition.
\begin{proposition}[\bf Approximate solutions]\label{prop approximate soluions}
There exists $\overline \mu > \Sigma_4(\beta)$ large enough (where $\Sigma_4(\beta)$ is given in Proposition \ref{inversione linearized no viscosity}) such that, for any $S \geq {\rm max}\{ \overline S\,,\,s_0 + \overline \mu \}$, there exists $\delta := \delta(S, \tau, d) \in (0, 1)$ small enough such that, if $\e^{\mathtt a} \gamma^{- 1} \leq \delta$, for any $\lambda = (\omega, \zeta) \in \Lambda_\infty^\gamma \cap \Gamma_\infty^\gamma$ (see \eqref{cantor finale ridu}, \eqref{prime di melnikov}), there exists an approximate solution of the form $v_{app} = v_e + \nu v_1$ of the functional equation ${\mathcal F}_\nu(v) = 0$ with the following properties:
\begin{equation}\label{stime v1 cal F va}
\begin{aligned}
& v_1 \in H^s_0, \quad \| v_1 \|_s \lesssim_s \e^{\mathtt a} \gamma^{- 1},  \\
& \| {\mathcal F}_{\nu}(v_{app}) \|_s \lesssim_s \e^{\mathtt a} \gamma^{- 1}  \nu^2 \,, \quad \forall s_0 \leq s \leq S - \overline \mu\,. 
\end{aligned}
\end{equation}
\end{proposition}
\begin{proof}
We look for an approximate solution up to order $O(\nu^2)$ of the form 
$$
v_{app} := v_e + \nu v_1
$$
where $v_e$ is the solution of the Euler equation in \eqref{sol eulero schema approx} and $v_1$ has to be determined. By \eqref{sol eulero schema approx}, one has 
\begin{equation}\label{cal F ve v1}
\begin{aligned}
{\mathcal F}_{\nu}(v_e + \nu v_1) & =  {\mathtt L}_0 v_e +\e  {\mathcal Q}(v_e) - \e F   + \nu \big(  {\mathtt L}_0 v_1 + \e  \di {\mathcal Q}(v_e) v_1 -  \Delta v_e \big) \\
& + \nu^2 \big(  - \Delta v_1+  \e  {\mathcal Q}(v_1)  \big)     \\
& = \nu \big(  {\mathtt L}_0 v_1 + \e  {\mathcal Q}'(v_e) v_1 -  \Delta v_e \big)  + \nu^2 \big(  - \Delta v_1+  \e  {\mathcal Q}(v_1)  \big)\,.
\end{aligned}
\end{equation}
We want to choose $v_1$ in such a way that 
\begin{equation}\label{eq v1}
 {\mathtt L}_0 v_1 + \e \di {\mathcal Q}(v_e) v_1 -  \Delta v_e  = 0\,.
\end{equation}
Note that $\Delta v_e$ has zero space average and ${\mathtt L}_0 + \e  \di{\mathcal Q}(v_e) = {\mathcal L}_e$, see \eqref{operatore linearizzato}-\eqref{definizione cal R}. Hence, by Proposition \ref{inversione linearized no viscosity} and by \eqref{sol eulero schema approx}, for some $\overline \mu \gg 0$ sufficiently large and for any $S > s_0 + \overline \mu$, if \eqref{ansatz} holds and $\e^{\mathtt a} \gamma^{- 1} \leq \delta < 1$, we define, for any $\lambda = (\omega, \zeta) \in \Lambda_\infty^\gamma \cap \Gamma_\infty^\gamma$, 
\begin{equation}\label{stima v1}
v_1 := {\mathcal L}_e^{- 1} \Delta v_e \in H_0^{s + 2}\,, \quad \| v_1 \|_{s + 2} \lesssim_s \e^{\mathtt a} \gamma^{- 1}\,, \quad \forall s_0 \leq s \leq S - \overline \mu
\end{equation}
By \eqref{sol eulero schema approx}, \eqref{cal F ve v1}, \eqref{eq v1} and estimates \eqref{stime euler quadratica}, \eqref{stima v1} (using also $\e^{\mathtt a} \gamma^{- 1} \leq \delta < 1$ sufficiently small) we conclude that
$$
\begin{aligned}
\| {\mathcal F}_\nu (v_e + \nu v_1) \|_s & = \nu^2\| - \Delta v_1 + \e {\mathcal Q}(v_1) \|_s   \leq  \nu^2 \big( \| \Delta v_1 \|_s + \| {\mathcal Q}(v_1) \|_s\big) \\
& \leq \nu^2 \big( \| v_1 \|_{s + 2} + \| v_1 \|_{s + 1}^2 \big)   \lesssim_s \nu^2 \e^{\mathtt a} \gamma^{- 1}\,.
\end{aligned}
$$
The claimed statement has then been proved. 
\end{proof}
\section{The fixed point argument}\label{sez fixed point}
In this section we want to find a solution $v$ of the functional equation ${\mathcal F}_\nu (v) = 0$ bifurcating from the approximate solution $v_{app} = v_e + \nu v_1$ constructed in the previous section . We search for a solution of the form $v = v_{app} + \psi$. By \eqref{altra forma cal F v}, \eqref{cal Q v1 v2} and using that $v_{app} = v_e + \nu v_1$, one computes 
\begin{equation}\label{cal F va z}
\begin{aligned}
{\mathcal F}_\nu(v_{app} + \psi) 
& = {\mathcal F}_\nu(v_{app}) + \big(  {\mathtt L}_0  + \e \di {\mathcal Q}(v_e) - \nu \Delta \big)[\psi] + \e \nu \di {\mathcal Q}(v_1)[\psi]  + \e{\mathcal Q}(\psi) \\
& = {\mathcal F}_\nu(v_{app})  + {\mathcal L}_\nu \psi +  \e\, \nu\, d {\mathcal Q}(v_1)[\psi]  + \e\,{\mathcal Q}(\psi)
\end{aligned}
\end{equation}
since the linear operator $ {\mathtt L}_0 + \e\di  {\mathcal Q}(v_e)- \nu \Delta $ is exactly the linearized Navier Stokes operator $\mL_{\nu}$ in \eqref{operatore linearizzato} obtained by linearizing the nonlinear functional ${\mathcal F}_\nu$ at the Euler solution $v_e$.
Therefore, the equation ${\mathcal F}_\nu(v_{app} + \psi) = 0$ reduces to
$$
{\mathcal L}_\nu \psi = - \big({\mathcal F}_{\nu}(v_{app}) +  \e\, \nu\, \di {\mathcal Q}(v_1)[\psi]  + \e\, {\mathcal Q}(\psi)\big)\,. 
$$
By the invertibility of the linear operator ${\mathcal L}_\nu$, proved in Proposition \ref{inversione linearized} for any $\nu>0$, we conclude that finding solutions of ${\mathcal F}_{\nu}(v_{app} + \psi) = 0$ is equivalent to solve a fixed point problem, that is
\begin{equation}\label{formulazione punto fisso}
\begin{aligned}
& \psi = {\mathcal S}_\nu(\psi) \quad \text{where} \\
& {\mathcal S}_\nu(\psi) := - {\mathcal L}_\nu^{- 1}\big({\mathcal F}_{\nu}(v_{app}) +  \e \,\nu \, \di {\mathcal Q}(v_1)[\psi]  + \e{\mathcal Q}(\psi)\big)\,.
\end{aligned}
\end{equation}
For any $s \geq 0$, $\eta > 0$, we define the closed ball 
$$
{\mathcal B}_s(\eta) := \big\{ z \in H^s_0 : \| z \|_s \leq \eta \big\}\,. 
$$
We want to show that for $S > {\rm max}\{ \overline S\,,\,s_0 + \overline \mu \}$ (as in Proposition \ref{prop approximate soluions}) (where the constant $\overline \mu \gg 0$ is given in Proposition \ref{prop approximate soluions}) if \eqref{sol eulero schema approx} holds and $\e^{\mathtt a} \gamma^{- 1} \ll 1$, for any $s_0 \leq s \leq S - \overline \mu$  and for any value of the viscosity parameter $\nu > 0$, the map ${\mathcal S}_\nu : {\mathcal B}_s(\nu) \to {\mathcal B}_s(\nu)$ is a contraction. We actually need this preliminary lemma which allows to estimate in a sharp way the terms $\di{\mathcal Q}(v_1)[\psi]$ and ${\mathcal Q}(\psi)$ appearing in \eqref{formulazione punto fisso}. 
\begin{lemma}\label{stima astratta quadratica eulero}
For $s\geq s_0$, we define
$$
\mN:H^{s + 1}_0\times H^{s + 1}_0 \to H^{s}_0 \,, \quad (v_1,v_2) \mapsto {\mathcal N}(v_1, v_2) := \big( \nabla_\bot (- \Delta)^{- 1} v_1 \big) \cdot \nabla v_2
$$
where $\nabla_\bot$ is as in \eqref{equazione vorticita media nulla}. 
Then for any $n \in \N$, $s \geq s_0$,
$$
\| (- \Delta)^{- \frac{n}{2}}{\mathcal N}(v_1, v_2) \|_s \lesssim_s \| v_1 \|_s \| v_2 \|_s, \quad \forall v_1, v_2 \in H^s.
$$
\end{lemma}
\begin{proof}
First of all, note that if $v_1, v_2$ have zero space average, using that ${\rm div}\big(\nabla_\bot \langle D \rangle^{- 2} v_1 \big) = 0$ and by integrating by parts, one can easily see that ${\mathcal N}(v_1, v_2)$ has zero space-average. Let $v_1, v_2 \in H^s_0$ and expand the bilinear form ${\mathcal N}(v_1, v_2)$ in Fourier coefficients. One obtains 
\begin{equation}\label{D -n cal N}
(- \Delta)^{- \frac{n}{2}} {\mathcal N}(v_1, v_2) = \sum_{\begin{subarray}{c}
\xi_1, \xi_2 \in \Z^2 \setminus \{ 0 \} \\
\ell_1, \ell_2 \in \Z^d \\
\xi_1 + \xi_2 \neq 0
\end{subarray}} N(\xi_1, \xi_2) \widehat v_1(\ell_1, \xi_1) \widehat v_2(\ell_2, \xi_2) e^{\ii (\ell_1 + \ell_2) \cdot \vphi} e^{\ii (\xi_1 + \xi_2) \cdot x}
\end{equation}
where 
$$
N(\xi_1, \xi_2) := \ii \dfrac{\xi_1^\bot \cdot \xi_2}{| \xi_1 + \xi_2 |^n | \xi_1 |^2}, \quad \xi_1, \xi_2 \in \Z^2 \setminus \{ 0 \}, \quad \xi_1 + \xi_2 \neq 0
$$
where  $y^\bot := (- y_2, y_1)$ for  $y = (y_1, y_2)$. Using that $|\xi_1^\perp|\leq |\xi_1 |$ and  $| \xi_2 | \lesssim | \xi_1 | + | \xi_1 + \xi_2 | \lesssim | \xi_1 | |\xi_1 + \xi_2| $   for any $\xi_1, \xi_2 \in \Z^2 \setminus \{ 0 \}$ with $\xi_1 + \xi_2 \neq 0$, one has
\begin{equation}\label{stima N xi 1 xi 2}
\begin{aligned}
|N(\xi_1, \xi_2)| & \leq \dfrac{|\xi_1^\bot| | \xi_2|}{\langle \xi_1 + \xi_2 \rangle^n \langle \xi_1 \rangle^2} 
\leq \dfrac{ | \xi_2|}{| \xi_1 + \xi_2 |^n | \xi_1 |} \lesssim 1\,,
\end{aligned}
\end{equation}
uniformly in $\xi_1,\xi_2 \in \Z^{2} \setminus \{ 0 \}$, $\xi_1 + \xi_2 \neq 0$.
Therefore, by \eqref{D -n cal N}, \eqref{stima N xi 1 xi 2} and using that $\langle \ell, \xi \rangle^s \lesssim_s \langle \ell', \xi' \rangle^s + \langle \ell - \ell', \xi - \xi' \rangle^s$ for any $\ell, \ell' \in \Z^d$, $\xi, \xi' \in \Z^2$, one has 
\begin{equation*}\label{faraonico}
	\begin{footnotesize}
		\begin{aligned}
			\| (- \Delta)^{- \frac{n}{2}} {\mathcal N}(v_1, v_2) \|_s^2 & = \sum_{(\ell, \xi) \in \Z^{d} \times (\Z^2 \setminus \{ 0 \})} \langle \ell, \xi \rangle^{2s} \Big| \sum_{(\ell', j') \in \Z^{d} \times (\Z^2 \setminus \{ 0 \})} N(\xi - \xi', \xi') \widehat v_1(\ell - \ell', \xi - \xi') \widehat v_2(\ell', \xi') \Big|^2  \\
			& \lesssim \sum_{(\ell, \xi) \in \Z^{d} \times (\Z^2 \setminus \{ 0 \})} \langle \ell, \xi \rangle^{2s} \Big( \sum_{(\ell', \xi') \in \Z^{d} \times (\Z^2 \setminus \{ 0 \})}  |\widehat v_1(\ell - \ell', \xi - \xi') ||\widehat v_2(\ell', \xi')| \Big)^2  \\
			& \lesssim_s A_1 + A_2 \,,
		\end{aligned}
	\end{footnotesize}
\end{equation*}
where
\begin{equation*}
\begin{aligned}
A_1 := & \sum_{(\ell, \xi) \in \Z^{d} \times (\Z^2 \setminus \{ 0 \})}  \Big( \sum_{(\ell', \xi') \in \Z^{d} \times (\Z^2 \setminus \{ 0 \})} \langle \ell', \xi' \rangle^{s} |\widehat v_1(\ell - \ell', \xi - \xi') ||\widehat v_2(\ell', \xi')| \Big)^2 \,, \\
A_2 := & \sum_{(\ell, \xi) \in \Z^{d} \times (\Z^2 \setminus \{ 0 \})}  \Big( \sum_{(\ell', \xi') \in \Z^{d} \times (\Z^2 \setminus \{ 0 \})} \langle \ell - \ell', \xi - \xi' \rangle^{s} |\widehat v_1(\ell - \ell', \xi - \xi') ||\widehat v_2(\ell', \xi')| \Big)^2\,. 
\end{aligned}
\end{equation*}
We estimate  $A_1$. 
The estimate of $A_2$ can be done similarly.
 By the Cauchy-Schwartz inequality and using that $\sum_{\ell', \xi'} \frac{1}{\langle \ell - \ell', \xi - \xi' \rangle^{2 s_0}} \leq C_0$ (since $s_0 > \frac{d + 2}{2}$, see \eqref{definizione s0}), one has
$$
\begin{aligned}
A_1 & \lesssim\sum_{\begin{subarray}{c}
(\ell, \xi) \in \Z^{d} \times (\Z^2 \setminus \{ 0 \}) \\
(\ell', \xi') \in \Z^{d} \times (\Z^2 \setminus \{ 0 \})
\end{subarray}}\langle \ell - \ell', \xi - \xi' \rangle^{2 s_0}  |\widehat v_1(\ell - \ell', \xi - \xi') |^2 \langle \ell', \xi' \rangle^{2 s} |\widehat v_2(\ell', \xi')|^2 \\
& \lesssim \sum_{(\ell', \xi') \in \Z^{d} \times (\Z^2 \setminus \{ 0 \})}  \langle \ell', \xi' \rangle^{2 s} |\widehat v_2(\ell', \xi')|^2   \sum_{(\ell, \xi) \in \Z^{d} \times (\Z^2 \setminus \{ 0 \})} \langle \ell - \ell', \xi - \xi' \rangle^{2 s_0}  |\widehat v_1(\ell - \ell', \xi - \xi') |^2 \\
& \lesssim  \| v_2 \|_s^2 \| v_1 \|_{s_0}^2\,.
\end{aligned}
$$
Similarly, one shows that $A_2 \lesssim \| v_2 \|_{s_0}^2 \| v_1 \|_s^2$ and the claimed estimate follows. 
\end{proof}
We are now ready to perform a fixed point argument on the map ${\mathcal S}_\nu$ defined in \eqref{formulazione punto fisso}. 
\begin{proposition}[\bf Contraction]\label{proposizione contrazione}
For any $S > {\rm max}\{ \overline S\,,\, s_0 + \overline \mu \}$ (where $\overline \mu \gg 0$ is the constant given in Proposition \ref{prop approximate soluions}) there exists $\delta := \delta(S, \tau, d) \in (0, 1)$ such that, if \eqref{sol eulero schema approx} holds and $\e^{\mathtt a} \gamma^{- 1} \leq \delta$, for any $s_0 \leq s \leq S - \overline \mu$, for any value of the viscosity $\nu > 0$ and for any value of the parameter $\lambda = (\omega, \zeta) \in \Lambda_\infty^\gamma \cap \Gamma_\infty^\gamma$ (see \eqref{cantor finale ridu}, \eqref{prime di melnikov}),
the map ${\mathcal S}_\nu : {\mathcal B}_s(\nu) \to {\mathcal B}_s(\nu)$, defined in \eqref{formulazione punto fisso}, is a contraction. 
\end{proposition}
\begin{proof}
We write the map ${\mathcal S}_\nu$ as 
\begin{equation}\label{seconda espressione cal S z}
\begin{aligned}
{\mathcal S}_\nu(\psi) &:= - {\mathcal L}_\nu^{- 1}{\mathcal F}_\nu (v_{app}) -  \e\, \nu\, \big({\mathcal L}_\nu^{- 1} (- \Delta) \big) \big((- \Delta)^{- 1} \di {\mathcal Q}(v_1)[\psi] \big) \\
& \qquad   - \e\, \big({\mathcal L}_\nu^{- 1} (- \Delta) \big) \big((- \Delta)^{- 1}{\mathcal Q}(\psi) \big)\,. 
\end{aligned}
\end{equation}
Let $s_0 \leq s \leq S - \overline \mu$, $\nu > 0$, $\psi \in {\mathcal B}_s(\nu)$. By Proposition \ref{inversione linearized}, one has that 
\begin{equation}\label{prima stima cal S z}
\begin{aligned}
\| {\mathcal S}_\nu(\psi) \|_s & \lesssim_s \nu^{- 1} \big( \| {\mathcal F}_\nu (v_{app})\|_s + \e \nu \| (- \Delta)^{- 1}\di {\mathcal Q}(v_1)[\psi] \|_s  \\
& \qquad + \e \| (- \Delta)^{- 1}{\mathcal Q}(\psi) \|_s   \big)\,.
\end{aligned}
\end{equation}
By \eqref{altra forma cal F v}-\eqref{cal Q v1 v2}, Proposition \ref{prop approximate soluions}, Lemma \ref{stima astratta quadratica eulero}, and using that $\| \psi \|_s \leq \nu$, one gets
$$
\begin{aligned}
& \| {\mathcal F}_\nu(v_{app})\|_s \lesssim_s \nu^2 \e^{\mathtt a} \gamma^{- 1}, \quad \| (- \Delta)^{- 1} \di {\mathcal Q}(v_1)[\psi] \|_s \lesssim_s \nu \e^{\mathtt a} \gamma^{- 1},  \\
&    \| (- \Delta)^{- 1}{\mathcal Q}(\psi) \|_s  \lesssim_s \nu^2 \,.
\end{aligned}
$$
The latter bounds, together with \eqref{prima stima cal S z} and having ${\mathtt a}\in (0,1)$, imply that 
\begin{equation*}\label{seconda stima cal S z}
\begin{aligned}
\| {\mathcal S}_\nu(\psi) \|_s & \lesssim_s \nu^{- 1} \big(\nu^2 \e^{\mathtt a} \gamma^{- 1} + \e^{\mathtt a + 1} \nu^2 \gamma^{- 1} + \e \nu^2   \big) \leq C(s) \nu \e^{\mathtt a} \gamma^{- 1}
\end{aligned}
\end{equation*}
for some constant $C(s) > 0$. Therefore, by taking $C(s) \e^{\mathtt a} \gamma^{- 1} \leq 1$, we get that $\mS_\nu$ maps ${\mathcal B}_s(\nu)$ into itself. By \eqref{seconda espressione cal S z}, for any $\psi \in {\mathcal B}_s(\nu)$, one computes the Fr\'echet differential of ${\mathcal S}_\nu$
\begin{equation*}\label{terza espressione cal F z}
	\begin{aligned}
		\di \,{\mathcal S}_\nu (\psi) &=  -  \e \nu \big({\mathcal L}_\nu^{- 1} (- \Delta) \big) \big(  (- \Delta)^{- 1} \di {\mathcal Q}(v_1) \big) \\
		& \qquad  - \e   \big({\mathcal L}_\nu^{- 1} (- \Delta) \big) \big((- \Delta)^{- 1} \di \mQ(\psi) \big) \,.
	\end{aligned}
\end{equation*}
Arguing as above, we obtain
$$
\| \di \,{\mathcal S}_\nu (\psi)\|_{{\mathcal B}(H^s_0)} \leq C(s) \big( \e^{\mathtt a + 1}  \gamma^{- 1}  + \e  \big) \leq C(s) \e 
$$
for some larger constant $C(s) > 0$, asking again $ \e^{\mathtt a} \gamma^{- 1} \leq \tfrac{1}{C(s)}$. By taking $C(s) \e \leq 1/2$, we conclude that $\| \di \,{\mathcal S}_\nu( \psi ) \|_{{\mathcal B}(H^s_0)} \leq 1/2$, implying that ${\mathcal S}_\nu$ is a contraction. 
\end{proof}

%


\section{Proof of Theorem \ref{teorema limite singolare}}\label{sezione teoremi principali}
In this final section we summarize up the whole construction and we conclude the proof of Theorem \ref{teorema limite singolare}. In the first place, we show that the non-resonance conditions \eqref{cantor finale ridu} and \eqref{prime di melnikov} hold for most values of the parameters $(\omega,\zeta)\in \R^d\times \R^2$.
\subsection{Measure estimate}\label{sezione stime di misura}
In this section we prove that the set
\begin{equation}\label{def cal G infty}
	{\mathcal G}_\infty^\gamma := \Lambda_\infty^\gamma \cap \Gamma_\infty^\gamma
\end{equation}
has large 
Lebesgue measure (recall \eqref{cantor finale ridu}, \eqref{prime di melnikov}). We actually show the following 
\begin{proposition} \label{prop measure estimate finale}
	Let 
	\begin{equation}\label{definizione finale tau}
		\tau := {\rm max}\{ d, 2 \} + 1\,. 
	\end{equation}
	Then $|\Omega_\e \setminus {\mathcal G}_\infty^\gamma| \lesssim \gamma$ (recall that $\Omega_\e$ is the set of parameters appearing in Theorem \ref{main theorem 2}).  
\end{proposition} 
The rest of this section is devoted to the proof of the latter proposition. We actually estimate the measure of the set $\Omega_\e \setminus \Lambda_\infty^\gamma$ . The estimate for $\Omega_\e \setminus \Gamma_\infty^\gamma$ follows similarly. We write
\begin{equation}\label{pirati caraibi 1}
	\Omega_\e \setminus \Lambda_\infty^\gamma = (\Omega_\e \setminus DC(\gamma, \tau)) \cup (DC(\gamma, \tau) \setminus \Lambda_\infty^\gamma)\,. 
\end{equation}
By a standard Diophantine estimate one shows that 
\begin{equation}\label{misura diofantei standard}
	|\Omega_\e \setminus DC(\gamma, \tau)| \lesssim \gamma\,. 
\end{equation}
Define
\begin{equation*}\label{def indici cal I}
	{\mathcal I} := \Big\{ (\ell, j, j') \in \Z^d \times (\Z^2 \setminus \{ 0\}) \times (\Z^2 \setminus \{ 0\})  : (\ell, j, j') \neq (0, j, j) \Big\}\,. 
\end{equation*}
By recalling \eqref{cantor finale ridu}, one has that 
\begin{equation}\label{def con risonanti}
	\begin{footnotesize}
		\begin{aligned}
			& \Lambda_\infty^\gamma \setminus DC(\gamma, \tau) \subseteq \bigcup_{(\ell, j, j') \in {\mathcal I}} {\mathcal R}_\gamma(\ell, j, j'),  \\
			&{\mathcal R}_\gamma(\ell, j, j') := \Big\{ \lambda = (\omega, \zeta) \in DC(\gamma, \tau) : |\ii \,\omega \cdot \ell + \mu_\infty(j) - \mu_\infty(j')| < \frac{2 \gamma}{\langle \ell \rangle^\tau |j|^\tau |j'|^\tau} \Big\}\,. 
		\end{aligned}
	\end{footnotesize}
\end{equation}
\begin{lemma}\label{stima risonanti}
	We have $|{\mathcal R}_\gamma(\ell, j, j')| \lesssim \gamma\langle \ell \rangle^{- \tau} |j|^{- \tau} |j'|^{- \tau}$ for any $(\ell, j, j') \in {\mathcal I}$. 
\end{lemma}
\begin{proof}
	By Lemma \ref{lemma blocchi finali}, one gets that 
	$$
	\begin{aligned}
		& \ii \,\omega \cdot \ell + \mu_\infty(j) - \mu_\infty(j') = \ii \lambda \cdot k + r_{j j'}(\lambda) =: \vphi(\lambda)\,, \\
		& \lambda = (\omega, \zeta), \quad k = (\ell, j - j'), \quad |r_{j j'}|^{{\rm Lip}(\gamma)} \lesssim \e\,. 
	\end{aligned}
	$$
	Since $(\ell, j, j') \in {\mathcal I}$, one has  $k \neq 0$. We write $\lambda = \frac{k}{|k|}s + \eta$, with $ \eta \cdot k = 0  $, and
	$$
	\begin{aligned}
		& \vphi(s) := \ii \,\lambda \cdot k + r_{j j'}(\lambda) = \ii s |k| + \tilde r_{j j'}(s)\,, \quad  \tilde r_{j j'}(s) := r_{j j'}\big(\tfrac{k}{|k|}s + \eta \big)\,. 
	\end{aligned}
	$$
	One then gets 
	$$
	|\vphi(s_1) - \vphi(s_2)| \geq (|k| - |r_{j j'}|^{\rm lip}) |s_1 - s_2| \geq (|k| - C \e \gamma^{- 1}) |s_1 - s_2| \geq \tfrac{1}{2} |s_1 - s_2|
	$$
	assuming $\e \gamma^{- 1} \leq \tfrac{1}{2C}$. This implies that 
	$$
	\Big| \Big\{ s : |\vphi(s)| <  \frac{2 \gamma}{\langle \ell \rangle^\tau |j|^\tau |j'|^\tau}\Big\} \Big| \lesssim \frac{ \gamma}{\langle \ell \rangle^\tau |j|^\tau |j'|^\tau}
	$$
	The claimed measure estimate then follows by a Fubini argument. 
\end{proof}


\begin{proof}
	[\textsc{Proof of Proposition \ref{prop measure estimate finale}}] By \eqref{def con risonanti} and Lemma \ref{stima risonanti}, one gets that $|DC(\gamma, \tau) \setminus \Lambda_\infty^\gamma|$ can be bounded by $C \gamma \sum_{\ell, j j'} \langle \ell \rangle^{- \tau} \langle j \rangle^{- \tau} \langle j' \rangle^{- \tau}$. This series converges by taking $\tau$ as in \eqref{definizione finale tau}. By recalling also \eqref{pirati caraibi 1}, \eqref{misura diofantei standard} one estimates $\Omega \setminus \Lambda_\infty^\gamma$. The estimate for $\Omega \setminus \Gamma_\infty^\gamma$ is similar. 
\end{proof}

\subsection{Proof of Theorem \ref{teorema limite singolare}}

We finally prove the main result of  Theorem \ref{teorema limite singolare} together with Corollary \ref{corollario bla bla}. At this final stage, we  link the constant $\gamma$, coming from the non-resonance conditions in the set \eqref{def cal G infty}, with the small parameter $\varepsilon>0$ by choosing the former as
\begin{equation}\label{scelta gamma}
	\gamma:= \e^{\frac{\mathtt a}{2}}
\end{equation}
\\[0.5mm]
\noindent
\begin{proof}
[\sc Proof of Theorem \ref{teorema limite singolare}.]
Let $s\geq s_0$ be fixed and $\gamma$ as in \eqref{scelta gamma}. Then, the smallness conditions $\e^{{\mathtt a}}\gamma^{-1} = \e^{\frac{\mathtt a}{2}}\leq\delta< 1$ is satisfied for $\e \leq \e_0< 1$ small enough. The time quasi-periodic solution of \eqref{equazione vorticita media nulla} is searched as zero of the nonlinear functional $\mF_{\nu}(v)$ defined in \eqref{equazione cal F vorticita}. We look for the solution $v_{\nu}=v_{\nu}(\varphi,x;\omega,\zeta)$ of the form $v_{\nu}=v_{e}+ \nu v_1 +\psi_\nu$. It depends on the parameters $\lambda = (\omega, \zeta) \in {\mathcal O}_\e:= {\mathcal G}_\infty^\gamma$, where ${\mathcal G}_\infty^{\gamma}$ is defined in \eqref{def cal G infty}. The function $v_e(\cdot; \lambda) \in H_0^{S}(\T^{d+2})$, $\lambda \in \Omega_\e$ is a  solution of the forced Euler equation provided by Theorem \ref{main theorem 2} with $S = {\rm max}\{s + \overline \mu, \overline S \}$, which fixes the regularity $q=q(S)$  of the forcing term $F=\nabla\times f$. It satisfies $\sup_{(\omega, \zeta) \in {\mathcal G}_\infty^\gamma }\| v_e(\,\cdot\, ;\omega,\zeta) \|_{s+\bar\mu} \lesssim_{s} \e^{\mathtt a} $, for some $\mathtt a \in (0,1)$ independent of $s$. The function $v_1\in H_0^{s}(\T^{d+2})$ is provided by Proposition \ref{prop approximate soluions}: It is defined as in \eqref{stima v1} and satisfies $\sup_{(\omega, \zeta) \in {\mathcal G}_\infty^\gamma }\|v_1 (\,\cdot\, ;\omega,\zeta)\|_{s} \leq  C(s) \e^{\mathtt a}\gamma^{-1} \leq 1$. The definition of the function $v_1$ implies that $v_e +\nu v_1$ solves the equation up to an error of order $O(\nu^2)$, namely $\| \mF_{\nu}(v_e+\nu v_1)\|_{s} \leq C(s)\nu^2 \e^{\mathtt a}\gamma^{-1}\leq \nu^2$, see \eqref{stime v1 cal F va}. The function $\psi_\nu$ is then defined as the fixed point for the map $\mS_\nu(\psi)$ in \eqref{cal F va z}-\eqref{formulazione punto fisso}. Proposition \ref{proposizione contrazione} shows that $\mS_\nu$ is a contraction map on the ball $\mB_s(\nu)$, which implies the of a unique $\psi_\nu$ in the ball $\| \psi_\nu\|_{s} \leq \nu$ such that $\psi_\nu=\mS_\nu(\psi_\nu)$. We conclude that $\mF_{\nu}(v_e + \nu v_1 + \psi_\nu)=0$ and that $\sup_{(\omega, \zeta) \in {\mathcal G}_\infty^\gamma } \| v_\nu(\,\cdot\, ; \omega,\zeta) - v_e (\,\cdot\, ; \omega,\zeta) \|_s \leq 2\nu $ with $v_\nu := v_e + \nu v_1 + \psi_\nu$, as required. Finally, using that $| \Omega\setminus {\mathcal O}_\e | \leq |\Omega\setminus \Omega_\e | + |\Omega\setminus {\mathcal G}_\infty^{\gamma} |$, we conclude that $\lim_{\e\to 0}| \Omega \setminus {\mathcal O}_\e| =0$ by Theorem \ref{main theorem 2}, Proposition \ref{prop measure estimate finale} and using that we have chosen $\gamma = \e^{\frac{\mathtt a}{2}}$ as in \eqref{scelta gamma}. 
\end{proof}
\noindent
\begin{proof}
[\sc Proof of Corollary \ref{corollario bla bla}.] Let $v_e \in H^{s+\bar\mu}_0(\T^{d + 2})$ and $v_\nu \in H^s_0(\T^{d + 2})$ as in Theorem \ref{teorema limite singolare} satisfying $\| v_\nu(\cdot; \lambda) - v_e (\cdot; \lambda) \|_s \lesssim_s \nu$ for any $\lambda = (\omega, \zeta) \in {\mathcal O}_\e$. Let
$$
v_{\nu}^\omega(t, x) := v_\nu (\omega t, x)\,, \quad  v_e^\omega(t, x) := v_e(\omega t, x)\,, \quad (t, x) \in \R \times \T^2\,. 
$$
$v_\nu^\omega$ is a global solution of the forced Navier Stokes equation and $v_e^\omega$ is a global solution of the forced Euler equation with external force $F(\omega t , x)$, $F := \nabla \times f$. By the latter definition, one clearly has 
$$
\| v_{\nu}^\omega -  v_e^\omega\|_{W^{\sigma, \infty}(\R \times \T^2)} \lesssim \| v_\nu - v_e \|_{W^{\sigma, \infty}(\T^d \times \T^2)}
$$
and using that $H^{s}(\T^d \times \T^2)$ is compactly embedded in $W^{\sigma, \infty}(\T^d \times \T^2)$ with $0 \leq \sigma \leq s - \big( \lfloor\frac{d + 2}{2} \rfloor+ 1\big)$, one obtains the chain of inequalities
$$
\begin{aligned}
\| v_{\nu}^\omega -  v_e^\omega\|_{W^{\sigma, \infty}(\R \times \T^2)}&  \lesssim \| v_\nu - v_e \|_{W^{\sigma, \infty}(\T^d \times \T^2)}  \lesssim \| v_\nu - v_e \|_s \lesssim_s \nu \,.
\end{aligned}
$$
The proof of the Corollary is then concluded. 
\end{proof}

\bigskip

\begin{flushright}
\textbf{Luca Franzoi}

\smallskip

NYUAD Research Institute

New York University Abu Dhabi

NYUAD Saadiyat Campus

129188 Abu Dhabi, UAE

\smallskip 

\texttt{lf2304@nyu.edu}

\bigskip

\textbf{Riccardo Montalto}

\smallskip

Dipartimento di Matematica ``Federigo Enriques''

Universit\`a degli Studi di Milano

Via Cesare Saldini 50

20133 Milano, Italy

\smallskip

\texttt{riccardo.montalto@unimi.it}
\end{flushright}



\end{document}